\input ./style/arxiv-vmsta.cfg
\documentclass[numbers,compress,v1.0.1]{vmsta}
\usepackage{vtexbibtags}
\usepackage{vtexbptsetup,htpicture}
\usepackage{authorquery}

\volume{5}
\issue{3}
\pubyear{2018}
\firstpage{247}
\lastpage{295}
\aid{VMSTA104}
\doi{10.15559/18-VMSTA104}
\articletype{research-article}

\ifdefined\HCode 
\newcommand{\probto}{\stackrel{{\rm P}}{\longrightarrow}}
\def\mycenter#1{\par#1\par}
\else
\newcommand{\probto}{\mathrel{\stackrel{{\rm P}}{\longrightarrow}}}
\def\mycenter#1{\par\centerline{#1}\noindent}
\fi

\startlocaldefs
\newcommand{\lleft}{\left}
\newcommand{\rright}{\right}
\newcommand{\rrVert}{\Vert}
\newcommand{\llVert}{\Vert}

\allowdisplaybreaks

\def\Me{\operatornamewithlimits{\mathbb{E}}}
\def\Prob{\operatornamewithlimits{\mathbb{P}}}
\def\var{\operatornamewithlimits{\mathrm{var}}}
\def\rank{\operatornamewithlimits{\mathrm{rk}}}

\def\trace{\operatornamewithlimits{\mathrm{tr}}}
\def\defect{\operatornamewithlimits{\mathrm{def}}}
\def\dist{\operatornamewithlimits{\mathrm{dist}}}
\def\colspan{\operatornamewithlimits{\mathrm{span}}}
\def\diag{\operatornamewithlimits{\mathrm{diag}}}
\def\Kernel{\operatornamewithlimits{\mathrm{Ker}}}

\def\pinvp(#1){(#1)^{\dagger}}
\def\pinvb(#1){{#1}^{\dagger}}
\def\pinvbb(#1){#1^{\dagger}}
\newcommand{\uin}[1]{\|#1\|_{\rm U}}

\newcommand{\xtrue}{{X_{0}}}
\newcommand{\xxtrue}{X_{\rm ext}^0}
\newcommand{\xxtruetra}{X_{\rm ext}^{0\,\top}}
\newcommand{\hxx}{\widehat X_{\rm ext}}
\newcommand{\hxxsv}{\widehat X^{\rm mod}_{\rm ext}}
\newcommand{\hxxno}[1]{\widehat X_{#1}^{\rm ext}}
\newcommand{\hxxnoT}[1]{\widehat X_{#1}^{{\rm ext}\,\top}}
\newcommand{\smallxxmat}{{\left(\begin{smallmatrix} X_0 \\ -I \end{smallmatrix}\right)}}
\newcommand{\setA}{\mathit{F}}

\def\colspana<#1>{\colspan\langle #1\rangle}

\newcommand{\Deltapm}{\Delta_{\rm pm}}
\newcommand{\Deltaem}{\Delta_{\rm em}}
\newcommand{\DeltaFro}{\Delta_{\min\eqref{eqTLS118}}}
\newcommand{\DeltaUin}{\Delta_{\min\eqref{eqTLS118uin}}}
\newcommand{\DeltaFeas}{\Delta_{\rm feas}}
\newcommand{\XSXTMT}{(T^{-1})^\top \textrm{M} T^{-1}}
\newcommand{\pinvXSXTMT}{T \pinvb(\textrm{M}) T^{\top}}

\theoremstyle{plain}
\newtheorem{thm}{Theorem}[section]

\newtheorem*{cora}{Corollary}
\newtheorem{lem}[thm]{Lemma}
\newtheorem{prop}[thm]{Proposition}
\theoremstyle{definition} 

\newtheorem{exmp}{Example}[section]
\theoremstyle{remark}
\newtheorem{rem}{Remark}[section]
\newtheorem{remforp}{Remark}[thm]

\newtheorem*{notation}{Notation}

\endlocaldefs

\begin{aqf}
\querytext{Q1}{Is it really an implicit form?}
\querytext{Q2}{Please check if "spanning by ..." is more appropriate.}
\querytext{Q3}{Please check if "Conversely" is suitable here.}
\querytext{Q4}{Please check if "both" is suitable here.}
\querytext{Q5}{Is the meaning of "the condition number ..." clear enough?}
\querytext{Q6}{Is it "Loewner order"?}
\querytext{Q7}{Please check the meaning of this sentence.}
\querytext{Q8}{Please check the sentence "If we can...".}
\querytext{Q9}{Please check the meaning of the sentence.}
\querytext{Q10}{Are the quotation marks necessary in the mathematical expressions, here and below?}
\querytext{Q11}{Is the sentence clear enough?}
\querytext{Q12}{Please check what the inequality is meant here.}
\end{aqf}
\begin{document}

\begin{frontmatter}
\pretitle{Research Article}

\title{Consistency of the total least squares estimator in the linear errors-in-variables regression}


\author{\inits{S.V.}\fnms{Sergiy}~\snm{Shklyar}\ead[label=e1]{shklyar@univ.kiev.ua}}
\address{\institution{Taras Shevchenko National University of Kyiv}}



\markboth{S.V. Shklyar}{Consistency of the total least squares estimator in linear errors-in-variables regression}

\begin{abstract}
This paper deals with a
homoskedastic errors-in-variables linear regression model
and properties of the total least squares (TLS) estimator.
We partly revise the consistency results for
the TLS estimator previously obtained by the author~\cite{Shklyar2011}.
We present complete and comprehensive proofs of consistency theorems.
A theoretical foundation for construction of the TLS estimator and its relation to
the generalized eigenvalue problem is explained.
Particularly, the uniqueness of the estimate is proved.
The Frobenius norm in the definition of the estimator
can be substituted by the spectral norm,
or by any other unitarily invariant norm;
then the consistency results are still valid.
\end{abstract}

\received{\sday{31} \smonth{10} \syear{2017}}
\revised{\sday{7} \smonth{5} \syear{2018}}
\accepted{\sday{8} \smonth{5} \syear{2018}}
\publishedonline{\sday{30} \smonth{5} \syear{2018}}
\end{frontmatter}

\section{Introduction}\label{sec:intro}
We consider a functional linear error-in-variables model.
Let
$\{a^0_i,\; i\ge 1\}$
be a sequence of unobserved nonrandom $n$-dimensional vectors.
The elements of the vectors are true explanatory variables or
(in other terminology) true regressors.
We observe $m$ $n$-dimensional random vectors
$a_1, \ldots, a_m$ and
$m$ $d$-dimensional random vectors $b_1, \ldots, b_m$.
They are thought to be true vectors $a^0_i$
and $X_0^\top a^0_i$, respectively, plus additive errors:
\begin{equation}
\label{eq:linregww} \begin{cases}
b_i = X_0^\top a^0_i + \tilde b_i, \\
a_i = a^0_i + \tilde a_i,
\end{cases}
\end{equation}
where
$\tilde a_i$ and $\tilde b_i$ are random measurement errors
in the regressor and in the response.
A nonrandom matrix $X_0$ is estimated based on observations
$a_i$, $b_i$, $i=1,\ldots,m$.

This problem is related to finding an approximate solution to incompatible linear equations (``overdetermined'' linear equation, because the number of equations exceeds the number of variables)
\[
A X \approx B,
\]
where $A = [a_1, \ldots, a_m]^\top$ is an $m \times n$ matrix and
$B = [b_1, \ldots, b_m]^\top$ is an $m \times d$ matrix.
Here $X$ is an unknown $n \times d$ matrix.

In the linear error-in-variables regression model \eqref{eq:linregww},
the \emph{Total Least Squares\/} (TLS) estimator in widely used.
It is a multivariate
equivalent to
the orthogonal regression estimator.
We are looking for conditions that provide consistency or
strong consistency of the estimator.
It is assumed (for granted) that the measurement errors
$\tilde c_i =  ( \begin{smallmatrix} \tilde a_i\\ \tilde b_i\end{smallmatrix} )$,
$i=1,2,\ldots$, are independent
and have the same covariance matrix $\varSigma$.
It may be singular.
In particular, some of regressors may be observed without errors.
(If the matrix $\varSigma$ is nonsingular, the proofs can be simplified.)
An intercept can be introduced into \eqref{eq:linregww}
by augmenting the model and inserting a constant error-free regressor.

Sufficient conditions for consistency of the estimator are presented in
Gleser \cite{Gleser1981},
Gallo \cite{Gallo82},
Kukush and Van Huffel \cite{KVH02}.
In \cite{Shklyar2011}, the consistency results are obtained under
less restrictive conditions than in \cite{KVH02}.
In particular, there is no requirement that
\[
\frac{\lambda^2_{\min} (A_0^\top A_0)}{\lambda_{\max} (A_0^\top A_0)} \to \infty \quad \mbox{as} \quad m\to\infty,
\]
where $A_0 = [a^0_1, \ldots, a^0_m]^\top$ is the matrix $A$
without measurement errors. Hereafter,  $\lambda_{\min}$
and $\lambda_{\max}$ denotes the minimum and maximum eigenvalues
of a matrix if all the eigenvalues are real numbers.  The matrix
$A_0^\top A_0^{}$ is symmetric (and positive semidefinite).  Hence,
its eigenvalues are real (and nonnegative).

The model where some variables are explanatory
and the other are response is called \textit{explicit}.
The alternative is the \textit{implicit} model,
where all the variables are treated equally.
In the \textit{implicit} model,
the $n$-dimensional linear subspace
in $\mathbb{R}^{n+d}$ is fitted to an observed set of points.
Some $n$-dimensional subspaces can be represented in a form
$\{(a, b) \in \mathbb{R}^{n+d} : b = X^\top a\}$
for some $n\times d$ matrix $X$;
such subspaces are called \textit{generic}.
The other subspaces are called \textit{non-generic}.
The true points lie on a generic subspace
$\{(a,b) : b = X_0^\top a\}$.
A consistently estimated subspace must be generic with high probability.
We state our results for the {explicit} model,
but use the ideas of the {implicit} model
in the definition of the estimator, as well as in proofs.

We allow errors in different variables to correlate.
Our problem is a minor generalization of
the mixed LS-TLS problem, which is studied
in \cite[Section~3.5]{VanHuffelVandewalle1991}.
In the latter problem, some explanatory variables are
observed without errors; the other explanatory variables
and all the response variables are observed with errors.
The errors have the same variance and are uncorrelated.
The basic LS model (where the explanatory variables are error-free,
and the response variables are error-ridden)
and the basic TLS model (where all the variables are observed
with error, and the errors are uncorrelated)
are marginal cases of the mixed LS-TLS problem.
By a linear transformation of variables our model
can be transformed into either a mixed LS-TLS or basic LS or basic TLS problem.
(We do not handle the case where there are more error-free variables than explanatory variables.)
Such a transformation does not always preserve the sets of
generic and non-generic subspaces.
The mixed LS-TLS problem can be transformed into the basic TLS problem
as it is shown in \cite{GOLUB1987317}.

The Weighted TLS and Structured TLS estimators are generalizations of the TLS estimator for the cases where the error covariance matrices
do not coincide
for different observations or where the errors for different observations are dependent;
more precisely, the independence condition is replaced with the condition on the ``structure of the errors''.
The consistency of these estimators is proved in
Kukush and Van Huffel \cite{KVH02} and
Kukush et al. \cite{KMF2005}.
Relaxing conditions for consistency of the Weighted TLS and Structured TLS estimators is an interesting topic for a future research.
For generalizations of the TLS problem, see the monograph
\cite{Markovsky2006} and the review \cite{WICS:WICS65}.

In the present paper, for a multivariate regression model
with multiple response variables
we consider two versions of the TLS estimator.
In these estimators, different norms
of the weighted residual matrix are minimized.
(These estimators coincide for the univariate regression model.)
The common way to construct the estimator is to minimize the Frobenius norm.
The estimator that minimizes the Frobenius norm also minimizes the spectral norm.
Any estimator that minimizes the spectral norm is consistent under conditions
of our consistency theorems (see Theorems~\ref{thm-3.1}--\ref{thm-3.3} in Section~\ref{sec:ourConsThs}).
We also provide a sufficient condition for uniqueness of the estimator
that minimizes the Frobenius norm.

In this paper, for the results on consistency of the TLS estimator which are stated in paper \cite{Shklyar2011},
we provide complete and comprehensive proofs and present all necessary auxiliary and complementary results.
For convenience of the reader we first present the sketch of proof.
Detailed proofs are postponed to the appendix.
Moreover, the paper contains new results on the relation between the TLS estimator and the generalized eigenvalue problem.

The structure of the paper is as follows.
In Section~\ref{sec:modelest} we introduce the model
and define the TLS estimator.
The consistency theorems
for different moment conditions on the errors
and for different senses of consistency
are stated in Section~\ref{sec:knownconsth},
and their proofs are sketched in Section~\ref{sec:sketch}.
Section~\ref{sect:eue} states the existence and uniqueness of the TLS estimator.
Auxiliary theoretical constructions and theorems are
presented in Section~\ref{sec:Theory}.
Section~\ref{sec:psdgep} explains the relationship between the TLS
estimator and the generalized eigenvalue problem.
The results in Section~\ref{sec:psdgep} are used
in construction of the TLS estimator and in the proof
of its uniqueness.
Detailed proofs are moved to the appendix (Section~\ref{sec:apx:Proofs}).

\subsection*{Notations}
At first, we list the \textit{general notation}.
For $v\,{=}\,(x_k)_{k=1}^n$ being a vector,
$\|v\| \,{=}\, \sqrt{\sum_{k=1}^n x_k^2}$ is the 2-norm of $v$.

For $M=(x_{i,j})_{i=1}^m {}_{j=1}^n$ being an $m\times n$ matrix,
$ \|M\| = \max_{v\neq 0} \frac{\|M v\|}{\|v\|}
= \sigma_{\max}(M)$ is
the spectral norm of $M$;
$\|M\|_F=\sqrt{\sum_{i=1}^m \sum_{j=1}^n x_{i,j}^2}$
is the Frobenius norm of $M$;\allowbreak{}
$\sigma_{\max}(M) = \sigma_1(M) \ge \sigma_2(M) \ge \cdots
\ge \sigma_{\min(m,n)}(M) \ge 0$
are the singular values of $M$, arranged in descending order;
$\colspana<M>$
is the column space of $M$;
$\rank M$ is the rank of $M$.
For a square $n\times n$ matrix $M$,\spacefactor=3000{}
$\defect M = n - \rank{M}$ is rank deficiency of $M$;
$\trace M = \sum_{i=1}^n x_{i,i}$ is the trace of $M$;
$\chi_M(\lambda) = \det(M - \lambda I)$
is the characteristic polynomial of $M$.
If $M$ is an
$n \times n$ matrix with real eigenvalues
(e.g., if $M$ is Hermitian or if $M$ admits a decomposition
$M = A B$, where $A$ and $B$ are Hermitian matrices, and
either $A$ or $B$ is positive semidefinite),
$\lambda_{\min} (M) = \lambda_1(M) \le \lambda_2(M)
    \le \cdots \le \lambda_n(M) = \lambda_{\max}(M)$
are eigenvalues of $M$ arranged in ascending order.

For $V_1$ and $V_2$ being linear subspaces of $\mathbb{R}^n$ of equal dimension $\dim V_1 = \dim V_2$,
$\|\sin\angle (V_1, V_2) \| = \| P_{V_1} - P_{V_2} \| =
    \| P_{V_1} (I - P_{V_2}) \|$
is  the greatest sine of the canonical angles between $V_1$ and $V_2$.
See Section~\ref{sss:sin} for more general definitions.

Now, list \textit{the model-specific notations}.
The notations (except for the matrix $\varSigma$) come from  \cite{KMF2005}.
The notations are listed here only for reference;
they are introduced elsewhere in this paper --
in Sections \ref{sec:intro} and \ref{sec:modelest}.

$n$ is the number of regressors, i.e., the number of explanatory variables for each observation;
$d$ is the number of response variables for each observation;
$m$ is the number of observations, i.e.,  the sample size.
\begin{description}
\item
[\def\arraystretch{0.5}{\inlinehtpicture{$C_0 = (A_0,\; B_0) =
\left(\begin{array}{@{}cc@{}}
\scriptstyle(a^0_1)^\top & \scriptstyle(a^0_1)^\top X_0 \\
\hdotsfor{2} \\
\scriptstyle(a^0_m)^\top & \scriptstyle(a^0_m)^\top X_0
\end{array}\right)
=
\left(\begin{array}{@{}c@{}}
\scriptstyle(c^0_1)^\top \\
\hdotsfor{1} \\
\scriptstyle(c^0_m)^\top
\end{array}\right)
$}}]
is the matrix of true variables.
It is an $m \times (n+d)$ nonrandom matrix.
The left-hand block $A_0$ of size $m \times n$
consists of true explanatory variables,
and the right-hand block $B_0$ of size $m \times d$
consists of true response variables.
\item
[\def\arraystretch{0.5}{\inlinehtpicture{$
\widetilde C = (\tilde A,\; \widetilde B)
=
\left(\begin{array}{@{}cc@{}}
\scriptstyle\tilde a_1^\top & \scriptstyle\tilde b_1^\top \\
\hdotsfor{2} \\
\scriptstyle\tilde a_m^\top & \scriptstyle\tilde b_m^\top
\end{array}\right)
=
\left(\begin{array}{@{}c@{}}
\scriptstyle\tilde c_1^\top \\
\ldots \\
\scriptstyle\tilde c_m^\top
\end{array}\right)
=
\left(\begin{array}{@{}ccc@{}}
\scriptstyle\delta_{1,1} & \scriptstyle\ldots & \scriptstyle\delta_{1,n+d} \\
\hdotsfor{3} \\
\scriptstyle\delta_{m,1} & \scriptstyle\ldots & \scriptstyle\delta_{m,n+d}
\end{array}\right)
$}}] is the matrix of errors.
It is an $m \times (n+d)$ random matrix.
\item
[\def\arraystretch{0.5}{\inlinehtpicture{$
C = (A,\; B) = C_0 + \widetilde C
=
\left(\begin{array}{@{}cc@{}}
\scriptstyle a_1^\top &\scriptstyle b_1^\top \\
\hdotsfor{2} \\
\scriptstyle a_m^\top &\scriptstyle b_m^\top
\end{array}\right)
$}}] is the matrix of observations.
It is an $m \times (n+d)$ random matrix.

\item[$\varSigma$] is a covariance matrix of errors for one observation.
For every $i$, it is assumed that $\Me\tilde c_i = 0$
and $\Me \tilde c_i \tilde c_i^\top  = \varSigma$.
The matrix $\varSigma$ is symmetric, positive semidefinite,
nonrandom, and of size $(n+d)\times(n+d)$.
It is assumed known when we construct the TLS estimator.

\item[$X_0$] is the matrix of true regression parameters.
It is a nonrandom $n \times d$ matrix
and is a parameter of interest.
\item[$\xxtrue = \binom{X_0}{-I}$] is
an augmented matrix of regression coefficients.
It is a nonrandom $(n+d) \times d$ matrix.
\item[$\widehat X$] is the TLS estimator of the matrix $X_0$.
\item[$\hxx$] is a matrix whose column space
$\colspana<\hxx>$ is considered an estimator of the subspace $\colspana<\xxtrue>$.
The matrix $\hxx$ is of size $(n+d)\times d$.
For fixed $m$ and $\varSigma$, $\hxx$ is a Borel measurable function
of the matrix $C$.
\end{description}
While in consistency theorems $m$ tends to $\infty$,
all matrices in this list except $\varSigma$, $X_0$ and $\xxtrue$
silently depend on $m$.
For example, in equations ``$\lim_{m\to\infty} \lambda_{\min} (A_0^\top A_0) = +\infty$''
and ``$\widehat X \to X_0$ almost surely''
the matrices $A_0$ and $\widehat X$ depend on $m$.{\sloppy\par}

\section{The model and the estimator}\label{sec:modelest}
\subsection{Statistical model}
It is assumed that the matrices $A_0$ and $B_0$ satisfy the relation
\begin{equation}
\label{eq:A0X0B0} \underset{m\times n} {A_0} \cdot \underset{n\times d}
{ \xtrue} = \underset{m\times d} {B_0}.
\end{equation}
They are observed with measurement errors $\tilde A$ and $\widetilde B$,
that is
\[
A=A_0+\tilde A, \qquad B=B_0+\widetilde B.
\]
The matrix $\xtrue$ is a parameter of interest.

Rewrite the relation in an \querymark{Q1}implicit form.
Let the $m \times (n+d)$ block matrices
$C_0, \widetilde C, C\in \mathbb{R}^{m\times(n+d)}$
be constructed by binding ``respective versions'' of matrices $A$ and $B$:
\[
C_0=[A_0\ B_0], \qquad \widetilde C=[\tilde
A\ \widetilde B], \qquad C=[A\ B].
\]
Denote
$\xxtrue= (\begin{smallmatrix}
\xtrue \\ -I_d
\end{smallmatrix} )
$.
Then
\begin{equation}
\label{eq:C0X0ext0} \underset{m\times(n+d)} {C_0} \cdot \underset{(n+d)
\times d}\xxtrue = \underset{m\times d}0.
\end{equation}

The entries of the matrix $\widetilde C$ are denoted $\delta_{ij}$;
the rows are $\tilde c_i$:
\[
\widetilde C=(\delta_{ij})_{i=1}^m
{}_{j=1}^{n+d}, \qquad \tilde c_i=(
\delta_{ij})_{j=1}^{n+d}.
\]

Throughout the paper the following three conditions are assumed to be true:\label{GlobalCond}
\begin{align}
&\mbox{The rows $\tilde c_i$ of the matrix $\widetilde C$ are
mutually independent random vectors.} \label{cond:G2}
\\
&\mbox{$\Me \widetilde C=0$, and $\Me \tilde c_i^{}
\tilde c_i^\top := (\Me \delta_{ij}
\delta_{ik})_{i=1,\,\,k=1}^{n+d\,\,n+d} = \varSigma$ for all $i{=}1{,
\ldots,m}$.} \label{cond:G1}
\\
&\mbox{$\rank (\varSigma \xxtrue)=d$.} \label{cond:G3}
\end{align}

\begin{exmp}[simple univariate linear regression with intercept]\label{example21}
For $i=1,\ldots,m$
\[
\begin{cases}
x_i=\xi_i+\delta_i; \\
y_i=\beta_0 + \beta_1 \xi_i + \varepsilon_i,
\end{cases}
\]
where the measurement errors $\delta_i$, $\varepsilon_i$,
$i=1,\ldots, m$, -- all the $2m$ variables --
are uncorrelated,
$\Me \delta_i = 0$, $\Me \delta_i^2 = \sigma_\delta^2$,
$\Me \varepsilon_i = 0$, and $\Me \varepsilon_i^2 = \sigma_\varepsilon^2$.
A sequence $\{(x_i,y_i),\ i=1,\ldots,m\}$ is observed.
The parameters $\beta_0$ and $\beta_1$ are to be estimated.

This example is taken from \cite[Section 1.1]{ChengVanNess}.
But the notation in Example~\ref{example21} and elsewhere in the paper is different.
Our notation is $a^0_i = (1, \xi_i)^\top$, $b^0_i = \eta_i$,
$a_i = (1, x_i)^\top$, $b_i = y_i$, $\delta_{i,1} = 0$,
$\delta_{i,2} = \delta_i$, $\delta_{i,3} = \varepsilon_i$,
$\varSigma = \diag(0, \sigma_\delta^2, \sigma_\varepsilon^2)$,
and $X_0 = (\beta_0, \beta_1)^\top$.
\end{exmp}

\begin{rem}\label{rem:condG3}
For some matrices $\varSigma$, \eqref{cond:G3} is satisfied for any $n\times d$ matrix $X_0$.
If the matrix $\varSigma$ in nonsingular, then condition \eqref{cond:G3} is satisfied.
If the errors in the explanatory variables and in the response
are uncorrelated, i.e., if the matrix $\varSigma$ has a
block-diagonal form
\[
\varSigma = \begin{pmatrix} \varSigma_{aa} & 0 \\ 0 & \varSigma_{bb} \end{pmatrix}
\]
(where $\varSigma_{aa} = \Me \tilde a_i \tilde a_i^\top$
and
$\varSigma_{bb} = \Me \tilde b_i \tilde b_i^\top$)
with nonsingular matrix $\varSigma_{bb}$,
then condition \eqref{cond:G3} is satisfied.
For example, in the basic mixed LS-TLS problem
$\varSigma$ is diagonal, $\varSigma_{bb}$ is nonsingular, and so
\eqref{cond:G3} holds true.
If the null-space of the matrix $\varSigma$
(which equals $\colspana<\varSigma>^\bot$ because $\varSigma$ is symmetric)
lies inside the subspace spanned by the first $n$ (of $n+d$) standard basis vectors,
then condition \eqref{cond:G3} is also satisfied.
On the other hand, if $\rank \varSigma < d$, then condition \eqref{cond:G3} is not satisfied.
\end{rem}

\subsection{Total least squares (TLS) estimator}
First, find the $m \times (n+d)$ matrix $\Delta$ for which the constrained minimum is attained
\begin{equation}
\begin{cases}
\| \Delta \, \pinvp(\varSigma^{1/2})  \|_F\to\min; \\
\Delta \, (I-P_\varSigma) =0; \\
\rank(C-\Delta) \le n.
\end{cases} \label{eqTLS118}
\end{equation}
\label{place:defhxx}\relax
Hereafter $\pinvb(\varSigma)$ is the Moore--Penrose pseudoinverse matrix of
the matrix $\varSigma$,
$P_\varSigma$ is an orthogonal projector onto the column space of $\varSigma$,
$P_\varSigma=\varSigma \pinvb(\varSigma)$.

Now, show that the minimum in \eqref{eqTLS118} is attained.
The constraint $\rank(C-\Delta) \le n $ is satisfied if and only if
all the minors of $C-\Delta$ of order $n+1$ vanish.
Thus the set of all $\Delta$
that satisfy the constraints (the constraint set)
is defined by $\frac{m! (n+d)!}{(n+1)!^2 (m-n-1)! (d-1)!} + 1$
algebraic equations; and so it is closed.
The constraint set is nonempty \textit{almost surely\/}
because it contains $\widetilde C$.
The functional $\|\Delta \pinvb(\varSigma)\|_F$
is a pseudonorm on $\mathbb{R}^{m\times (n+d)}$,
but it is a norm on the linear subspace
$\{\Delta : \Delta \, (I - \pinvb(\varSigma)) = 0\}$,
where it induces a natural subspace topology.
The constraint set is closed on the subspace
(with the norm),
and whenever it is nonempty (i.e., almost surely),
it has a minimal-norm element.

Notice that under condition \eqref{cond:G3}
the constrain set is non-empty always
and not just almost surely. This follows from
Proposition~\ref{prop:gep6.7}.

For the matrix $\Delta$ that is a solution to minimization problem
\eqref{eqTLS118}, consider the rowspace
 $\colspana<(C-\Delta)^\top>$ of the matrix $C-\Delta$.
Its dimension does not exceed $n$.
Its orthogonal basis can be completed to the orthogonal basis
in $\mathbb{R}^{n+d}$,
and the complement consists of $n+d-\rank(C-\Delta) \ge d$
vectors.
Choose $d$ vectors from the complement, which are linearly independent,
and bind them (as column-vectors) into $(n+d) \times d$
matrix $\hxx$.
The matrix $\hxx$ satisfies the equation
\begin{gather}
(C-\Delta) \widehat X_{\rm ext} = 0. \label{eqTLSX124}
\end{gather}
If the lower $d\times d$ block of the matrix $\hxx$ is a
nonsingular matrix,
by linear transformation of columns
(i.e., by right-multiplying by some nonsingular matrix)
the matrix $\hxx$ can be transformed to the form
\[
\begin{pmatrix} \widehat X \\ -I
\end{pmatrix},
\]
where $I$ is $d\times d$ identity matrix.
The matrix $\widehat X$ satisfies the equation
\begin{gather}
(C-\Delta) \begin{pmatrix} \widehat X \\  -I \end{pmatrix} = 0 . \label{eqTLSX126}
\end{gather}
(Otherwise, if the lower block of the matrix $\hxx$ is singular,
then our estimation fails.
Note that whether the lower block of the matrix $\hxx$ is singular
might depend not only on the observations $C$,
but also on the choice of the matrix $\Delta$ where the minimum
in \eqref{eqTLS118} in attained and the $d$ vectors that make matrix $\hxx$.
We will show that the lower block of the matrix $\hxx$ is nonsingular
with high probability regardless of the choice of $\Delta$ and $\widehat X_{\rm ext}$.)

Columns of the matrix $\widehat X_{\rm ext}$
should span the eigenspace (generalized invariant space)
of the matrix pencil $\langle C^\top C , \varSigma\rangle$
which corresponds to the
$d$ smallest generalized eigenvalues.
%
That the columns of the matrix $\hxx$
span the generalized invariant space
corresponding to finite generalized eigenvalues
is written in the matrix notation
as follows:
\[
\exists M{\in}\mathbb{R}^{d\times d}:\; C^\top C \widehat
X_{\rm ext} = \varSigma \widehat X_{\rm ext} M.
\]

Possible problems that may arise in the course of solving
the minimization problem~\eqref{eqTLS118} are discussed in \cite{Shklyar2011}.
We should mention that our two-step definition
$\eqref{eqTLS118}$ \& $\eqref{eqTLSX126}$ of the TLS estimator
is slightly different from the conventional definition
in \cite[Sections 2.3.2 and 3.2]{VanHuffelVandewalle1991} or in
\cite{KVH02}.
In these papers, the problem from which the estimator
$\widehat X$ is found is equivalent to the following:
\begin{equation}
\begin{cases}
\| \Delta \, \pinvp(\varSigma^{1/2}) \|_F\to\min; \\
\Delta \, (I-P_\varSigma) = 0; \\
(C - \Delta) \begin{pmatrix} \widehat X \\ -I \end{pmatrix} = 0,
\end{cases} \label{eqTLS700}
\end{equation}
where the optimization is performed for $\Delta$ and $\widehat X$ that satisfy the constraints in
\eqref{eqTLS700}.
If our estimation defined with \eqref{eqTLS118} and \eqref{eqTLSX126} succeeds,
then the minimum values in \eqref{eqTLS118} and \eqref{eqTLS700} coincide,
and the minimum in \eqref{eqTLS700} is attained for $(\Delta, \widehat X)$
that is the solution to \eqref{eqTLS118} \& \eqref{eqTLSX126}.
\querymark{Q3}Conversely, if our estimation succeeds for at least one choice of $\Delta$ and $\hxx$,
then all the solutions to  \eqref{eqTLS700} can be obtained with different choices of
$\Delta$ and $\hxx$.
However, strange things may happen if our estimation always fails.

Besides (\ref{eqTLS118}), consider the optimization problem
\begin{equation}
\begin{cases}
\lambda_{\max} (\Delta \pinvb(\varSigma) \Delta^\top) \to\min; \\
\Delta \, (I-P_\varSigma) = 0; \\
\rank(C-\Delta) \le n.
\end{cases} \label{eqTLS220}
\end{equation}
It will be shown that every $\Delta$ that minimizes (\ref{eqTLS118})
also minimizes (\ref{eqTLS220}).

We can construct the optimization problem
that generalizes \querymark{Q4}both (\ref{eqTLS118}) and (\ref{eqTLS220}).
Let $\uin{M}$ be a unitarily invariant norm
on $m\times(n+d)$ matrices.
Consider the optimization problem
\begin{equation}
\begin{cases}
\uin{\Delta \, \pinvp(\varSigma^{1/2})}\to\min; \\
\Delta \, (I-P_\varSigma) =0; \\
\rank(C-\Delta) \le n.
\end{cases} \label{eqTLS118uin}
\end{equation}
Then every $\Delta$ that minimizes (\ref{eqTLS118})
also minimizes (\ref{eqTLS118uin}), and
every $\Delta$ that minimizes (\ref{eqTLS118uin})
also minimizes (\ref{eqTLS220}).
If $\uin{M}$ is the Frobenius norm, then
optimization problems (\ref{eqTLS118}) and (\ref{eqTLS118uin}) coincide,
and if $\uin{M}$ is the spectral norm, then
optimization problems (\ref{eqTLS220}) and (\ref{eqTLS118uin}) coincide.

\begin{rem}
A solution to problem (\ref{eqTLS118}) or (\ref{eqTLS220})
does not change if the matrix $\varSigma$ is multiplied by
a positive scalar factor.
Thus, instead of assuming that the matrix $\varSigma$
is known completely, we can assume that $\varSigma$
is known up to a scalar factor.
\end{rem}

\section{Known consistency results}\label{sec:knownconsth}
In this section we briefly revise known consistency results.
One of conditions for the consistency of the TLS estimator
is the convergence of $\frac{1}{m} A_0^\top A_0$
to a nonsingular matrix.
It is required, for example, in \cite{Gleser1981}.
The condition is relaxed in the paper by Gallo
\cite{Gallo82}.
\begin{thm}[Gallo \cite{Gallo82}, Theorem 2]\label{thm:Gallo}
Let $d=1$,
\begin{align*}
m^{-1/2} \lambda_{\min} \bigl(A_0^\top
A_0 \bigr) &\to \infty \quad\mbox{as}\quad m\to\infty,
\\
\frac{\lambda_{\min}^2 (A_0^\top A_0)}{\lambda_{\max} (A_0^\top A_0)}& \to \infty \quad\mbox{as}\quad m\to\infty,
\end{align*}
and the measurement errors $\tilde c_i$ are identically distributed,
with finite fourth moment $\Me \|\tilde c_i\|^4 < \infty$.
Then $\widehat X \probto \xtrue$, $m\to\infty$.
\end{thm}

The theorem can be generalized for the multivariate regression.  The
condition that the errors on different observations have the same
distribution can be dropped.
Instead, Kukush and Van Huffel \cite{KVH02} assume that the fourth moments of
the error distributions are bounded.
\begin{thm}[Kukush and Van Huffel \cite{KVH02}, Theorem 4a]\label{thm:Kukush:4a}
Let
\begin{align*}
\sup_{\substack{i\ge 1\\ j=1{,\ldots,}n+d}} \Me |\delta_{ij}|^{4} &<
\infty,
\\
m^{-1/2} \lambda_{\min} \bigl(A_0^\top
A_0 \bigr) &\to \infty\quad\mbox{as}\quad m\to\infty,
\\
\frac{\lambda_{\min}^2 (A_0^\top A_0)}{\lambda_{\max} (A_0^\top A_0)} &\to \infty\quad\mbox{as}\quad m\to\infty.
\end{align*}
Then $\widehat X\probto \xtrue$ as $m\to\infty$.
\end{thm}

Here is the strong consistency theorem:
\begin{thm}[Kukush and Van Huffel \cite{KVH02}, Theorem 4b]\label{thm:Kukush:4b}
Let for some $r\ge 2$ and $m_0\ge 1$,
\begin{align*}
\sup_{\substack{i\ge 1\\ j=1{,\ldots,}n+d}} \Me |\delta_{ij}|^{2r}&<\infty,
\\
\sum_{m=m_0}^\infty \biggl(\frac{\sqrt{m}} {
\lambda_{\min} (A_0^\top A_0)}
\biggr)^r &< \infty,
\\
\sum_{m=m_0}^\infty \biggl( \frac
{\lambda_{\max}  (A_0^\top A_0)}{
\lambda_{\min}^2 (A_0^\top A_0)}
\biggr)^r &<\infty.
\end{align*}
Then $\widehat X \to \xtrue$ as $m\to\infty$, almost surely.
\end{thm}

In the following consistency theorem the moment condition imposed
on the errors is relaxed.
\begin{thm}[Kukush and Van Huffel \cite{KVH02}, Theorem 5b]\label{thm:Kukush:5b}
Let for some $r$, $1\le r < 2$,
\begin{align*}
\sup_{\substack{i\ge 1\\ j=1{,\ldots,}n+d}} \Me |\delta_{ij}|^{2r}&<\infty,
\\
m^{-1/r} \lambda_{\min} \bigl(A_0^\top
A_0 \bigr) &\to \infty \quad\mbox{as}\quad m\to\infty,
\\
\frac{\lambda_{\min}^2 (A_0^\top A_0)}{\lambda_{\max} (A_0^\top A_0)} &\to \infty \quad\mbox{as}\quad m\to\infty.
\end{align*}
Then $\widehat X\probto \xtrue$ as $m\to\infty$.
\end{thm}

\label{sec:ourConsThs}
Generalizations of Theorems
\ref{thm:Kukush:4a},  \ref{thm:Kukush:4b}, and \ref{thm:Kukush:5b}
are obtained in \cite{Shklyar2011}.
An essential improvement is achieved.
Namely, it is not required that $\lambda_{\min}^{-2} (A_0^\top
A_0)\lambda_{\max} (A_0^\top A_0)$ converge to $0$.{\sloppy\par}
\begin{thm}[Shklyar \cite{Shklyar2011}, Theorem 4.1, generalization of Theorems \ref{thm:Kukush:4a} and
\ref{thm:Kukush:5b}]
\label{thm-3.1}
Let for some $r$, $1\le r \le 2$,
\begin{align*}
\sup_{\substack{i\ge 1\\ j=1{,\ldots,}n+d}} \Me |\delta_{ij}|^{2r}&<\infty,
\\
m^{-1/r} \lambda_{\min} \bigl(A_0^\top
A_0 \bigr) &\to \infty\quad\mbox{as}\quad m\to\infty.
\end{align*}
Then $\widehat X\probto \xtrue$ as $m\to\infty$.
\end{thm}

\begin{thm}[Shklyar \cite{Shklyar2011}, Theorem 4.2, generalization of Theorem~\ref{thm:Kukush:4b}]
\label{thm-3.2}
Let for some $r\ge 2$ and $m_0\ge 1$,
\begin{align*}
\sup_{\substack{i\ge 1\\ j=1{,\ldots,}n+d}} \Me |\delta_{ij}|^{2r}&<\infty,
\\
\sum_{m=m_0}^\infty \biggl(\frac{\sqrt{m}} {
\lambda_{\min} (A_0^\top A_0)}
\biggr)^r &< \infty.
\end{align*}
Then $\widehat X \to \xtrue$ as $m\to\infty$, almost surely.
\end{thm}

In the next theorem strong consistency is obtained for $r<2$.
\begin{thm}[Shklyar \cite{Shklyar2011}, Theorem 4.3]
\label{thm-3.3}
Let for some $r$ ($1 \le r \le 2$) and $m_0\ge 1$,
\begin{gather*}
\sup_{\substack{i\ge 1\\ j=1{,\ldots,}n+d}} \Me |\delta_{ij}|^{2r}<\infty,
\qquad \sum_{m=m_0}^\infty \frac{1}{\lambda_{\min}^r (A_0^\top A_0)}<
\infty.
\end{gather*}
Then
$\widehat X \to \xtrue$ as $m\to\infty$, almost surely.
\end{thm}

The key point of the proof is the application
of our own theorem on perturbation bounds for generalized eigenvectors
(Theorems \ref{lem-pertrv0} and \ref{lem-pertrVV0},
see also \cite{Shklyar2011}).
The conditions were relaxed by renormalization of the data.

\section{Existence and uniqueness of the estimator}\label{sect:eue}
When we speak of sequence $\{A_m,\; m\ge 1\}$ of random events
parametrized by sample size $m$, we say
that a random event occurs \textit{with high probability\/}
if the probability of the event tends to 1 as $m\to\infty$,
and we say that a random event occurs \textit{eventually}
if almost surely there exists $m_0$ such that
the random event occurs whenever $m > m_0$,
that is $\Prob(\liminf\limits_{m\to\infty} A_m) = 1$.
(In this definition, $A_m$ are random events.
Elsewhere in this paper, $A_m$ are matrices.)

\begin{thm}\label{thm:wdue}
Under the conditions of Theorem~\ref{thm-3.1},
the following three events occur with high probability;
under the conditions of Theorem \ref{thm-3.2} or~\ref{thm-3.3},
the following relations occur eventually.
\begin{enumerate}
\item\label{wdue1}
The constrained minimum in \eqref{eqTLS118} is attained.
If $\Delta$ satisfies the constraints in \eqref{eqTLS118}
$($particularly, if matrix $\Delta$ is a solution
to optimization problem \eqref{eqTLS118}$)$, then
the linear equation \eqref{eqTLSX124} has a solution $\hxx$
that is a full-rank matrix.
\item\label{wdue15}
The optimization problem \eqref{eqTLS118} has a unique solution $\Delta$.
\item\label{wdue2}
For any $\Delta$ that is a solution to \eqref{eqTLS118},
equation \eqref{eqTLSX126}
(which is a linear equation in $\widehat X$)
has a unique solution.
\end{enumerate}
\end{thm}

\begin{thm}\label{thm:spnv}\mbox{}
\begin{enumerate}
\item\label{thm-spnv:part1}
The constrained minimum in \eqref{eqTLS220} is attained.
If $\Delta$ satisfies the constraints in \eqref{eqTLS220}, then
the linear equation \eqref{eqTLSX124} has a solution $\hxx$
that is a full-rank matrix.

\item\label{thm-spnv:part2}
Under the conditions of Theorem~\ref{thm-3.1},
the following random event occurs with high probability:
for any $\Delta$ that is a solution to \eqref{eqTLS220},
equation \eqref{eqTLSX126}
has a solution $\widehat X$.
(Equation~\eqref{eqTLSX126} might have multiple solutions.)
The solution is a consistent estimator of $X_0$,
i.e., $\widehat X \to X_0$ in probability.

\item\label{thm-spnv:part3}
Under the conditions of Theorem \ref{thm-3.2} or~\ref{thm-3.3},
the following random event occurs eventually:
for any $\Delta$ that is a solution to \eqref{eqTLS220},
equation \eqref{eqTLSX126}
has a solution $\widehat X$.
The solution is a strongly consistent estimator of $X_0$,
i.e., $\widehat X \to X_0$ almost surely.
\end{enumerate}
\end{thm}

\begin{remforp}\label{rem:spnvuni}
Theorem \ref{thm:spnv} can be generalized in
the following way:
all references to \eqref{eqTLS220}
can be changed into references to \eqref{eqTLS118uin}.
Thus, if Frobenius norm in the definition of the estimator
is changed to any unitarily invariant norm,
the consistency results are still valid.
\end{remforp}

\section{Sketch of the proof of Theorems \ref{thm-3.1}--\ref{thm-3.3}}
\label{sec:sketch}
Denote
\[
N = C_0^\top C_0^{} +
\lambda_{\rm min} \bigl(A_0^\top A_0^{}
\bigr) I .
\]
Under the conditions of any of the consistency
theorems in Section~\ref{sec:knownconsth}
there is a convergence
$\lambda_{\min} (A_0^\top A_0^{}) \to \infty$. Hence
the matrix $N$ is nonsingular for $m$ large enough.
The matrix $N$ is used as the denominator
in the law of large numbers.
Also, it is used for rescaling the problem:
\querymark{Q5}the condition number of $N^{-1/2} C_0^\top C_0^{} N^{-1/2}$
equals $2$ at most.

The proofs of consistency theorems differ one from another,
but they have the same structure
and common parts.
First, the law of large numbers
\begin{equation}
\label{eq:spstr1} N^{-1/2} \bigl(C^\top C - C_0^\top
C_0^{} - m \varSigma \bigr) N^{-1/2} =
N^{-1/2} \sum_{i=1}^m
\bigl(c_i^\top c_i^{} -
\bigl(c_i^{0} \bigr)^\top c_i^{0}
- \varSigma \bigr) N^{-1/2} \to 0
\end{equation}
holds either in probability or almost surely,
which depends on the theorem being proved.
The proof varies for different theorems.

The inequalities \eqref{neq:maxsinbouni2} and \eqref{neq:maxsinbomul2}
imply that whenever convergence \eqref{eq:spstr1} occurs,
the sine between vectors $\hxx$ and $\xxtrue$
(in the univariate regression)
or the largest of sines of canonical values between column
spans of matrices $\hxx$ and $\xxtrue$ tends to $0$
as the sample size $m$ increases:
\begin{equation}
\label{eq:spstr2} \big\|\sin\angle(\hxx, \xxtrue)\big\| \le \big\|\sin\angle
\bigl(N^{1/2}\hxx, N^{1/2}\xxtrue \bigr)\big\| \to 0.
\end{equation}
To prove \eqref{eq:spstr2}, we use some algebra, the fact that
$\xxtrue$ (in the univariate model) or
the columns of $\xxtrue$ (in the multivariate model) are the minimum-eigenvalue
eigenvectors of matrix $N$
(see ineq.~\eqref{neq:sinescal}),
and eigenvector perturbation theorems --
Lemma~\ref{lem-pertrv0} or Lemma~\ref{lem-pertrVV0}.

Then, by Theorem~\ref{thm-neqsindist} 
we conclude that
\begin{equation}
\label{eq:spstr3} \|\widehat X - X_0\| \to 0.
\end{equation}

\section{Relevant classical results}\label{sec:Theory}
We use some classical results.
However, we state them in a form convenient for our study
and provide the proof for some of them.
\subsection{Generalized eigenvectors and eigenvalues}
In this paper we deal with real matrices.
Most theorems in this section
can be generalized for matrices with complex entries
by 
requiring that matrices be Hermitian rather than symmetric,
and by complex conjugating where it is necessary.
\begin{thm}[Simultaneous diagonalization of a definite matrix pair]
\label{thm-gedpdp}
Let $A$ and $B$ be $n \times n$ symmetric matrices such that
for some $\alpha$ and $\beta$ the matrix $\alpha A+\beta B$
is positive definite. Then there exist a nonsingular matrix $T$
and diagonal matrices $\varLambda$ and ${\rm M}$ such that
\[
A = \bigl(T^{-1} \bigr)^\top \varLambda T^{-1},
\qquad B= \bigl(T^{-1} \bigr)^\top {\rm M} T^{-1}.
\]
\end{thm}
If in the decomposition
$T=[u_1,u_2,\ldots,u_n]$, $\varLambda =\diag(\lambda_1,\allowbreak\ldots,\lambda_n)$,
${\rm M}=\diag(\mu_1,\ldots,\mu_n)$,
then the numbers $\lambda_i/\mu_i\in \mathbb{R}\cup \{\infty\}$ are called generalized eigenvalues,
and the columns $u_i$ of the matrix $T$ are called the right generalized eigenvectors
of the matrix pencil $\langle A, B\rangle$
because the following relation holds true:
\[
\mu_i A u_i = \lambda_i B u_i.
\]

Theorem~\ref{thm-gedpdp} is well known; see
Theorem~IV.3.5 in \cite[page 318]{StewartSun}.
The conditions of  Theorem \ref{thm-gedpdp} can be changed as follows:
\begin{thm}
\label{thm-gedpsp}
Let $A$ and $B$ be symmetric positive semidefinite matrices.
Then there exist a nonsingular matrix $T$
and diagonal matrices $\varLambda$ and ${\rm M}$ such that
\begin{equation}
\label{eq:0501} A = \bigl(T^{-1} \bigr)^\top \varLambda
T^{-1}, \qquad B= \bigl(T^{-1} \bigr)^\top {\rm M}
T^{-1}.
\end{equation}
\end{thm}

In Theorem \ref{thm-gedpdp} $\lambda_i$ and $\mu_i$ cannot be
equal to 0 for the same $i$, while in Theorem \ref{thm-gedpsp} they can.
On the other hand, in Theorem \ref{thm-gedpdp} $\lambda_i$ and $\mu_i$
can be any real numbers, while in Theorem \ref{thm-gedpsp}
$\lambda_i \ge 0$ and $\mu_i\ge 0$.
Theorem~\ref{thm-gedpsp} is proved in \cite{Newcomb1961}.

\begin{remforp}\label{rem:remarkeqrank}
If the matrices $A$ and $B$ are symmetric and positive semidefinite,
then
\begin{equation}
\label{eq:remarkeqrank} \rank\langle A, B \rangle = \rank(A+B),
\end{equation}
where
\[
\rank\langle A, B \rangle = \max_k \rank(A + k B)
\]
is the \textit{determinantal rank\/} of the matrix pencil $\langle A, B \rangle$. (For square $n\times n$ matrices
$A$ and $B$, the determinantal rank characterizes
if the matrix pencil is regular or singular.
The matrix pencil $\langle A,B\rangle$ is regular
if  $\rank\langle A, B \rangle = n$, and
singular if $\rank\langle A, B \rangle < n$.)

The inequality $\rank\langle A, B \rangle \ge \rank(A+B)$
follows from the definition of the determinantal rank.
For all $k \in \mathbb{R}$ and for all such vectors $x$
that $(A + B) x = 0$ we have
$x^\top A x + x^\top B x = 0$, and
because of positive semidefiniteness of matrices $A$ and $B$,
$x^\top A x \ge 0$ and $x^\top B x \ge 0$.
Thus, $x^\top A x = x^\top B x = 0$.
Again, due to positive semidefiniteness of $A$ and $B$,
$A x = B x = 0$ and $(A + kB) x = 0$.
Thus, for all $k \in \mathbb{R}$
\begin{align*}
\bigl\{x : (A + B) x = 0 \bigr\} &\subset \bigl\{x : (A + k B) x = 0 \bigr\},
\\
\rank(A+B) &\ge \rank(A+kB),
\\
\rank\langle A, B\rangle = \max_{k} \rank (A+kB) &\le
\rank(A + B),
\end{align*}
and \eqref{eq:remarkeqrank} is proved.
\end{remforp}

\begin{remforp}\label{rem:remark5.2-1}
Let $A$ and $B$ be positive semidefinite matrices of the same size
such that $\rank(A+B) = \rank(B)$.
The representation \eqref{eq:0501} might be not unique.
But there exists a representation \eqref{eq:0501} such that
\begin{align*}
\lambda_i &= \mu_i = 0 \quad \text{if} \quad i=1,\ldots,
\defect(B),
\\
\mu_i &> 0 \quad \text{if} \quad i = \defect(B)+1,\ldots,n,
\\
T &= \bigl[\!\underset{n \times \defect(B)} {T_1}\,\, \underset{n
\times \rank(B)} {T_2}\! \bigr],
\\
T_1^\top T_2^{} &= 0.
\end{align*}
(Here if the matrix $B$ is nonsingular, then
$T_1$ is $n\times0$ empty matrix;
if $B=0$, then $T_2$ is $n\times0$ matrix.
In these marginal cases, $T_1^\top T_2$ is an empty matrix
and is considered to be zero matrix.)
The desired representation can be obtained from
\cite{deLeeuw} for $S = 0$ (in de Leeuw's notation).
This representation is constructed as follows.
Let the columns of matrix $T_1$ make the orthogonal normalized basis
of $\Kernel(B) = \{v : B v = 0\}$.
There exists $n\times \rank(B)$ matrix $F$ such that $B = F F^\top$.
Let the columns of matrix $L$ be the orthogonal normalized eigenvectors
of the matrix $\pinvb(F) A (\pinvb(F))^\top$.
Then set $T_2 =  (\pinvb(F))^\top L$.
Note that the notation $S$, $F$ and $L$ is borrowed from
\cite{deLeeuw}, and is used only once.
Elsewhere in the paper,
the matrix $F$ will have a different meaning.
\end{remforp}

\begin{prop}\label{prop:forremark5.2-1}
Let $A$ and $B$ be symmetric positive semidefinite
matrices such that
$\rank(A+B) = \rank(B)$.
In the simultaneous diagonalization in Theorem \ref{thm-gedpsp}
with Remark~\ref{rem:remark5.2-1}
\begin{align*}
\pinvb(B) &= T \pinvb(\mathrm{M}) T^\top,
\\
\pinvb(\mathrm{M}) &= \diag \bigl(\underbrace{0, \ldots, 0}_{\defect(B)},
\mu^{-1}_{\defect(B)+1}, \ldots, \mu^{-1}_n
\bigr).
\end{align*}
\end{prop}
\begin{proof}
Let us verify the Moore--Penrose conditions:
\begin{align}
\label{eq:czMPi} \bigl(T^{-1}\bigr)^\top \textrm{M} T^{-1}\,\pinvXSXTMT\,\bigl(T^{-1}\bigr)^\top \textrm{M} T^{-1}
&= \bigl(T^{-1}\bigr)^\top \textrm{M} T^{-1},
\\
\label{eq:czMPii} \pinvXSXTMT\,\bigl(T^{-1}\bigr)^\top \textrm{M} T^{-1}\,\pinvXSXTMT &= \pinvXSXTMT,
\end{align}
and the fact that the matrices $\XSXTMT\,\pinvXSXTMT$ and
$\pinvXSXTMT \times\break \XSXTMT$ are symmetric. The equalities
$\eqref{eq:czMPi}$ and $\eqref{eq:czMPii}$ can be
verified directly; and the symmetry properties can be
reduced to the equality
\begin{equation}
\label{eq:pinvB-forsym} \bigl(T^{-1} \bigr)^\top P_{\mathrm{M}}
T^\top = T^{} P_{\mathrm{M}} T^{-1}
\end{equation}
with $P_{\mathrm{M}} = \mathrm{M} \pinvb(\mathrm{M}) =
\diag(\underbrace{0,\ldots,0}\limits_{\defect(B)},
      \underbrace{1,\ldots,1}_{\rank(B)}) $.

Since $T_1^\top T_2^{} = 0$, $T^\top T^{}$ is
a block diagonal matrix.
Hence $P_{\mathrm{M}} T^\top T =
T^\top T^{} P_{\mathrm{M}}$,
whence \eqref{eq:pinvB-forsym} follows.
\end{proof}

\subsection{Angle between two linear subspaces}\label{sss:sin}
Let $V_1$ and $V_2$ be linear subspaces of $\mathbb{R}^n$,
with $\dim V_1 = k_1 \le \dim V_2 = k_2$.
Then there exists an orthogonal $n \times n$ matrix U
such that
\begin{align}
\label{eq:defcaV1} V_1 &= \colspan \left\langle  U \begin{pmatrix}
\diag_{k_2 \times k_1}
(\cos \theta_i, \; i=1,\ldots,k_1) \\
\diag_{(n-k_2) \times k_1}
(\sin \theta_i, \; i=1,\ldots,\min(n - k_2, \: k_1 ))
\end{pmatrix}
\right\rangle  ,
\\
V_2 &= \colspan \left\langle  U \begin{pmatrix}
I_{k_2} \\
0_{(n-k_2) \times k_2}
\end{pmatrix} \right\rangle  .
\label{eq:defcaV2}
\end{align}
Here rectangular diagonal matrices are allowed.
If in \eqref{eq:defcaV1} there are more cosines than sines
(i.e., if $k_2 + k_1 > n$),
then the excessive cosines should be equal to 1,
so the columns of the bidiagonal matrix in \eqref{eq:defcaV1} are unit vectors (which are orthogonal
to each other).
Here the columns of $U$ are the vectors of some convenient ``new'' basis in
$\mathbb{R}^n$, so $U$ is a transitional matrix from the standard basis to ``new'' basis;
the columns of matrix products in $\colspana<\cdots>$
in \eqref{eq:defcaV1} and \eqref{eq:defcaV2} are the vectors of the bases
of subspaces $V_1$ and $V_2$;
the bidiagonal matrix in \eqref{eq:defcaV1} and
the diagonal matrix in \eqref{eq:defcaV2}
are the transitional matrices from ``new'' basis in $\mathbb{R}^n$
to the bases in $V_1$ and $V_2$, respectively.

The angles $\theta_k$ are called the canonical angles between $V_1$ and $V_2$.
They can be selected so that
 $0 \le \theta_k \le \frac{1}2 \pi$
(to achieve this, we might have to reverse some vectors of the bases).

Denote $P_{V_1}$ the matrix of the orthogonal projector onto $V_1$.
The singular values of the matrix $P_{V_1} (I-P_{V_2})$
are equal to $\sin\theta_k$ ($k=1,\ldots,k_1$);
besides them, there is a singular value $0$
of multiplicity $n-k_1$.

Denote the greatest of the sines of the canonical eigenvalues
\begin{equation}
\label{eq:defmaxsin} \big\|\sin \angle(V_1, V_2) \big\| = \max
_{k=1,\ldots,k_1} \sin \theta_k = \big\| P_{V_1} (I -
P_{V_2}) \big\| .
\end{equation}

If $\dim V_1 = 1$, $V_1 = \colspana<v>$, then
\[
\sin \angle(v, V_2) = \biggl\llVert (I - P_{V_2})
\frac{v}{\|v\|} \biggr\rrVert = \dist \biggl( \frac{1}{\|v\|}v,
V_2 \biggr).
\]
This can be generalized for $\dim V_1 \ge 1$:
\[
\big\|\sin \angle(V_1, V_2) \big\| = \max_{v \in V_1 \setminus\{0\}}
\biggl\llVert (I - P_{V_2}) \frac{v}{\|v\|} \biggr\rrVert ,
\]
whence
\begin{align}
\big\|\sin \angle(V_1, V_2) \big\|^2 &= \max
_{v \in V_1 \setminus\{0\}} \frac{v^\top (I - P_{V_2}) v} {
\|v\|^2},
\nonumber
\\
1 - \big\|\sin \angle(V_1, V_2) \big\|^2 &= \min
_{v \in V_1 \setminus\{0\}} \frac{v^\top P_{V_2} v} {
\|v\|^2}. \label{eq:1-sin2}
\end{align}

If $\dim V_1 = \dim V_2$, then
$\|\sin \angle(V_1, V_2) \| = \| P_{V_1} - P_{V_2} \|$, and therefore\break
$\|\sin \angle(V_1, V_2) \| = \|\sin \angle(V_2, V_1) \|$.
Otherwise the right-hand side of \eqref{eq:defmaxsin}
may\break change if $V_1$ and $V_2$ are swapped
(particularly, if $\dim V_1 < \dim V_2$,
then $\| P_{V_1} (I - P_{V_2}) \|$ may or may not
be equal to 1, but always $\| P_{V_2} (I - P_{V_1}) \| = 1$;
see the proof of Lemma~\ref{lem:bofosin2}
in the appendix).

We will often omit ``span'' in arguments of sine.
Thus, for $n$-row matrices $X_1$ and $X_2$,
$\|\sin\angle(X_1, V_2)\| = \|\sin\angle(\colspana<X_1>,V_2)\|$ and
$\|\sin\angle(X_1, X_2)\| =\break \|\sin\angle(\colspana<X_1>,\colspana<X_2>)\|$.

{\sloppy
\begin{lem}\label{lem:subsset1112132}
Let $V_{11}$, $V_2$ and $V_{13}$ be three linear subspaces in $\mathbb{R}^n$,
with $\dim V_{11} = d_1 < \dim V_2 = d_2 < \dim V_{13} = d_3$
and $V_{11} \subset V_{13}$.
Then there exists such a linear subspace $V_{12}\subset\mathbb{R}^n$
that $V_{11} \subset V_{12} \subset V_{13}$, $\dim V_{12} = d_2$, and
$\|\sin\angle(V_{12},V_2)\| = 1$.
\end{lem}
}\begin{proof}
Since $\dim V_{13} + \dim V_2^\bot = d_3 + n - d_2 > n$,
there exists a vector $v \neq 0$, $v \in V_{13} \cap V_2^\bot$.
Since $\max(d_1, 1) \le \dim \colspana<V_{11}, v> \le d_1 + 1$,
it holds that
\[
\dim \colspana<V_{11}, v> \le d_2 < \dim
V_{13} .
\]
Therefore, there exists a $d_2$-dimensional subspace $V_{12}$
such that $\colspana<V_{11}, v> \,{\subset}\, V_{12} \subset V_{13}$.
Then $V_{11}\subset V_{12} \subset V_{13}$ and
$v \in V_{12} \cap V_2^\bot$.
Hence
$P_{V_{12}} (I - P_{V_2}) v = v$,
$\|P_{V_{12}} (I - P_{V_2})\| \ge 1$,
and due to equation \eqref{eq:defmaxsin},
$\|\sin\angle(V_{12},\, V_2)\| = 1$.
Thus, the subspace $V_{12}$ has the desired properties.
\end{proof}

\subsection{Perturbation of eigenvectors and invariant spaces}
\begin{lem}
\label{lem-pertrv0}
Let $A$, $B$, $\tilde A$ be symmetric matrices,
$\lambda_{\min}(A)=0$, $\lambda_2(A)>0$
and $\lambda_{\min}(B)\ge 0$.
Let $A x_0=0$ and  $B x_0\neq 0$
(so $x_0$ is an eigenvector of the matrix $A$ that corresponds to the
minimum eigenvalue).
Let minimum of the function
\[
f(x):=\frac{x^\top (A+\tilde A) x}{x^\top B x}, \qquad x^\top B x>0,
\]
be attained at the point $x_*$. Then
\[
\sin^2 \angle(x_*,x_0) \le \frac{\|\tilde A\|}{\lambda_2(A)} \biggl( 1+
\frac{\|x_0\|^2}{x_0^\top B x_0}\, \frac{x^\top B x} {\|x\|^2} \biggr).
\]
\end{lem}
\begin{remforp}\label{rem:lem-pertrv0}
The function $f(x)$ may or may not attain the minimum.
Thus the condition $f(x_*) = \min_{x^\top B x>0} f(x)$
sometimes cannot be satisfied.  But the theorem is still true
if
\begin{equation}
\label{eq-infmin} \liminf\limits
_{x\to x_*}f(x) = \inf\limits
_{x:\; x^\top\! B x > 0} f(x)
\end{equation}
and $x_* \neq 0$.
\end{remforp}

Now proclaim the multivariate generalization of
Lemma~\ref{lem-pertrv0}.
We will not generalize Remark~\ref{rem:lem-pertrv0}.
Instead, we will check that the minimum is attained
when we use Lemma~\ref{lem-pertrVV0}
(see Proposition~\ref{prop-6eximl}).

\begin{lem}
\label{lem-pertrVV0}
Let $A$, $B$, $\tilde A$ be $n \times n$ symmetric matrices,
$\lambda_i(A) = 0$ for all $i{=}1,\ldots,d$,
$\lambda_{d+1}(A) > 0$,
$\lambda_{\min}(B) \ge 0$.
Let $X_0$ be $n \times d$ matrix such that
$A X_0 = 0$ and the matrix $X_0^\top B X_0^{}$ is nonsingular.
Let the functional
\begin{align}
f(X) &= \lambda_{\max} \bigl( \bigl(X^\top B X
\bigr)^{-1} X^\top (A+\tilde A) X \bigr) \quad \mbox{if $X \in
\mathbb{R}^{n\times d}$ and $X^\top B X > 0$,}
\nonumber
\\
f(X) &\;\mbox{is not defined otherwise,} \label{eqf-lempertrVV0}
\end{align}
attain its minimum.
Then for any point $X$ where the minimum is attained,
\[
\big\|\sin\angle (X, X_0)\big\|^2 \le \frac{\|\tilde A\|}{\lambda_{d+1}(A)} \bigl(
1 + \|B\|\, \lambda_{\max} \bigl( \bigl(X_0^\top B
X_0^{} \bigr)^{-1} X_0^\top
X_0^{} \bigr) \bigr) .
\]
\end{lem}

\subsection{Rosenthal inequality}
In the following theorems, a random variable $\xi$
is called \textit{centered\/} if $\Me \xi = 0$.\vadjust{\goodbreak}
\begin{thm}
\label{thm-RosenNeq-2+}
Let $\nu\ge2$ be a nonrandom real number.
Then there exist $\alpha\ge 0$ and $\beta\ge 0$ such that
for any set of centered
mutually independent random variables
$\{\xi_i, i=1,\ldots, m\}$,
$m{\ge}1$,
the following inequality holds true:
\[
\Me \Biggl[ \Bigg| \sum_{i=1}^m
\xi_i \Bigg|^\nu \Biggr] \le \alpha \sum
_{i=1}^m \Me \bigl[|\xi_i|^\nu
\bigr] + \beta \Biggl( \sum_{i=1}^m \Me
\xi_i^2 \Biggr) ^{\nu/2}.
\]
\end{thm}
Theorem~\ref{thm-RosenNeq-2+} is well known; see
\cite[Theorem 2.9, page 59]{Petrov1995}.

\begin{thm}
\label{thm-RosenNeq-12}
Let $\nu$ be a nonrandom real number, $1\le\nu\le2$.
Then there exists $\alpha\ge 0$ such that
for any set of centered
mutually independent random variables
$\{\xi_i, i=1,\ldots, m\}$,
$m{\ge}1$,
the inequality holds true:
\[
\Me \Biggl[ \Bigg| \sum_{i=1}^m
\xi_i \Bigg|^\nu \Biggr] \le \alpha \sum
_{i=1}^m \Me \bigl[|\xi_i|^\nu
\bigr].
\]
\end{thm}
\begin{proof}
The desired inequality is trivial for $\nu = 1$.
For all $1 < \nu \le 2$
it is a consequence of the Marcinkiewicz--Zygmund inequality
\[
\Me \Biggl[ \Bigg| \sum_{i=1}^m
\xi_i \Bigg|^\nu \Biggr] \le \alpha \Me \Biggl[ \Biggl( \sum
_{i=1}^m \xi_i^2
\Biggr)^{\nu/2} \Biggr] \le \alpha \Me \sum_{i=1}^m
|\xi_i|^\nu = \alpha \sum_{i=1}^m
\Me |\xi_i|^\nu.
\]
Here the first inequality is due to Marcinkiewicz and Zygmund
\cite[Theorem~13]{MZ1937}.
The second inequality follows from the fact that
for  $\nu\le 2$,
\begin{equation*}
\Biggl( \sum_{i=1}^m
\xi_i^2 \Biggr)^{\nu/2} \le \sum
_{i=1}^m |\xi_i|^\nu.
\qedhere
\end{equation*}
\end{proof}

\section{Generalized eigenvalue problem for positive semidefinite matrices}\label{sec:psdgep}

In this section we explain the relationship between the TLS
estimator and the generalized eigenvalue problem.
The results of this section are important
for constructing the TLS estimator.
Proposition~\ref{prop:gep6.7} is used to state the uniqueness
of the TLS estimator.

\begin{lem}\label{lem:psdgsvd1}
Let $A$ and $B$ be $n\times n$ symmetric positive semidefinite matrices,
with simultaneous diagonalization
\[
A = \bigl(T^{-1} \bigr)^\top \varLambda T^{-1},
\qquad B = \bigl(T^{-1} \bigr)^\top \mathrm{M}
T^{-1},
\]
with
\[
\varLambda = \diag(\lambda_1,\ldots,\lambda_n), \qquad
\mathrm{M} = \diag(\mu_1,\ldots,\mu_n)
\]
(see Theorem \ref{thm-gedpsp} for its existence).
For $i=1,\ldots,n$ denote
\[
\nu_i = \begin{cases}
      \lambda_i / \mu_i & \text{if $\mu_i>0$,} \\
      0 & \text{if $\lambda_i = 0$,} \\
      +\infty & \text{if $\lambda_i>0$, $\mu_i=0$.}
        \end{cases}
\]
Assume that $\nu_1 \le \nu_2 \le \cdots  \le \nu_n$.
Then
\begin{equation}
\label{eq-lem587} \nu_i = \min \bigl\{ \lambda\ge 0 \mathrel| \text{``}
\exists V, \ \dim V = i : (A - \lambda B)|_V \le 0
\text{''} \bigr\},\vadjust{\goodbreak}
\end{equation}
i.e., $\nu_i$ is the smallest number $\lambda \ge 0$, such that there exists
an $i$-dimensional subspace $V \subset \mathbb{R}^n$,
such that the quadratic form $A - \lambda B$
is negative semidefinite on $V$.
\end{lem}

\begin{remforp}
$\nu_i < \infty$ if and only if
\[
\exists \lambda\ \exists V, \ \dim V = i : (A - \lambda B)|_V \le 0
.
\]
\end{remforp}

\begin{remforp}
Let $\nu_i<\infty$.
The minimum in (\ref{eq-lem587}) is attained
for $V$ being the linear span of first $i$ columns
of the matrix $T$
(i.e., the linear span of the eigenvectors
of the matrix pencil $\langle A, B \rangle$
that correspond to the $i$ smallest generalized eigenvalues).
That is
\[
(A - \nu_i B) |_V \le 0 \quad \text{for} \quad V =
\colspan \bigl\langle T ( \begin{smallmatrix} I_k
\\
0_{(n-k)\times k} \end{smallmatrix} ) \bigr\rangle  .
\]
\end{remforp}

In Propositions \ref{prop-l588}--\ref{prop:6.5} the following optimization problem
is considered.
For a fixed $(n+d)\times d$ matrix $X$ find an $m\times (n+d)$
matrix $\Delta$ where the constrained minimum is attained:
\begin{equation}
\label{TLS-fixX568} \begin{cases}
 \Delta \pinvb(\varSigma) \Delta^\top \to \min ;\\
\Delta \, (I - P_\varSigma) = 0; \\
(C - \Delta) X = 0 .
\end{cases}
\end{equation}
Here the matrix $X$ is assumed to be of full rank:
\begin{equation}
\label{X_frank581} \rank X = d .
\end{equation}

\begin{prop}
\label{prop-l588}
\mbox{1. }The constraints in (\ref{TLS-fixX568}) are compatible if and only if
\begin{equation}
\label{eq:nemset} \operatorname{span}\big\langle X^\top C^\top\big\rangle \subset
\operatorname{span}\big\langle X^\top \varSigma\big\rangle.
\end{equation}
Here $\colspana<M>$ is a column space of the matrix $M$.

\mbox{2. }Let the constraints in (\ref{TLS-fixX568}) be compatible.
Then the least element of the partially ordered set (in the \querymark{Q6}Loewner order)
$\{ \Delta \pinvb(\varSigma) \Delta^\top : \Delta \, (I{-}P_\varSigma) = 0
\allowbreak\; \text{and}\allowbreak\; (C{-}\Delta) X = 0\}$
is attained for $\Delta = C X \pinvp(X^\top \varSigma X) X^\top \varSigma$
and is equal to $C X \pinvp(X^\top \varSigma X) X^\top C^\top$.
This means the following:{\sloppy\par}

\mbox{2a. }%
For  $\Delta = C X \pinvp(X^\top \varSigma X) X^\top \varSigma$,
it holds that
\begin{align}
\label{eq:prop-l588-2a1} \Delta \, (I - P_\varSigma) &= 0, \qquad (C - \Delta) X = 0,
\\
\label{eq:prop-l588-2a2} \Delta \pinvb(\varSigma) \Delta^\top &= C X
\bigl(X^\top \varSigma X \bigr)^{\dagger} X^\top
C^\top;
\end{align}

\mbox{2b. }%
For any $\Delta$ which satisfies the constraints $\Delta \, (I - P_\varSigma) = 0$
and $(C - \Delta) X = 0$,
\begin{equation}
\label{eq:prop-l588-2b} \Delta \pinvb(\varSigma) \Delta^\top \ge C X
\bigl(X^\top \varSigma X \bigr)^{\dagger} X^\top
C^\top.
\end{equation}
\end{prop}

\begin{remforp}\label{rem:6.2-1}
If the constraints are compatible, the least element (and the unique minimum)
is attained at a single point. Namely, the equalities
\begin{align*}
\Delta \, (I - P_\varSigma) &= 0, \qquad (C - \Delta) X = 0,
\\
\Delta \pinvb(\varSigma) \Delta^\top &= C X \bigl(X^\top
\varSigma X \bigr)^{\dagger} X^\top C^\top
\end{align*}
imply $\Delta = C X \pinvp(X^\top \varSigma X) X^\top \varSigma$.
\end{remforp}

\begin{prop}\label{prop:simpifdef}
Let the matrix pencil $\langle C^\top C, \varSigma\rangle$
be definite and \eqref{X_frank581} hold.
The constraints in (\ref{TLS-fixX568}) are compatible
if and only if the matrix $X^\top \varSigma X$
is nonsingular.
Then Proposition~\ref{prop-l588} still holds true
if $(X^\top \varSigma X)^{-1}$ is substituted
for $\pinvp(X^\top \varSigma X)$.\vadjust{\goodbreak}
\end{prop}

\begin{prop}\label{prop:6.4}
Let $X$ be an $(n+d)\times d$ matrix which satisfies \eqref{X_frank581} and makes
the constraints in (\ref{TLS-fixX568}) compatible.
Then for $k=1,2,\ldots,d$,
\begin{align}
&\min_{\substack{\Delta (I - P_\varSigma) = 0\\
(C - \Delta) X = 0}} \lambda_{k+m-d} \bigl(\Delta \pinvb(
\varSigma) \Delta^\top \bigr)\notag
\\
&\quad= \min \bigl\{ \lambda\ge 0 : \mbox{``}\exists V{\subset} \colspana<X>,\; \dim
V{=}k : \bigl(C^\top C - \lambda \varSigma \bigr)|_V \le 0
\mbox{''} \bigr\}. \label{eq-prop64}
\end{align}
\end{prop}

\begin{remforp}\label{remark:forprop6.4}
In the left-hand side of (\ref{eq-prop64}) the minima are
attained for the same $\Delta = C X \pinvp(X^\top \varSigma X) X^\top \varSigma$
for all $k$ (the $k$ sets where the minima are attained have
non-empty intersection;
we will show that the intersection comprises of a single element).

One can choose a stack of subspaces
\[
V_1 \subset V_2 \subset \cdots \subset V_d
= \colspana<X>
\]
such that $V_k$ is the element where the minimum in the right-hand side of (\ref{eq-prop64}) is attained, i.e.,
for all $k=1,\ldots,d$,
\begin{gather*}
\dim V_k = k, \qquad V_k \subset\colspana<X>, \qquad
\bigl(C^\top C - \nu_k \varSigma \bigr) |_{V_k}
\le 0,
\end{gather*}
with $\nu_k =
\min_{\substack{\Delta (I - P_\varSigma) = 0\\
(C - \Delta) X = 0}}
\lambda_{k+m-d} (\Delta \pinvb(\varSigma) \Delta^\top)$.
\end{remforp}

In Propositions \ref{prop:6.5} to \ref{prop:gep6.7},
we will use notation
from simultaneous diagonalization of matrices
$C^\top C$ and $\varSigma$:
\begin{equation}
\label{eq:diddecomposCCS} C^\top C = \bigl(T^{-1} \bigr)^\top
\varLambda T^{-1}, \qquad \varSigma = \bigl(T^{-1}
\bigr)^\top {\rm M} T^{-1},
\end{equation}
where
\begin{align*}
\varLambda &= \diag(\lambda_1,\ldots,\lambda_{n+d}), \qquad
{ \rm M} = \diag(\mu_1,\ldots,\mu_{n+d}),
\\
T &= [u_1, u_2,\ldots, u_d,\ldots,
u_{n+d}] .
\end{align*}
If Remark~\ref{rem:remark5.2-1} is applicable,
let the simultaneous diagonalization
be constructed accordingly.
For $k=1,\ldots,n{+}d$ denote
\[
\nu_i = \begin{cases}
      \lambda_k / \mu_k & \text{if $\mu_k>0$,} \\
      0 & \text{if $\lambda_k = 0$,} \\
      +\infty & \text{if $\lambda_k>0$, $\mu_k=0$.}
        \end{cases}
\]
Let $\nu_k$ be arranged in ascending order.

\begin{prop}\label{prop:6.5}
Let $X$ be an $(n+d)\times d$ matrix which
satisfies \eqref{X_frank581} and makes constraints
in \eqref{TLS-fixX568} compatible.
Then
\begin{equation}
\label{neq-lk-l669} \min_{\substack{\Delta \, (I - P_\varSigma) = 0\\
(C - \Delta) X = 0}} \lambda_{k+m-d} \bigl(
\Delta \pinvb(\varSigma) \Delta^\top \bigr) \ge \nu_k .
\end{equation}

If $\nu_d < \infty$, then for $X=[u_1,u_2,\ldots,u_d]$
the inequality in (\ref{neq-lk-l669}) becomes an equality.
\end{prop}

\begin{cora}
In the minimization problem (\ref{eqTLS220}), the constrained minimum
is equal to
\[
\min_{\substack{\Delta \, (I-P_\varSigma) = 0\\
\rank(C-\Delta) \le n}} \lambda_{\max} \bigl( \Delta \pinvb(
\varSigma) \Delta^\top \bigr) = \nu_d .\vadjust{\goodbreak}
\]
\end{cora}

\begin{prop}\label{prop:gs6.6}
In the minimization problem (\ref{eqTLS118}) the constrained minimum is
equal to
\[
\min_{\substack{\Delta (I-P_\varSigma) = 0\\
\rank(C-\Delta) \le n}} \big\| \bigl(\Delta \, \varSigma^{1/2}
\bigr)^{\dagger}\big\|_F = \sqrt{ \sum
_{k=1}^d \nu_k } .
\]

Whenever the minimum in (\ref{eqTLS118}) is attained for some matrix $\Delta$,
the minimum in (\ref{eqTLS220}) is attained for the same $\Delta$.
\end{prop}

\begin{prop}\label{prop:uniprop}
Let $\uin{M}$ be an arbitrary unitarily invariant
norm on $m\times n$ matrices.
Singular values of the matrix $M$ are arranged
in descending order and denoted $\sigma_i(M)$:
\[
\sigma_1(M) \ge \sigma_2(M) \ge \cdots \ge
\sigma_{\min(m,n)}(M) \ge 0 .
\]
Let $M_1$ and $M_2$ be $m\times n$ matrices. Then
\begin{enumerate}
\item
If $\sigma_i(M_1) \le \sigma_i(M_2)$ for all $i=1,\ldots,\min(m,n)$, then $\uin{M_1} \le \uin{M_2}$.
\item\label{prop:uniprop:part2}
If $\sigma_1(M_1) < \sigma_1(M_2)$ and $\sigma_i(M_1) \le \sigma_i(M_2)$ for all $i=2,\ldots,\min(m,n)$, then $\uin{M_1} < \uin{M_2}$.
\end{enumerate}
\end{prop}

\begin{prop}\label{prop:gs6.6uin}
Consider the optimization problem \eqref{eqTLS118uin}
with arbitrary unitarily invariant norm $\uin{M}$.
Then
\begin{enumerate}
\item\label{prop:gs6.6uin:part1}
Any minimizer $\Delta$ to the optimization problem
\eqref{eqTLS118}
also minimizes
\eqref{eqTLS118uin}.
\item\label{prop:gs6.6uin:part2}
Any minimizer $\Delta$ to the optimization problem
\eqref{eqTLS118uin} also minimizes \eqref{eqTLS220}.
\end{enumerate}
\end{prop}

\begin{prop}\label{prop:gep6.7}
For any $\Delta$ where the minimum in (\ref{eqTLS118}) is attained
and the corresponding solution $\widehat X_{\rm ext}$
of the linear equations (\ref{eqTLSX124})
($\widehat X_{\rm ext}$ is an
$(n+d)\times d$ matrix of rank $d$),
it holds that
\begin{equation}
\label{neq0l709} \colspana<{u_i : \nu_i<
\nu_d}> \subset \colspana<\widehat X_{\rm ext}> \subset
\colspana<{u_i : \nu_i\le\nu_d}> .
\end{equation}

Conversely, if $\nu_d < +\infty$
and the matrix $\widehat X_{\rm ext}$
satisfies conditions (\ref{neq0l709}),
then there exists
a common solution $\Delta$ to the minimization problem (\ref{eqTLS118})
and the linear equations~(\ref{eqTLSX124}).
\end{prop}

As a consequence, if $\nu_d < \nu_{d+1}$,
then (\ref{eqTLS118}) and (\ref{eqTLSX124})
unambiguously determine $\colspana<\widehat X_{\rm ext}>$
of rank $d$.

\begin{prop}
\label{prop-6eximl}
Let $\langle C^\top C, \varSigma\rangle$
be a definite matrix pencil.
Then for any $\Delta$ where the minimum in (\ref{eqTLS220})
is attained,
the corresponding solution $\widehat X_{\rm ext}$
of the linear equations (\ref{eqTLSX124})
(such that $\rank \widehat X_{\rm ext} = d$)
is a point where the minimum of the functional
\begin{equation}
X \mapsto \lambda_{\max} \bigl( \bigl(X^\top \varSigma X
\bigr)^{-1} X^\top C^\top C X \bigr), \quad X{\in}
\mathbb{R}^{(n+d)\times d},\quad X^\top \varSigma X{>}0, \label{eqf-l816}
\end{equation}
is attained.   It is also
a point where the minimum of
\begin{equation}
\label{eqf-l816mm} X \mapsto \lambda_{\max} \bigl( \bigl(X^\top
\varSigma X \bigr)^{-1} X^\top \bigl(C^\top C - m
\varSigma \bigr) X \bigr),
\end{equation}
is attained.
\end{prop}
The functional~\eqref{eqf-l816mm} equals the functional~(\ref{eqf-l816}) minus $m$.\vadjust{\goodbreak}



\section{Appendix: Proofs}\label{sec:apx:Proofs}
\subsection*{Detailed proofs of Theorems \ref{thm-3.1}--\ref{thm-3.3}}
\subsection{Bounds for eigenvalues of some matrices used in the proof}\label{ssapx-pr1}
\subsubsection{Eigenvalues of the matrix $C_0^\top C_0^{}$}
The $(n+d)\times(n+d)$ matrix $C_0^\top C_0^{}$
is symmetric and positive semidefinite.
Since $C_0 \xxtrue = A_0 \xtrue - B_0 = 0$,
the matrix $C_0^\top C_0^{}$ is rank deficient with eigenvalue 0
of multiplicity at least $d$.
As $A_0^\top A_0^{}$ is a $n\times n$ principal submatrix of $C_0^\top C_0^{}$,
\begin{equation}
\lambda_{d+1} \bigl(C_0^\top C_0
\bigr) \ge \lambda_{\min} \bigl(A_0^\top
A_0 \bigr) \label{neq418}
\end{equation}
by the Cauchy interlacing theorem (Theorem IV.4.2 from \cite{StewartSun} used $d$ times).

{\sloppy
Due to inequality (\ref{neq418}), if the matrix $A_0^\top A_0$ is
nonsingular, then $\lambda_{n+1}(C_0^\top C_0)>0$, whence
$\rank(C_0^\top C_0) = d$. If the conditions of Theorem \ref{thm-3.1},
\ref{thm-3.2} or \ref{thm-3.3} hold true, then $\lambda_{\min}
(A_0^\top A_0) \to \infty$, and thus
\[
\lambda_{d+1} \bigl(C_0^\top C_0
\bigr) \ge \lambda_{\min} \bigl(A_0^\top
A_0 \bigr) > 0
\]
for $m$ large enough.}

\begin{prop}\label{prop-mpdef}
If conditions \eqref{cond:G2}--\eqref{cond:G3}
hold true,
and conditions of either of Theorems
\ref{thm-3.1}, \ref{thm-3.2}, or \ref{thm-3.3} hold true,
then for $m$ large enough
$\langle C^\top C, \varSigma\rangle$
is a definite matrix pencil almost surely.
More specifically,
\[
\exists m_0 \; \forall m>m_0 :\; \Prob \bigl(
C^\top C + \varSigma > 0 \bigr) = 1 .
\]
\end{prop}

\begin{proof}
{\it 1.\spacefactor=3000}
If the matrix $\varSigma$ is nonsingular, then Proposition \ref{prop-mpdef}
is obvious.  Due to condition~\eqref{cond:G3},
$\rank \varSigma \ge d$ (see Remark~\ref{rem:condG3}),
whence $\varSigma \neq 0$.
In what follows, assume that $\varSigma$ is a singular but non-zero matrix.
Let $F =  (\begin{smallmatrix} F_1 \\ F_2 \end{smallmatrix} )$
be a $(n+d) \times (n+d-\rank(\varSigma))$ matrix whose columns make the
basis of the null-space $\Kernel(\varSigma) = \{x : \varSigma x = 0\}$
of the matrix $\varSigma$.

\paragraph{2}
Now prove that columns of the matrix $[I_n\; X_0]\: F$ are linearly independent.
Assume the contrary.  Then for some
$v \in \mathbb{R}^{n+d-\rank{(\varSigma)}} \setminus \{0\}$,
\begin{align}
\nonumber
[I_n\quad X_0]\: F v &= 0,
\\
\nonumber
F_1 v &= - X_0 F_2 v,
\\
\label{eq-l617} F v &= \bigl( \begin{smallmatrix} X_0
\\
-I_d \end{smallmatrix} \bigr) F_2 v= X^0_{\rm ext}
F_2 v ,
\\
\label{eq-l617-4} 0&= \varSigma F v= \varSigma X^0_{\rm ext}
\cdot F_2 v .
\end{align}

Furthermore, $F v \neq 0$ because $v \neq 0$ and the columns of $F$ are linearly independent.
Hence, by (\ref{eq-l617}), $F_2 v \neq 0$.

Equality \eqref{eq-l617-4} implies that the columns of the matrix $\varSigma \xxtrue$ are linearly dependent,
and this contradicts condition~\eqref{cond:G3}.
The contradiction means that columns of the matrix $[I\; \xxtrue]\: F$ are linearly independent.

\paragraph{3}
If the conditions of either Theorem
 \ref{thm-3.1}, \ref{thm-3.2}, or \ref{thm-3.3}
hold true, then the matrix $A_0^\top A_0$ is positive definite
for $m$ large enough.

\paragraph{4}
Under conditions \eqref{cond:G2} and \eqref{cond:G1},
$\tilde{C} F = 0$ almost surely.
Indeed, $\Me \tilde{c}_i = 0$ and $\var [\tilde {c}_i F] = F^\top \varSigma F = 0$,
$i{=}1,2,\ldots,m$.

\paragraph{5}
It remains to prove the implication:
\[
\mbox{if} \quad A_0^\top A_0^{} >
0 \quad \mbox{and} \quad \tilde{C} F = 0, \quad \mbox{then} \quad
C^\top C + \varSigma > 0.
\]
The matrices $C^\top C$ and $\varSigma$ are positive semidefinite.
Suppose that $x^\top (C^\top C + \varSigma) x = 0$ and prove that $x=0$.
Since $x^\top (C^\top C + \varSigma) x = 0$, $C x =0$ and $\varSigma x = 0$.
The vector $x$ belongs to the null-space of the matrix $\varSigma$.
Therefore, $x = F v$ for some vector $v \in \mathbb{R}^{n+d-\rank {\varSigma}}$.
Then
\begin{align}
\label{eq:0psart5} 0 = A_0^\top C x &= A_0
(C_0 + \tilde C) x
\nonumber
\\
&= A_0 C_0 F v + A_0 \tilde C F v
\nonumber
\\
&= A_0^\top A_0^{} \:
[I_n\quad X_0]\: F v + 0 .
\end{align}
As the matrix $A_0^\top A_0^{}$ is nonsingular
and columns of the matrix $[I_n\; X_0]\: F$ are linearly independent,
the columns of the matrix $A_0^\top A_0^{} \: [I_n\; X_0]\: F$
are linearly independent as well.  Hence, \eqref{eq:0psart5} implies
$v = 0$, and so
$x = F v = 0$.

We have proved that the equality $x^\top (C^\top C + \varSigma) x = 0$
implies $x = 0$.
Thus, the positive semidefinite matrix $C^\top C + \varSigma$
is nonsingular, and so positive definite.
\end{proof}

\subsubsection{Eigenvalues and common eigenvectors of
$N$ and $N^{-\frac{\scriptstyle 1}{\scriptstyle2}} C_0^\top C_0^{} N^{-\frac{\scriptstyle 1}{\scriptstyle 2}}$}
The rank-deficient positive semidefinite symmetric matrix
$C_0^\top C_0$ can be factorized as:
\begin{align*}
C_0^\top C_0^{} &= U \diag \bigl(
\lambda_{\min} \bigl(C_0^\top C_0
\bigr), \lambda_2 \bigl(C_0^\top C_0
\bigr), \ldots, \lambda_{n+d} \bigl(C_0^\top
C_0 \bigr) \bigr) U^\top
\\
&= U \diag \bigl(\lambda_j \bigl(C_0^\top
C_0 \bigr); \; j=1,\ldots,n+d \bigr) U^\top,
\end{align*}
with an orthogonal matrix $U$ and
\[
\lambda_{\min} \bigl(C_0^\top C_0
\bigr) = \lambda_{2} \bigl(C_0^\top
C_0 \bigr) = \cdots = \lambda_d \bigl(C_0^\top
C_0 \bigr) = 0 .
\]

Then the eigendecomposition of the matrix
$N = C_0^\top C_0 + \lambda_{\min} (A_0^\top A_0) I$
is
\begin{displaymath}
N = U \diag \bigl(\lambda_j \bigl(C_0^\top
C_0 \bigr) + \lambda_{\min} \bigl(A_0^\top
A_0 \bigr); \; j=1,\ldots,n+d \bigr) U^\top .
\end{displaymath}
Notice that
\begin{equation}
\lambda_{\min}(N) = \cdots = \lambda_d(N) =
\lambda_{\min} \bigl(A_0^\top A_0
\bigr). \label{eql514}
\end{equation}
The matrix $N$ is nonsingular as soon as $A_0^\top A_0$ is nonsingular.
Hence, under the conditions of Theorem \ref{thm-3.1}, \ref{thm-3.2}, or \ref{thm-3.3},
the matrix $N$ is nonsingular for $m$ large enough.

Since $C_0 \xxtrue=0$, it holds that
\begin{equation}
N \xxtrue = \lambda_{\min} \bigl(A_0^\top
A_0 \bigr) \xxtrue . \label{eql523}
\end{equation}

As soon as $N$ is nonsingular,
the matrices
$N^{-1/2}$ and $N^{-1/2} C_0^\top C_0 N^{-1/2}$
have the eigendecomposition
\begin{align*}
N^{-1/2} &= U \diag \biggl( \frac{1}{\sqrt{\lambda_j(C_0^\top C_0) \,{+}\, \lambda_{\min} (A_0^\top A_0)}} ; \; j\,{=}\,1,\ldots,n\,{+}\,d
\biggr) U^\top ,
\\
N^{-1/2} C_0^\top C_0 N^{-1/2}
&= U \diag \biggl( \frac{\lambda_j(C_0^\top C_0)}{\lambda_j(C_0^\top C_0) + \lambda_{\min} (A_0^\top A_0)} ; \; j=1,\ldots,n{+}d \biggr)
U^\top .
\end{align*}

Thus, the eigenvalues of
$N^{-1/2}$ and $N^{-1/2} C_0^\top C_0 N^{-1/2}$
satisfy the following:
\begin{align}
\big\| N^{-1/2} \big\| = \lambda_{\max} \bigl(N^{-1/2} \bigr) &=
\frac{1}   {
\sqrt{\lambda_{\min}(A_0^\top A_0)}}; \label{eq-l549}
\\
\label{eq:NCCNeig} \lambda_j \bigl(N^{-1/2}
C_0^\top C_0 N^{-1/2} \bigr) &= 0,
\quad j=1,\ldots,d;
\\
\label{neq:NCCNeig} {\textstyle\frac{1}2} \le \lambda_j
\bigl(N^{-1/2} C_0^\top C_0
N^{-1/2} \bigr) &\le 1, \quad j=d{+}1,\ldots,n{+}d .
\end{align}
As a result,
\begin{equation}
{\textstyle\frac{1}2} n \le \trace \bigl(N^{-1/2}
C_0^\top C_0^{} N^{-1/2}
\bigr) \le n. \label{eq:traceN-hC0C0N-h}
\end{equation}
Because $\trace (C_0^{} N^{-1} C_0^\top) =
\trace (C_0^{} N^{-1/2} N^{-1/2} C_0^\top) =
\trace (N^{-1/2} C_0^\top C_0^{} N^{-1/2})$,
\begin{equation}
{\textstyle\frac{1}2} n \le \trace \bigl(C_0
N^{-1} C_0^\top \bigr) \le n . \label{neq-1282}
\end{equation}
These properties 
will be
used in Sections \ref{ss:useGEPpert} and \ref{ssapx-pr3}.

\subsection{Use of eigenvector perturbation theorems}\label{ss:useGEPpert}
\subsubsection{Univariate regression ($d=1$)}
Remember inequalities \eqref{eql514} (whence \eqref{neq-l549} follows) and
\eqref{eql523}:
\begin{gather}
\hxx^\top N \hxx \ge \lambda_{\min} \bigl(A_0^\top
A_0 \bigr) \hxx^\top \hxx ; \label{neq-l549}
\\
\nonumber
N \xxtrue = \lambda_{\min} \bigl(A_0^\top
A_0 \bigr) \xxtrue .
\end{gather}
Then
\begin{align}
\nonumber
\frac{(\hxx^\top \xxtrue)^2}{\hxx^\top \hxx \cdot \xxtruetra \xxtrue} &\ge \frac{(\hxx^\top N \xxtrue)^2}   {
\hxx^\top N \hxx \cdot \xxtruetra N \xxtrue},
\\
\nonumber
\cos^2 \angle\bigl(\hxx, \xxtrue\bigr) &\ge \cos^2 \angle
\bigl(N^{1/2} \hxx, N^{1/2} \xxtrue \bigr),
\\
\label{neq:sinescal} \sin^2 \angle\bigl(\hxx, \xxtrue\bigr) &\le \sin^2
\angle \bigl(N^{1/2} \hxx, N^{1/2} \xxtrue \bigr).
\end{align}

Now, apply Lemma~\ref{lem-pertrv0} on the perturbation bound for the
minimum-eigenvalue eigenvector. The unperturbed symmetric matrix
is $N^{-1/2} C_0^\top C_0 N^{-1/2}$, satisfying
\begin{align*}
\lambda_{\min} \bigl(N^{-1/2} C_0^\top
C_0 N^{-1/2} \bigr) &= 0,
\\
N^{-1/2} C_0^\top C_0 N^{-1/2}
N^{1/2} \xxtrue &= 0,
\\
\lambda_2 \bigl(N^{-1/2} C_0^\top
C_0 N^{-1/2} \bigr) &\ge {\textstyle\frac{1}2} .
\end{align*}
The null-vector of the unperturbed matrix is $N^{-1/2} \xxtrue$.\vadjust{\goodbreak}

The column vector $\hxx$ is a generalized eigenvector of the matrix
pencil $\langle C^\top C, \break \varSigma\rangle$.
Denote the corresponding eigenvalue by $\lambda_{\min}$. Thus,
\[
C^\top C \hxx = \lambda_{\min} \cdot \varSigma \hxx .
\]
The perturbed matrix is $N^{-1/2} (C^\top C - m \varSigma) N^{-1/2}$;
the minimum eigenvalue of the matrix pencil
$\langle N^{-1/2} (C^\top C - m \varSigma) N^{-1/2},\allowbreak \; N^{-1/2} \varSigma N^{-1/2}\rangle$
is equal to
$\lambda_{\min} - m$,
and the eigenvector is $N^{1/2} \hxx$:
\[
N^{-1/2} \bigl(C^\top C - m \varSigma \bigr) N^{-1/2}
N^{1/2} \hxx = (\lambda_{\min} - m) N^{-1/2} \varSigma
N^{-1/2} N^{1/2} \hxx .
\]

We have to verify that $N^{-1/2} \varSigma N^{-1/2} N^{1/2} \xxtrue \neq 0$;
this follows from condition~\eqref{cond:G3}.
Obviously, the matrix $N^{-1/2} \varSigma N^{-1/2}$ is positive semidefinite:
\begin{equation}
\label{neq:NSNpsd} N^{-1/2} \varSigma N^{-1/2} \ge 0.
\end{equation}

Denote
\[
\epsilon = \big\| N^{-1/2} \bigl(C^\top C - m \varSigma \bigr)
N^{-1/2} - N^{-1/2} C_0^\top
C_0 N^{-1/2} \big\| .
\]

By Lemma \ref{lem-pertrv0}
\[
\sin^2 \angle \bigl(N^{1/2} \hxx, N^{1/2} \xxtrue
\bigr) \le \frac{\epsilon}{0.5} \biggl(1 + \frac{\xxtruetra N \xxtrue}{\xxtruetra \varSigma \xxtrue} \cdot
\frac{\hxx^\top \varSigma \hxx}{\hxx^\top N \hxx} \biggr).
\]
Use (\ref{eql523}) and (\ref{neq-l549}) again, and also use \eqref{neq:sinescal}:
\begin{align}
\sin^2 \angle\bigl(\hxx, \xxtrue\bigr) & \le \sin^2 \angle
\bigl(N^{1/2} \hxx, N^{1/2} \xxtrue \bigr)
\nonumber
\\
& \le 2 \epsilon \biggl(1 + \frac{\xxtruetra \xxtrue}{\xxtruetra \varSigma \xxtrue} \cdot \frac{\hxx^\top \varSigma \hxx}{\hxx^\top \hxx} \biggr)
\nonumber
\\
& \le 2 \epsilon \biggl(1 + \frac{\xxtruetra \xxtrue \cdot \|\varSigma\|}               {
\xxtruetra \varSigma \xxtrue} \biggr) . \label{neq:maxsinbouni2}
\end{align}

\subsubsection{Multivariate regression ($d\ge 1$)}
What follows is valid for both univariate ($d = 1$) and multivariate ($d > 1$) regression.

Due to \eqref{eql514}, $N \ge \lambda_{\min}(A_0^\top A_0) I$
in the Loewner order; thus inequality (\ref{neq-l549})
holds in the Loewner order. Hence
\begin{align*}
\forall v\in\mathbb{R}^d \setminus \{0\} : \, &\frac{v^\top \hxx^\top \xxtrue (\xxtruetra \xxtrue)^{-1} \xxtruetra
       \hxx v}      {
v^\top \hxx^\top \hxx v}
\\
&\quad\ge \lambda_{\min} \bigl(A_0^\top A_0
\bigr) \frac{v^\top \hxx^\top \xxtrue (\xxtruetra \xxtrue)^{-1} \xxtruetra
       \hxx v}      {
v^\top \hxx^\top N \hxx v}.
\end{align*}
With inequality (\ref{eql523}), we get
\begin{align*}
\nonumber
&\frac{v^\top \hxx^\top \xxtrue (\xxtruetra \xxtrue)^{-1} \xxtruetra
       \hxx v}      {
v^\top \hxx^\top \hxx v}
\\
&\quad\ge \frac{v^\top N \hxx^\top \xxtrue (\xxtruetra N \xxtrue)^{-1} \xxtruetra
       N \hxx v}      {
v^\top \hxx^\top N \hxx v}.
\end{align*}
Using equation \eqref{eq:1-sin2} to determine the sine
and noticing that
\begin{align*}
P_{\xxtrue} &= \xxtrue \bigl(\xxtruetra \xxtrue\bigr)^{-1} \xxtruetra,
\\
P_{N^{1/2} \xxtrue}& = N^{1/2} \xxtrue \bigl(\xxtruetra N \xxtrue\bigr)^{-1}
\xxtruetra N^{1/2},
\end{align*}
we get
\begin{gather}
\nonumber
1 - \big\|\sin\angle \bigl(\hxx, \xxtrue\bigr)\big\|^2 \ge 1 - \big\|\sin\angle
\bigl(N^{1/2}\hxx, N^{1/2}\xxtrue \bigr)\big\|^2,
\\
\label{neq:maxsinescal} \big\|\sin\angle \bigl(\hxx, \xxtrue\bigr)\big\| \le \big\|\sin\angle
\bigl(N^{1/2}\hxx, N^{1/2}\xxtrue \bigr)\big\|.
\end{gather}

The TLS estimator $\hxx$ is defined as a solution
to the linear equations \eqref{eqTLSX124} for $\Delta$
that brings the minimum to \eqref{eqTLS118}.
By Proposition~\ref{prop:gs6.6}, the same $\Delta$
brings the minimum to \eqref{eqTLS220}.
By Proposition~\ref{prop-6eximl},
the functions \eqref{eqf-l816} and \eqref{eqf-l816mm}
attain their minima at the point $\hxx$.
Therefore, the minimum of the function
\begin{equation}\label{eq:2300}
M \mapsto \lambda_{\max} \bigl( \bigl(M^\top N^{-1/2}
\varSigma N^{-1/2} M \bigr)^{-1} M^\top
N^{-1/2} \bigl(C^\top C - m \varSigma \bigr) N^{-1/2} M
\bigr)
\end{equation}
is attained for $M = N^{1/2} \hxx$.

{\sloppy
Now, apply Lemma~\ref{lem-pertrVV0} on perturbation bounds
for a generalized invariant subspace.
The unperturbed matrix (denoted $A$ in
Lemma~\ref{lem-pertrVV0}) is $N^{-1/2} C_0^\top C_0 N^{-1/2}$; its
nullspace is the column space of the matrix $N^{1/2} \xxtrue$
(which is denoted $X_0$ in Lemma~\ref{lem-pertrVV0}). The
perturbed matrix ($A + \tilde A$ in Lemma~\ref{lem-pertrVV0}) is
$N^{-1/2} (C^\top C - m \varSigma) N^{-1/2}$.
The matrix $B$ in Lemma~\ref{lem-pertrVV0} equals
$N^{-1/2} \varSigma N^{-1/2}$.
The norm of the perturbation is
denoted $\epsilon$  (it is $\|\tilde A\|$  in Lemma~\ref{lem-pertrVV0}).
The $(n+d)\times d$ matrix which brings the minimum to \eqref{eq:2300} is $N^{1/2} \hxx$.
\querymark{Q7}The other conditions of Lemma~\ref{lem-pertrVV0}
are \eqref{eq:NCCNeig}, \eqref{neq:NCCNeig}, and \eqref{neq:NSNpsd}.
We have
\begin{align*}
&\big\|\sin\angle \bigl(N^{1/2} \hxx,N^{1/2}\xxtrue \bigr)
\big\|^2
\\
&\quad\le \frac{\epsilon}{0{.}5} \bigl(1+ \big\| N^{-1/2} \varSigma N^{-1/2}
\big\|\, \lambda_{\max} \bigl(\bigl(\xxtruetra \varSigma \xxtrue\bigr)^{-1}
\xxtruetra N \xxtrue \bigr) \bigr) .
\end{align*}}\relax

Again, with \eqref{neq:maxsinescal}, \eqref{eql523} and \eqref{eq-l549},
we have
\begin{align}
\nonumber
& \big\|\sin\angle\bigl(\hxx, \xxtrue\bigr)\big\|^2 \\
&\quad\le \big\|\sin\angle
\bigl(N^{1/2} \hxx,N^{1/2}\xxtrue \bigr)\big\|^2\notag
\\
\nonumber
&\quad \le 2\epsilon \biggl( 1 + \frac{\|\varSigma\|}{\lambda_{\min}(A_0^\top A_0)}\, \lambda_{\max}
\bigl( \lambda_{\min} \bigl(A_0^\top
A_0 \bigr) \bigl(\xxtruetra \varSigma \xxtrue\bigr)^{-1} \xxtruetra
\xxtrue \bigr) \biggr)
\nonumber
\\
&\quad = 2\epsilon \bigl( 1 + \|\varSigma\|\, \lambda_{\max} \bigl(\bigl(
\xxtruetra \varSigma \xxtrue\bigr)^{-1} \xxtruetra \xxtrue \bigr) \bigr) .
\label{neq:maxsinbomul2}
\end{align}

\subsection{Proof of the convergence $\epsilon\to0$}
\label{ssapx-pr3}
In this section, we prove the convergences
\begin{align*}
M_1 &= N^{-1/2} C_0^\top \widetilde C
N^{-1/2} \to 0,
\\
M_2 &= N^{-1/2} \bigl(\widetilde C^\top \widetilde
C - m \varSigma \bigr) N^{-1/2} \to 0
\end{align*}
in probability for Theorem \ref{thm-3.1}, and almost surely for
Theorems \ref{thm-3.2} and \ref{thm-3.3}.
As $\epsilon = \|M_1^{} + M_1^\top + M_2\|$, the convergences $M_1 \to 0$
and $M_2 \to 0$ imply $\epsilon \to 0$.

\begin{proof}[End of the proof of Theorem \ref{thm-3.1}.]
It holds that
\begin{align*}
\| M_1 \|_F^2 &= \big\| N^{-1/2}
C_0^\top \tilde C N^{-1/2} \big\|_F^2
= \trace \bigl( N^{-1/2} C_0^\top \tilde C
N^{-1} C_0 \tilde C^\top N^{-1/2} \bigr)
\\
&= \trace \bigl(C_0^{} N^{-1}
C_0^\top \tilde C N^{-1} \tilde C^\top
\bigr) = \sum_{i=1}^m \sum
_{j=1}^m c^0_{i}
N^{-1} \bigl(c^0_{j} \bigr)^\top
\tilde c_j N^{-1} \tilde c_i^\top .
\end{align*}
The right-hand side can be simplified since $\Me \tilde c_j N^{-1} \tilde c_i^\top = 0$ for $i\neq j$ and\break
$\Me \tilde c_i N^{-1} \tilde c_i^\top = \trace (\varSigma N^{-1})$:
\[
\Me \|M_1\|_F^2 = \sum
_{i=1}^m c_{0i} N^{-1}
c_{0i}^\top \trace \bigl(\varSigma N^{-1} \bigr) =
\trace \bigl(C_0 N^{-1} C_0^\top
\bigr) \trace \bigl(\varSigma N^{-1} \bigr) .
\]
The first multiplier in the right-hand side is bounded due to \eqref{neq-1282}
as $\trace (C_0 N^{-1} C_0^\top) \le n$, for $m$ large enough.
Now, construct an upper bound for the second multiplier:
{\sloppy
\begin{align*}
\trace \bigl(\varSigma N^{-1} \bigr) &= \big\| N^{-1/2}
\varSigma^{1/2} \big\|_F^2 \le \big\|N^{-1/2}
\big\|^2 \big\|\varSigma^{1/2}\big\|_F^2 =
\lambda_{\max} \bigl(N^{-1} \bigr) \trace \varSigma
\\
&= \frac{\trace \varSigma}{\lambda_{\min}(N)} = \frac{\trace \varSigma}{\lambda_{\min}(A_0^\top A_0^{})}.
\end{align*}}%

Finally,
\[
\Me \|M_1\|_F^2 \le \frac {n \trace \varSigma}{
\lambda_{\min} (A_0^\top A_0)}.
\]

The conditions of Theorem~\ref{thm-3.1} imply
that $\lambda_{\max} (A_0^\top A_0) \to \infty$;
therefore, $M_1 \probto 0$ as $m\to\infty$.
{\sloppy\par}

Now, we prove that $M_2 \probto 0$ as $m{\to}\infty$.
We have
\begin{align}
\nonumber
M_2 &= N^{-1/2} \bigl(\tilde C^\top
\tilde C - m \varSigma \bigr) N^{-1/2} ,
\\
\label{neq-1308} \| M_2 \| &\le \big\|N^{-1/2}\big\| \, \big\|\tilde
C^\top \tilde C - m \varSigma\big\| \, \big\|N^{-1/2}\big\| =
\frac{ \llVert  \sum_{i=1}^m ( \tilde c_i^\top \tilde c_i^{} - \varSigma)  \rrVert }     {
\lambda_{\min} (A_0^\top A_0^{})} .
\end{align}
Now apply the Rosenthal inequality (case $1 \le \nu \le 2$;
Theorem \ref{thm-RosenNeq-12}) to construct a bound for $\Me\|M_2\|^r$:
\[
\Me \|M_2\|^r \le \frac{{\rm const}
\sum_{i=1}^m \Me \| \tilde c_i^\top  \tilde c_i^{} - \varSigma\|^r}{
\lambda_{\min}^r (A_0^\top A_0^{})} .
\]
By the conditions of Theorem~\ref{thm-3.1}, the sequence
$\{ \Me \|\tilde c_i^\top \tilde c_i^{} - \varSigma\|^r, \allowbreak
\ i=1,2,\ldots\}$ is bounded.  Hence
\begin{align*}
\Me \|M_2\|^r &\le \frac{O(m)}{
\lambda_{\min}^r (A_0^\top A_0^{})} \quad\mbox {as}\ m
\to\infty,
\\
\Me \|M_2\|^r &\to 0 \quad\mbox{and}\quad M_2
\probto 0 \quad\mbox {as}\ m\to\infty. \qedhere
\end{align*}
\end{proof}

\begin{proof}[End of the proof of Theorem \ref{thm-3.2}.]
\[
M_1 = \sum_{i=1}^m
N^{-1/2} c_{0i}^\top \tilde c_i
N^{-1/2} .
\]
By the Rosenthal inequality (case $\nu \ge 2$; Theorem~\ref{thm-RosenNeq-2+})
\begin{align*}
\Me \|M_1\|^{2r} & \le {\rm const} \sum
_{i=1}^m \Me \bigl\llVert N^{-1/2}
c_{0i}^\top \tilde c_i N^{-1/2} \bigr
\rrVert ^{2r} + {}
\\
& \quad + {\rm const} \Biggl( \sum_{i=1}^m
\Me \big\| N^{-1/2} c_{0i}^\top \tilde c_i
N^{-1/2} \big\|^2 \Biggr)^{r} .
\end{align*}
Construct an upper bound for the first summand:
\begin{align*}
\sum_{i=1}^m \Me \bigl\llVert
N^{-1/2} c_{0i}^\top \tilde c_i
N^{-1/2} \bigr\rrVert ^{2r} &\le \sum
_{i=1}^m \big\| N^{-1/2} c_{0i}^\top
\big\|^{2r} \max\limits
_{i=1,\ldots,m} \Me \| \tilde c_i
\|^{2r} \big\| N^{-1/2} \big\|^{2r} ,\\
\sum_{i=1}^m \big\| N^{-1/2}
c_{0i}^\top \big\|^{2r} &\le \Biggl( \sum
_{i=1}^m \big\| N^{-1/2} c_{0i}^\top
\big\|^2 \Biggr)^{r}
\\
&= \Biggl( \sum_{i=1}^m c_{0i}
N^{-1} c_{0i}^\top \Biggr)^{r} = \bigl(
\trace \bigl(C_0 N^{-1} C_0^\top
\bigr) \bigr)^{r} \le n^{r}
\end{align*}
by inequality (\ref{neq-1282}).
By the conditions of Theorem~\ref{thm-3.2}, the sequence
$\{\max\limits_{i=1,\ldots,m}
\Me \| \tilde c_i \|^{2r}, \allowbreak\ m=1,2,\ldots\}$
is bounded.  Remember that $\|N^{-1/2}\| = \lambda_{\min}^{-1/2}
(A_0^\top A_0)$. Thus,{\sloppy
\[
\sum_{i=1}^m \Me \bigl\llVert
N^{-1/2} c_{0i}^\top \tilde c_i
N^{-1/2} \bigr\rrVert ^{2r} = \frac{O(1)}{\lambda_{\min}^r (A_0^\top A_0)} \quad\mbox{as
$m{\to}\infty$}.
\]}\relax
The asymptotic relation
\[
\sum_{i=1}^m \Me \bigl\llVert
N^{-1/2} c_{0i}^\top \tilde c_i
N^{-1/2} \bigr\rrVert ^{2} = \frac{O(1)}{\lambda_{\min} (A_0^\top A_0)}
\]
can be proved similarly; in order to prove it,
we use boundedness of the sequence
$\{\max\limits_{i=1,\ldots,m}
\Me \| \tilde c_i \|^{2}, \allowbreak\ m=1,2,\ldots\}$.
Finally,
\[
\Me \|M_1\|^{2r} = \frac{O(1)}{\lambda_{\min}^r (A_0^\top A_0^{})} \quad \mbox{as $m
\to\infty$.}
\]

The conditions of Theorem \ref{thm-3.2} imply that
$\sum_{m=m_0}^\infty  \Me \| M_1 \|^{2r} < \infty$,
whence $M_1 \to 0$ as $m\to\infty$, almost surely.

Now, prove that $M_2\to 0$ almost surely.
In order to construct a bound for $\Me\|M_2\|^r$,
use the Rosenthal inequality (case $\nu\ge2$; Theorem~\ref{thm-RosenNeq-2+})
as well as (\ref{neq-1308}):
\begin{align*}
\Me \|M_2\|^r &\le \frac{\Me \llVert  \sum_{i=1}^m
(c_i^\top \tilde c_i^{} - \varSigma)  \rrVert ^r}{
\lambda_{\min}^r (A_0^\top A_0^{})}
\\
&\le \frac{{\rm const}
\sum_{i=1}^m \Me \|\tilde c_i^\top \tilde c_i^{} - \varSigma\|^r}{
\lambda_{\min}^r (A_0^\top A_0^{})} + \frac{{\rm const}
 (\sum_{i=1}^m \Me\|\tilde c_i^\top \tilde c_i^{} - \varSigma\|^2
 )^{r/2}}{
\lambda_{\min}^r (A_0^\top A_0^{})} .
\end{align*}
Under the conditions of Theorem \ref{thm-3.2}, the sequences
$\{ \Me\|\tilde c_i^\top \tilde c_i^{} - \varSigma\|^r,\allowbreak
\ i=1,2,\ldots\}$
and
$\{ \Me\|\tilde c_i^\top \tilde c_i^{} - \varSigma\|^2,\allowbreak
\ i=1,2,\ldots\}$
are bounded.  Thus,
\begin{align*}
\Me \|M_2\|^r &= \frac{O(m^{r/2})}{
\lambda_{\min}^r (A_0^\top A_0^{})} \quad\mbox{as $m\to
\infty$;}
\\
\sum_{m=m_0}^\infty \Me \| M_2
\|^r &< \infty,
\end{align*}
whence $M_2 \to 0$ as $m\to\infty$, almost surely.
\end{proof}

\begin{proof}[End of the proof of Theorem \ref{thm-3.3}.]
The proof of the asymptotic relation
\[
\Me \|M_1\|^{2r} = \frac{O(1)}{\lambda_{\min}^r (A_0^\top A_0^{})} \quad\mbox{as $m\to
\infty$}
\]
from Theorem~\ref{thm-3.2} is still valid.
The almost sure convergence $M_1 \to 0$ as $m\to\infty$
is proved in the same way as in Theorem \ref{thm-3.2}.

Now, show that $M_2 \to 0$ as $m\to\infty$, almost surely.
Under the condition of Theorem \ref{thm-3.3},
\begin{gather*}
\Me\big\|\tilde c_m^\top \tilde c_m^{}
- \varSigma\big\|^r = O(1), \qquad \sum_{m=m_0}^\infty
\frac{\Me \|\tilde c_m^\top \tilde c_m^{} - \varSigma\|^r}{
\lambda_{\rm min}^r (A_0^\top A_0^{})} < \infty,
\end{gather*}
and $\Me \tilde c_i^\top \tilde c_i^{} - \varSigma = 0$.
The sequence of nonnegative numbers
$\{\lambda_{\min} (A_0^\top A_0),
\allowbreak\ m=1,2,\ldots\}$
never decreases and tends to $+\infty$.
Then, by the Law of large numbers in
\cite[Theorem 6.6, page 209]{Petrov1995}
\[
\frac{1}{\lambda_{\rm min} (A_0^\top A_0^{})} \sum_{i=1}^m \bigl(
\tilde c_i^\top \tilde c_i - \varSigma \bigr)
\to 0 \quad \mbox{as~$m\to \infty$,} \quad \mbox{a.s.,}
\]
whence, with (\ref{neq-1308}),
\begin{align*}
\|M_2\| &\le \frac{ \llVert  \sum_{i=1}^m
(\tilde c_i^\top \tilde c_i - \varSigma) \rrVert }{
\lambda_{\min} (A_0^\top A_0^{})} \to 0 \quad \mbox{as $m\to\infty$,
\quad a.s.;}
\\
M_2 &\to 0 \quad \mbox{as $m\to\infty$}, \quad \mbox{a.s.} \qedhere
\end{align*}
\end{proof}

\subsection{Proof of the uniqueness theorems}
\begin{proof}[Proof of Theorem~\ref{thm:wdue}.]
The random events \ref{wdue1}, \ref{wdue15} and \ref{wdue2}
are defined in the statement of this theorem on page~\pageref{thm:wdue}.
The random event \ref{wdue1} always occurs.
This was proved in Section~\ref{place:defhxx}
where the estimator $\hxx$ is defined.
In order to prove the rest, we first construct
the random event \eqref{eq:wXnear}, which occurs
either with high probability or eventually.
Then we prove that, whenever \eqref{eq:wXnear} occurs,
there is the existence and ``more than uniqueness''
in the random event \ref{wdue2}, and then
prove that the random event \ref{wdue15} occurs.

Now, we construct a modified version $\hxxsv$
of the estimator $\hxx$ in the following way.
If there exist such  solutions $(\Delta, \hxx)$
to \eqref{eqTLS118} \&  \eqref{eqTLSX124}
that\vadjust{\goodbreak} $\|\sin\angle(\hxx,\xxtrue)\|\ge (1+\|X_0\|^2)^{-1/2}$,
let $\hxxsv$ come from one of such solutions.
Otherwise, if for every solution $(\Delta, \hxx)$
to \eqref{eqTLS118} \&  \eqref{eqTLSX124}
$\|\sin\angle(\hxx,\xxtrue)\| < (1+\|X_0\|^2)^{-1/2}$,
let $\hxxsv$ come from one of these solutions.
In any case, let us construct $\hxxsv$ in such a way
that it is a random matrix.
It is possible; that follows from \cite{Pfanzagl1969}.

Thus we construct a matrix $\hxxsv$ such that:
\begin{enumerate}
\item $\hxxsv$ is a $(d+n)\times n$ random matrix;
\item for some
$\Delta \in \mathbb{R}^{m \times (d+n)}$,
   $(\Delta,  \hxxsv)$ is a solution to \eqref{eqTLS118} \&  \eqref{eqTLSX124};
\item
        if $\|\sin\angle(\hxxsv,\xxtrue)\| < (1+\|X_0\|^2)^{-1/2}$,
        then $\|\sin\angle(\hxx,\xxtrue)\| < (1+\|X_0\|^2)^{-1/2}$ for any
        solution $(\Delta,  \hxx)$ to \eqref{eqTLS118} \&  \eqref{eqTLSX124}.
\end{enumerate}

From the proof of Theorem~\ref{thm-3.1} it follows that
$\|\sin\angle(\hxxsv,\xxtrue)\| \to 0$ in probability as $m\to\infty$.
From the proof of Theorem~\ref{thm-3.2} or \ref{thm-3.3}
it follows that $\|\sin\angle(\hxxsv,\break\xxtrue)\| \to 0$ almost surely.
Then
\begin{equation}
\label{eq:wXnear} \big\| \sin\angle\bigl(\hxxsv, \xxtrue\bigr)\big\| < \frac{1}{\sqrt{1 + \|X_0\|^2}}
\end{equation}
either with high probability or almost surely.

Whenever the random event \eqref{eq:wXnear} occurs,
for any solution $\Delta$ to \eqref{eqTLS118}
and the corresponding full-rank solution $\hxx$ to \eqref{eqTLSX124}
(which always exists)
it holds that
$\|\sin\angle(\hxx,\break\xxtrue)\| < (1+\|X_0\|^2)^{-1/2}$,
whence, due to Theorem~\ref{thm-neqsindist},
the bottom $d\times d$ block of the matrix $\hxx$ is nonsingular.
Right-multiplying $\hxx$ by a nonsingular matrix, we can transform it
into a form $ (\begin{smallmatrix} \widehat X \\ -I \end{smallmatrix} )$.
The constructed matrix $\widehat X$ is a solution to equation~\eqref{eqTLSX126}
for given $\Delta$.
Thus, we have just proved that if the random event \eqref{eq:wXnear} occurs, then
for any $\Delta$ which is a solution to \eqref{eqTLS118},
equation \eqref{eqTLSX126} has a solution.

Now, prove the uniqueness of $\widehat X$.
Let $(\Delta_1, \widehat X_1)$ and $(\Delta_2, \widehat X_2)$
be two solutions to \eqref{eqTLS118} \& \eqref{eqTLSX126}.
Show that $\widehat X_1 = \widehat X_2$.
(If we \querymark{Q8}can for $\Delta_1 = \Delta_2$,
then the random event~\ref{wdue2} occurs.)
Denote
$\hxxno1 =  ( \begin{smallmatrix} \widehat X_1 \\ -I \end{smallmatrix}  )$
and
$\hxxno2 =  ( \begin{smallmatrix} \widehat X_2 \\ -I \end{smallmatrix}  )$.
By Proposition~\ref{prop:gep6.7},
$\colspana<\hxxno1> \subset \colspana<u_k, \allowbreak \; \nu_k \le d>$
and
$\colspana<\hxxno2> \subset \colspana<u_k, \allowbreak \; \nu_k \le d>$,
where $\nu_k$ and $u_k$ are generalized eigenvalues (arranged in ascending order)
and respective eigenvectors of the matrix pencil $\langle X^\top X, \, \varSigma\rangle$.

Assume by contradiction that $\widehat X_1 \neq \widehat X_2$.
Then $\rank [\hxxno1,\; \hxxno2] \ge d+1$,
where $[\hxxno1,\; \hxxno2]$ is an  $(n+d) \times 2d$ matrix constructed of $\hxxno1$ and $\hxxno2$.
Then
\[
d^* = \rank\langle u_k, \;\allowbreak \nu_k \le d\rangle
\ge \rank \begin{bmatrix} \hxxno1, &\hxxno2\end{bmatrix} \ge d+1
\]
(which means $\nu_d = \nu_{d+1}$).
Then $d_* - 1 < d < d^*$, where
$d_* - 1 = \dim\colspana<u_k, \allowbreak \; \nu_k < d>$,
$d = \dim\colspana<\xxtrue>$
and
$d^* = \dim\colspana<u_k, \allowbreak \; \nu_k \le d>$\allowbreak{}
(notation $d_*$ and $d^*$ comes from the proof of
Proposition~\ref{prop:gep6.7}).
By Lemma~\ref{lem:subsset1112132},
there exists a $d$-dimensional subspace $V_{12}$ for which
$\colspana<u_k, \, \nu_k < d> \subset V_{12} \subset\colspana<u_k, \, \nu_k \le d>$
and $\|\sin\angle(V_{12}, \xxtrue)\| = 1$.
Bind a basis of the $d$-dimensional subspace $V_{12} \subset \mathbb{R}^{(n+d)}$
into the $(n+d)\times d$ matrix $\hxxno3$, so $\colspana<\hxxno3> = V_{12}$.
Again, by\vadjust{\goodbreak} Proposition~\ref{prop:gep6.7} for some matrix $\Delta$,
$(\Delta, \hxxno3)$ is a solution to \eqref{eqTLS118} \& \eqref{eqTLSX126}.
Then\vadjust{\goodbreak} $\|\sin\angle(\hxxno3,\break \xxtrue)\| = 1 \ge (1 + \|X_0\|^2)^{-1/2}$.
Then $\|\sin\angle(\hxxsv, \xxtrue)\| \ge (1 + \|X_0\|^2)^{-1/2}$,
which contradicts \eqref{eq:wXnear}.
Thus, the random event~\ref{wdue2} occurs.

Now prove that the random event~\ref{wdue15} occurs.
Let $\Delta_1$ and $\Delta_2$ be two solutions to
the optimization problem~\eqref{eqTLS118}.
Whenever the random event~\eqref{eq:wXnear} occurs,
the respective solutions $\widehat X_1$ and $\widehat X_2$
to equation~\eqref{eqTLSX126} exist.
By already proved uniqueness, they are equal, i.e.,
$\widehat X_1 = \widehat X_2$.
Then both $\Delta_1$ and $\Delta_2$ are
solutions to the optimization problem
\begin{equation}
\label{eq:cmpot1} \begin{cases}
\| \Delta \, \pinvp(\varSigma^{1/2}) \|_F\to\min; \\
\Delta \, (I-P_\varSigma) = 0; \\
(C - \Delta) \hxxno1 = 0
\end{cases}
\end{equation}
for the fixed
$\hxxno1 =  ( \begin{smallmatrix} \widehat X_1 \\ -I \end{smallmatrix}  ) =
 ( \begin{smallmatrix} \widehat X_2 \\ -I \end{smallmatrix}  )$.
By Proposition~\ref{prop-l588} and Remark~\ref{rem:6.2-1},
the least element in the optimization problem~\eqref{TLS-fixX568}
for $X = \hxxno1$ is attained for the unique matrix
$\Delta = C \hxxno1 \pinvp(\hxxnoT1 \varSigma \hxxno1) \hxxnoT1 \varSigma$.
Since it is attained, it is also attained for
both $\Delta_1$ and $\Delta_2$.  Hence,
$\Delta_1 = \Delta_2
$.
Thus, the random event \ref{wdue15} occurs.

We proved that the random event \ref{wdue1} always occurs,
and the random events \ref{wdue15} and \ref{wdue2} occur whenever \eqref{eq:wXnear} occurs,
which occurs either with high probability or eventually as desired.
\end{proof}

\begin{rem}
This uniqueness of the solution $\Delta$ to the optimization problem \eqref{eqTLS118}
agrees with the uniqueness result in \cite{GOLUB1987317}.
The solution is unique if $\nu_d < \nu_{d+1}$.
\end{rem}

\begin{proof}[Proof of Theorem~\ref{thm:spnv}]
\ref{thm-spnv:part1}.
In Theorem~\ref{thm:wdue}, the event \ref{wdue1} occurs always, not just with high probability or eventually.
The solution $\Delta$ to \eqref{eqTLS118} exists and also solves \eqref{eqTLS220} due to
Proposition~\ref{prop:gs6.6}.  Thus, the first sentence of Theorem~\ref{thm:spnv} is true.
The second sentence of Theorem~\ref{thm:spnv} has been already proved, since the constraints
in the optimization problems \eqref{eqTLS118} and \eqref{eqTLS220} are the same.

\ref{thm-spnv:part2} \& \ref{thm-spnv:part3}.
The proof of consistency of the estimator defined with \eqref{eqTLS220} \& \eqref{eqTLSX126}
and of the existence of the solution is similar to the proof
for the estimator
defined with \eqref{eqTLS118} \& \eqref{eqTLSX126}
in Theorems \ref{thm-3.1}--\ref{thm-3.3} and \ref{thm:wdue}.
The only difference is skipping the use of Proposition~\ref{prop:gs6.6}.
Notice that we do not prove the uniqueness of the solution because we cannot use Proposition~\ref{prop:gep6.7}.
\end{proof}

\paragraph{To Remark~\ref{rem:spnvuni}}
The amended Theorem~\ref{thm:spnv}
can be proved similarly.
In the proof of part~\ref{thm-spnv:part1},
read ``The solution $\Delta$ to \eqref{eqTLS118} $\ldots$  solves \eqref{eqTLS118uin} due to
Proposition~\ref{prop:gs6.6uin}.''
In the proof of parts \ref{thm-spnv:part2} and
\ref{thm-spnv:part3}, read
``The only difference is using Proposition~\ref{prop:gs6.6uin}, part~\ref{prop:gs6.6uin:part2}
instead of Proposition~\ref{prop:gs6.6}.''

\subsection*{Proofs of auxiliary results}
\subsection{Proof of lemmas on perturbation bounds for invariant subspaces}
\begin{proof}[Proof of Lemma~\ref{lem-pertrv0} and Remark~\ref{rem:lem-pertrv0}.]
For the proof of Lemma~\ref{lem-pertrv0} itself, see\break parts 2 and 3 of the proof below.
For the proof of Remark~\ref{rem:lem-pertrv0}, see parts 2, 3 and 4 below.
Part 1 is a mere discussion of why the conditions of Remark~\ref{rem:lem-pertrv0}
are more general than ones of Lemma~\ref{lem-pertrv0}.

In the proof, we assume that $\{ x : x^\top B x > 0\}$ is the domain
of the function $f(x)$.  The assumption affects the definition of $\lim_{x \to x_*} f(x)$,
and $\inf f$ is the infimum of $f(x)$ \textit{over the domain}.

\paragraph{1} At first, clarify the conditions of Remark~\ref{rem:lem-pertrv0}.
As it is, the existence of a point $x$ such that
\begin{equation}
\label{eq-infmin_altx0} \liminf\limits
_{\vec t\to x}f(\vec t) = \inf\limits
_{\vec t^\top\! B \vec t > 0} f(\vec t)
\end{equation}
is assumed in Remark~\ref{rem:lem-pertrv0}.
Now, prove that, under the preceding condition of Remark~\ref{rem:lem-pertrv0},
there exists a vector $x \neq 0$ that satisfies \eqref{eq-infmin_altx0}.

The function $f(x)$ is homogeneous of degree 0, i.e.,
\[
f(k x) = f(x)\quad \text{if $k\in \mathbb{R} \setminus \{0\}$ and
$x^\top B x>0$}.
\]
Hence, all values which are attained by $f(x)$ on its domain
$\{x :\allowbreak  x^\top B x>0\}$,
are also attained on the bounded set
$\{x :\allowbreak \|x\|{=}1,\allowbreak\, x^\top B x>0\}$:
\[
f \bigl( \bigl\{x : \|x\|{=}1, \, x^\top B x>0 \bigr\} \bigr) = f
\bigl( \bigl\{x : x^\top B x>0 \bigr\} \bigr) .
\]
Then
\[
\inf_{\substack{\|x\|=1\\x^\top B x >0}} f(x) = \inf_{x^\top B x>0} f(x) .
\]

Let $\setA$ be a closure of $\{x :\allowbreak \|x\|{=}1,\allowbreak\, x^\top B x>0\}$.
There is a sequence $\{x_k, \break k = 1, 2, \ldots\}$ such that
$\|x_k\|{=}1$ and  $x_k^\top B x_k>0$ for all $k$, and
$\lim_{k\to\infty} f(x_k) = \inf_{x^\top B x>0} f(x)$.
Since $\setA$ is a compact set, there exists $x_* \in \setA$ which is
a limit of some subsequence $\{x_{k_i}, \: i = 1, 2, \ldots\}$ of $\{x_k, \: k = 1, 2, \ldots\}$.
Then either
\begin{equation}
\label{eq:li-subseq1} \liminf_{x \to x_*} f(x) \le \inf
_{x^\top B x>0} f(x)
\end{equation}
or, if $x_{k_i} = x_*$ for $i$ large enough,
\begin{equation}
\label{eq:li-subseq3} f(x_*) \le \inf_{x^\top B x>0} f(x) .
\end{equation}
(In equations \eqref{eq:li-subseq1} and \eqref{eq:li-subseq3},
we assume that
$\{x :\allowbreak x^\top B x>0\}$
is a domain of $f(x)$,
so \eqref{eq:li-subseq3} implies $x_*^\top B x_*^{} > 0$.)
Again, due to the homogeneity, $\liminf\limits_{x \to x_*} f(x) \le f(x_*)$ if $f(x_*)$ makes sense.
Hence \eqref{eq:li-subseq1} follows from \eqref{eq:li-subseq3} and thus holds true either way.

Taking the limit in the relation $f(x) \ge \inf f$, we obtain the opposite inequality
\begin{equation*}
\liminf_{x \to x_*} f(x) \ge \inf_{x^\top B x>0} f(x).
\end{equation*}
Thus, the equality \eqref{eq-infmin} holds true for some $x_* \in \setA$.
Note that $\|x_*\| = 1$, so $x_* \neq 0$.

\paragraph{2}
Prove that under the conditions of Lemma~\ref{lem-pertrv0} or Remark~\ref{rem:lem-pertrv0}
\begin{equation*}
\lleft[ \begin{array}{ll}
\displaystyle \mbox{either} & f(x_*) \le f(x) \\
\displaystyle \mbox{or}     & x_* ^\top (A + \tilde A) x_* \le 0.
\end{array} \rright.
\end{equation*}

Because the matrix $B$ is symmetric and positive semidefinite,
$x^\top B x = 0$ if and only if $B x = 0$,
and
$x^\top B x > 0$ if and only if $B x \neq 0 $.
As $B x_0 \neq 0$, $x_0^\top B x_0 > 0$
and the function $f(x)$ is well-defined at $x_0$.

Under the conditions of Lemma~\ref{lem-pertrv0}
the function $f(x)$ is well-defined at $x_0$
and attains its minimum at $x_*$, so $f(x_*) \le f(x_0)$.

Under the conditions of Remark~\ref{rem:lem-pertrv0} we consider 3 cases
concerning the value of $x_*^\top B x_*$.

\subparagraph{Case 1}  $x_*^\top B x_* < 0$.
But on the domain of $f(x)$ the inequality $x^\top B x > 0$ holds true.
Since $x_*$ is a limit point of the domain of $f(x)$,
the inequality $x_*^\top B x_* \ge 0$ holds true,
and Case 1 is impossible.

\subparagraph{Case 2}  $x_*^\top B x_* = 0$.
Prove that $x_* ^\top (A + \tilde A) x_* \le 0$.
On the contrary, let $x_* ^\top (A + \tilde A) x_* > 0$.
Remember once again that $x^\top B x > 0$ on the domain of $f(x)$.
Then
\[
\lim_{x \to x_*} f(x) = \lim_{x \to x_*}
\frac{x^\top (A + \tilde A) x}{x^\top B x} = +\infty,
\]
which cannot be $\inf f(x)$.
The contradiction obtained implies that $x_* ^\top (A + \tilde A) x_* \le 0$.

\subparagraph{Case 3}  $x_*^\top B x_* > 0$.
Then the function $f(x)$ is well-defined at $x_*$, and
\[
f(x_*) = \lim_{x \to x_*} f(x) = \inf f(x) \le f(x_0) .
\]
So, $f(x_*)  \le f(x_0)$ in Case 3.

\paragraph{3} Proof of Lemma~\ref{lem-pertrv0} and
proof of Remark~\ref{rem:lem-pertrv0}
when $f(x_*) \le f(x_*)$.
Then
\begin{displaymath}
\frac{x^\top (A + \tilde A) x}{x^\top B x} \le \frac{x_0^\top (A + \tilde A) x_0}{x_0^\top B x_0}\,.
\end{displaymath}
As $A x_0=0$,
\begin{equation*}
x^\top A x \le - x^\top \tilde A x + \frac{x_0^\top \tilde A x_0 \, x^\top B x}{x_0^\top B x_0} \le
\|\tilde A\| \biggl( \|x\|^2 + \frac{\|x_0\|^2 x^\top B x}{x_0^\top B x_0} \biggr) .
\end{equation*}

With use of eigendecomposition of $A$, the inequality
$x^\top A x \ge \lambda_2(A) \, \|x\|^2\times \sin^2\angle(x,x_0)$
can be proved.
Hence the desired inequality follows:
\[
\lambda_2(A) \sin^2\angle(x,x_0) \le \|
\tilde A\| \biggl( 1 + \frac{\|x_0\|^2}{x_0^\top B x_0^{}} \cdot \frac{x^\top B x}{\|x\|^2} \biggr).
\]

\paragraph{4} Proof of  Remark~\ref{rem:lem-pertrv0}
when $x_* ^\top (A + \tilde A) x_* \le 0$.
Then
\begin{align*}
x^\top A x &\le - x^\top \tilde A x,
\\
\lambda_2(A) \|x\|^2 \sin^2
\angle(x,x_0) &\le \|\tilde A\|\,\|x\|^2,
\\
\lambda_2(A) \sin^2\angle(x,x_0) &\le \|
\tilde A\|,
\end{align*}
whence the desired inequality follows.
\end{proof}

\begin{notation}
If $A$ and $B$ are symmetric matrices of the same size,
and furthermore the matrix $B$ is positive definite, denote
\[
\max\frac{A}{B} = \lambda_{\max} \bigl(B^{-1} A \bigr)
.
\]
The notation is used in the proof of Lemma~\ref{lem-pertrVV0}.
\end{notation}

\begin{lem}\label{lem:bofosin2}
Let $1 \le d_1 \le n$, $0 \le d_2 \le n$.
Let $X \in \mathbb{R}^{n \times d_1}$ be a matrix of full rank,
and $V$ be a $d_2$-dimensional subspace in $\mathbb{R}^n$.
Then
\begin{align*}
\max \frac{X^\top (I - P_V) X}{X^\top X} &= \big\|\sin \angle(X, V)\big\|^2 \quad \mbox{if}
\quad d_1 \le d_2,
\\
\max \frac{X^\top (I - P_V) X}{X^\top X} &= 1 \quad \mbox{if} \quad d_1 >
d_2.
\end{align*}
\end{lem}
\begin{proof}
Using the min-max theorem, the relation $\colspana<X> = \colspan\langle P_{\colspana<X>} \rangle$
and simple properties of orthogonal projectors,
construct the inequality
\begin{align*}
&\max\frac{X^\top (I-P_V) X}{X^\top X} \\
&\quad= \max_{v \in \mathbb{R}^{d_1}\setminus \{0\}} \frac{v^\top X^\top (I - P_V) X v} {v^\top X^\top  X v}
\\
 &\quad= \max_{w \in \colspana<X> \setminus \{0\}} \frac{w^\top (I - P_V) w} {w^\top   w} = \max
_{v \in \mathbb{R}^n \setminus \{0\}} \frac{v^\top P_{\colspana<X>} (I - P_V) P_{\colspana<X>} v} {v^\top P_{\colspana<X>} P_{\colspana<X>}  v}
\\
&\quad\ge \max_{v \in \mathbb{R}^n \setminus \{0\}} \frac{v^\top P_{\colspana<X>} (I - P_V) P_{\colspana<X>} v} {v^\top   v} = \lambda_{\max}
\bigl(P_{\colspana<X>} (I - P_V) P_{\colspana<X>} \bigr)
\\
&\quad= \lambda_{\max} \bigl(P_{\colspana<X>} (I - P_V) (I -
P_V) P_{\colspana<X>} \bigr) = \big\| P_{\colspana<X>} (I -
P_V) \big\|^2 .
\end{align*}
On the other hand,
\begin{align*}
\max_{w \in \colspana<X> \setminus \{0\}} \frac{w^\top (I - P_V) w} {w^\top   w} &= \max_{w \in \colspana<X> \setminus \{0\}}
\frac{w^\top P_{\colspana<X>} (I - P_V) P_{\colspana<X>} w} {w^\top   w}
\\
&\le \max_{v \in \mathbb{R}^n \setminus \{0\}} \frac{v^\top P_{\colspana<X>} (I - P_V) P_{\colspana<X>} v} {v^\top   v}.
\end{align*}
Thus,
\[
\max\frac{X^\top (I-P_V) X}{X^\top X} = \big\| P_{\colspana<X>} (I - P_V)
\big\|^2.
\]
If $d_1 \le d_2$, then $\| P_{\colspana<X>} (I - P_V) \| = \|\sin\angle(X, V)\|$ due to
\eqref{eq:defmaxsin}.
Otherwise, if $d_1 > d_2$, then
\[
\dim\colspana<X> + \dim V^\bot = \rank X + n - \dim V =
d_1 + n - d_2 > n.
\]
Hence the subspaces $\colspana<X>$ and $V^\bot$ have nontrivial intersection,
i.e., there exists $w\neq 0$, $w \in \colspana<X> \cap V^\bot$.
Then $P_{\colspana<X>} (I - P_V) w = w$, whence $\| P_{\colspana<X>} (I - P_V) \| \ge 1$.
On the other hand, $\| P_{\colspana<X>} (I - P_V) \| \le \| P_{\colspana<X>} \| \times \allowbreak \| (I - P_V) \| \le 1$.
Thus, $\| P_{\colspana<X>} (I - P_V) \| = 1$. This completes the proof.
\end{proof}

\begin{proof}[Proof of Lemma~\ref{lem-pertrVV0}.]
The matrix $B$ is positive semidefinite,
the matrix $X_0^\top B X_0$ is positive definite,
and the matrix $X_0$ is of full rank $d$
(hence, $n \ge d$).
The matrix $A$ satisfies inequality
$A \ge \lambda_{d+1}(A) (I - P_{\colspana <X_0>})$
in the Loewner order.

Let $X$ be a point where the functional $f(x)$ defined in  \eqref{eqf-lempertrVV0} attains its minimum.
Since $X_0^\top B X_0$ is positive definite, $f(X_0)$ makes sense.
Thus, $f(X) \le f(X_0)$,
\[
\max\frac{X^\top (A + \tilde A) X}{X^\top B X} \le \max\frac{X_0^\top (A + \tilde A) X_0} {X_0^\top B X_0^{}} .
\]
Using the relations
\begin{align*}
X^\top \tilde A X &\ge - \|\tilde A\|\, X^\top X, \qquad
X_0^\top \tilde A X_0 \le \|\tilde A\|\,
X_0^\top X_0,
\\
X^\top B X &\le \|B\|\, X^\top X, \qquad A X_0= 0,
\end{align*}
we have
\begin{align}
\max\frac{X^\top A X - \|\tilde A\| X^\top X}{\|B\|\, X^\top X} &\le \max\frac{\|\tilde A\|\, X_0^\top X_0^{}}{X_0^\top B X_0^{}},
\nonumber
\\
\frac{1}{\|B\|} \cdot \biggl( \max\frac{X^\top A X}{X^\top X} - \|\tilde A\| \biggr)
&\le \|\tilde A\| \max\frac{X_0^\top X_0^{}}{X_0^\top B X_0^{}} \label{neq-l1360}.
\end{align}

Since $A \ge \lambda_{d+1}(A) (I - P_{\colspana <X_0>})$,
by Lemma~\ref{lem:bofosin2}
\[
\lambda_{d+1} (A) \, \big\|\sin\angle(X,X_0)\big\|^2
\le \lambda_{d+1} (A) \max\frac{X^\top (I - P_{\colspana <X_0>})}{X^\top X} \le \max
\frac{X^\top A X}{X^\top X}.
\]
Then the desired inequality follows from (\ref{neq-l1360}):
\begin{equation*}
\big\|\sin\angle(X,X_0)\big\|^2 \le \frac{\|\tilde A\|}{\lambda_{d+1}(A)} \biggl(1
+ \|B\| \max\frac{X_0^\top X_0^{}} {X_0^\top B X_0^{}} \biggr). \qedhere
\end{equation*}
\end{proof}

\subsection{Comparison of $\|\sin\angle(\hxx,\xxtrue)\|$ and
$\|\widehat X - \xtrue\|$}
In the next theorem and in its proof, matrices $A$, $B$
and $\varSigma$ have different meaning than elsewhere in the paper.
\begin{thm}\label{thm-neqsindist}
Let $ (\begin{smallmatrix} A \\ B \end{smallmatrix} )$
and $ (\begin{smallmatrix} X_0 \\ -I \end{smallmatrix} )$
be full-rank $(n + d) \times d$ matrices.
If
\begin{equation}
\lleft\llVert \sin\angle \lleft( \lleft(\begin{matrix} A \\ B \end{matrix}
\rright),\: \lleft(\begin{matrix} X_0 \\ -I \end{matrix} \rright) \rright)
\rright\rrVert < \frac {1}{\sqrt{1 + \|X_0\|^2}}, \label{neq-cond-1700}
\end{equation}
then:
\begin{enumerate}
\item[1\char41] the matrix $B$ is nonsingular;
\item[2\char41] $\| A B^{-1} + X_0 \| \le
\frac{ (1+\|X_0\|^2)\:
(\|X_0\| s^2 + s \sqrt{1-s^2})}{
 1 - (1+\|X_0\|^2)\: s^2 }$ with $s =
 \llVert  \sin\angle  (
 (\begin{smallmatrix} A \\ B \end{smallmatrix} ),\:
 (\begin{smallmatrix} X_0 \\ -I \end{smallmatrix} )
 )  \rrVert $.
\end{enumerate}
\end{thm}

\begin{proof}
{\it 1.} Split the matrix
$P^\bot_\smallxxmat$,
which is an orthogonal projector along the column space
of the matrix $ (\begin{smallmatrix} X_0 \\ -I \end{smallmatrix} )$,
into four blocks:
\[
I - P_\smallxxmat = P^\bot_\smallxxmat = \begin{pmatrix} {\bf P}_1 & {\bf P}_2 \\
{\bf P}_2^\top & {\bf P}_4
\end{pmatrix}
.
\]
Up to the end of the proof, ${\bf P}_1$ means
the upper-left $n\times n$ block of the
$(n+p)\times (n+p)$ matrix $P^\bot_\smallxxmat$.
Prove that
$\lambda_{\min} ({\bf P}_1) = \frac{1}{1+\|X_0\|^2}$.

Let $X_0 = U \varSigma V^\top$ be a singular value decomposition of the matrix $X_0$
(here $\varSigma$ is a diagonal $n\times d$ matrix,
$U$ and $V$ are orthogonal matrices). Then
\begin{align*}
P^\bot_\smallxxmat &= I - \begin{pmatrix} X_0 \\ -I \end{pmatrix} \lleft(
\begin{pmatrix} X_0 \\ -I \end{pmatrix}^{\!\top} \begin{pmatrix} X_0 \\ -I \end{pmatrix} \rright)^{\!-1} \begin{pmatrix} X_0 \\ -I \end{pmatrix}
\\
&= \begin{pmatrix}
U (I - \varSigma (\varSigma^\top \varSigma + 1)^{-1} \varSigma^\top) U^\top &
U \varSigma (\varSigma^\top \varSigma + I)^{-1}  V^\top \\
V (\varSigma^\top \varSigma + I)^{-1} \varSigma^\top U^\top &
V (I - (\varSigma^\top \varSigma + I)^{-1}) V^\top
\end{pmatrix} .
\end{align*}

The $n\times n$ matrix $I - \varSigma (\varSigma^\top \varSigma + I)^{-1} \varSigma^\top$ is diagonal;
its diagonal entries are $\frac{1}{1+\sigma_i^2(X_0)}$,\,
$i=1,\ldots,n$, where
\mycenter{
\begin{tabular}{l}
$\sigma_i(X_0)$ is the $i$-th singular value of $X_0$ if
$1 \le i \le \min(n,d)$,\\
$\sigma_i(X_0) = 0$ if $\min(n,d) < i \le n$.
\end{tabular}}
Those diagonal entries comprise all the eigenvalues of ${\bf P}_1$;
\[
\lambda_{\min}({\bf P}_1) = \frac{1}{1 + \sigma_{\rm max}^2(\|X_0\|)} =
\frac{1}{1 + \|X_0\|^2} .
\]

\paragraph{2}
Due to equation \eqref{eq:defmaxsin}, the square of the largest of
\querymark{Q9}sines of canonical eigenvalues between the subspaces
$V_1$ and $V_2$ is equal to
\[
\big\| \sin \angle(V_1, V_2) \big\|^2 = \max
_{v \in V_1\setminus \{0\}} \frac{v^\top P^\bot_{V_2} v}{\|v\|^2} .
\]

Hence for $v \in V_1$, $v \neq 0$,
\begin{equation}
\big\| \sin \angle(V_1, V_2) \big\|^2 \ge
\frac{v^\top P^\bot_{V_2} v}{\|v\|^2} . \label{neq-ts-1777}
\end{equation}

\paragraph{3}
Prove the first statement of Theorem~\ref{thm-neqsindist}
by contradiction.
Suppose that the matrix $B$ is singular.
Then there exist $f \in \mathbb{R}^d \setminus \{0\}$ and $u = A f \in \mathbb{R}^n$
such that $B f = 0$ and
\[
\begin{pmatrix} u \\ 0_{d \times 1} \end{pmatrix} = \begin{pmatrix} A f \\ B f \end{pmatrix} \in V_1 ,
\]
where $V_1 \subset \mathbb{R}^{n+d}$ is the column space of
the matrix $ (\begin{smallmatrix} A \\ B \end{smallmatrix} )$.
As the columns of the matrix
$ (\begin{smallmatrix} A \\ B \end{smallmatrix} )$  are linearly independent,
$ (\begin{smallmatrix} u \\ 0 \end{smallmatrix} ) \neq 0$.
Then, by \eqref{neq-ts-1777},
\begin{align*}
\lleft\llVert \sin\angle \lleft( \lleft(\begin{matrix} A \\ B \end{matrix}
\rright),\: \lleft(\begin{matrix} X_0 \\ -I \end{matrix} \rright) \rright)
\rright\rrVert ^2 &\ge \frac{
\lleft(\begin{matrix} u \\ 0\end{matrix} \rright)^\top
P^\bot_\smallxxmat
\lleft(\begin{matrix} u \\ 0\end{matrix} \rright)
}{
 \llVert  (\begin{smallmatrix} u \\ 0\end{smallmatrix} ) \rrVert ^2
} = \frac{u^\top {\bf P}_1 u} {\|u\|^2}
\ge
\\
&\ge \lambda_{\min}({\bf P}_1) = \frac{1}{1 + \|X_0\|^2} ,
\end{align*}
which contradicts condition (\ref{neq-cond-1700}).

\paragraph{4} Prove inequality (\ref{neq-1870}).
(Later on we will show that the second statement of Theorem~\ref{thm-neqsindist}
follows from (\ref{neq-1870})).
There exists such a vector $f \in \mathbb{R}^d \setminus \{0\}$
that $\|(A B^{-1} + X_0)\: f\| = \|A B^{-1} + X_0\|\: \|f\|$.
Denote
\begin{align*}
u &= \bigl(A B^{-1} + X_0 \bigr) f,
\\
z &= \begin{pmatrix} A \\ B \end{pmatrix} B^{-1} f = \begin{pmatrix} A B^{-1} f \\ f \end{pmatrix} =
\begin{pmatrix} u \\ 0 \end{pmatrix} - \begin{pmatrix} X_0 \\ -I \end{pmatrix} f \in V_1 .
\end{align*}
Since
$(X_0^\top, -I)
P^\bot_\smallxxmat = 0$ and
$P^\bot_\smallxxmat
 (\begin{smallmatrix} X_0 \\ -I \end{smallmatrix}  ) = 0$,
\begin{align*}
z^\top P^\bot_\smallxxmat z &= \lleft(
\begin{pmatrix} u \\ 0 \end{pmatrix} - \begin{pmatrix} X_0 \\ -I \end{pmatrix} f \rright)^{\!\top}
P^\bot_\smallxxmat \lleft( \begin{pmatrix} u \\ 0 \end{pmatrix} -
\begin{pmatrix} X_0 \\ -I \end{pmatrix} f \rright)
\\
&= \begin{pmatrix} u \\ 0 \end{pmatrix}^{\!\top} P^\bot_\smallxxmat
\begin{pmatrix} u \\ 0 \end{pmatrix} = u^\top {\bf P}_1 u
\\
&\ge \|u\|^2 \lambda_{\min} ({\bf P}_1) =
\frac{\|A B^{-1} + X_0\|^2\:\|f\|^2} {1 + \|X_0\|^2} .
\end{align*}

Notice that $z \neq 0$ because $B^{-1} f \neq 0$ and the columns of the matrix
$ (\begin{smallmatrix}A \\ B\end{smallmatrix} )$
are linearly independent. Thus,
\[
0 < \|z\|^2 = \big\|A B^{-1} f\big\|^2 +
\big\|f^2\big\| \le \bigl(1 + \big\|A B^{-1}\big\|^2 \bigr) \,
\|f\|^2 .
\]
By (\ref{neq-ts-1777}),
\begin{align}
\nonumber
\left\llVert \sin\angle\lleft( \begin{pmatrix} A \\ B \end{pmatrix},\,
\begin{pmatrix} X_0 \\ -I \end{pmatrix} \rright) \right\rrVert ^2 &\ge
\frac{z^\top P^\bot_{ (\begin{smallmatrix}
X_0 \\ -I \end{smallmatrix} )} z}{\|z\|^2} \ge \frac{\|A B^{-1} + X_0\|^2}{
(1 + \|X_0\|^2)\, (1 + \|A B^{-1}\|^2)},
\\
\label{neq-1870} \left\llVert \sin\angle\lleft( \begin{pmatrix} A \\ B \end{pmatrix},\,
\begin{pmatrix} X_0 \\ -I \end{pmatrix} \rright) \right\rrVert &\ge \frac{\|A B^{-1} + X_0\|}{
\sqrt{1 + \|X_0\|^2}\,\sqrt{1 + (\|X_0\| + \|A B^{-1} + X_0\|)^2}} .
\end{align}

\paragraph{5} Prove that the second statement of Theorem~\ref{thm-neqsindist}
follows from (\ref{neq-1870}).
The function
\begin{equation}
s(\delta) := \frac{\delta}{
\sqrt{1 + \|X_0\|^2}\,\sqrt{1 + (\|X_0\| + \delta)^2}} \label{form-1881}
\end{equation}
is strictly increasing on $[0,{+}\infty
)$, with $s(0)=0$  and $\lim_{\delta \to +\infty} s(\delta) = \frac{1}{\sqrt{1+\|X_0\|^2}}$.
Therefore, inequality (\ref{neq-1870}) implies the implication:
\begin{align*}
\mbox{if {}}\big\|A B^{-1} + X_0\big\| &> \delta,
\\
\mbox{then {}} \left\llVert \sin\angle\lleft( \begin{pmatrix} A   \\  B \end{pmatrix},\,
\begin{pmatrix} X_0 \\ -I \end{pmatrix} \rright) \right\rrVert &> \frac{\delta}{
\sqrt{1 + \|X_0\|^2}\,\sqrt{1 + (\|X_0\| + \delta)^2}} \rlap{.}
\end{align*}
The equivalent contrapositive implication is as follows:
\begin{align}
\mbox{if {}} \left\llVert \sin\angle\lleft( \begin{pmatrix} A \\ B \end{pmatrix},\,
\begin{pmatrix} X_0 \\ -I \end{pmatrix} \rright) \right\rrVert &\le \frac{\delta}{
\sqrt{1 + \|X_0\|^2}\,\sqrt{1 + (\|X_0\| + \delta)^2}},\notag
\\
\mbox{then {}}\big\|A B^{-1} + X_0\big\| &\le \delta \rlap{.}
\label{neq1905}
\end{align}
The inverse function to $s(\delta)$ in (\ref{form-1881}) is
\[
\delta(s) := \frac{(1 + \|X_0\|^2) \:
(s^2 \,\|X_0\| + s \sqrt{1-s^2})}{
1 - (1 + \|X_0\|^2) s^2}.
\]
Substitute
$\delta = \delta (  \llVert  \sin \angle  (
 (\begin{smallmatrix} A \\ B \end{smallmatrix}
\vphantom{smallxxmat} ),
 (\begin{smallmatrix} X_0 \\ -I \end{smallmatrix} )
 )  \rrVert   )$
into \eqref{neq1905} and obtain the following statement:
\begin{align*}
\mbox{if {}} \left\llVert \sin\angle\lleft( \begin{pmatrix} A \\ B \end{pmatrix}\,
\begin{pmatrix} X_0 \\ -I \end{pmatrix} \rright) \right\rrVert &\le \left\llVert \sin\angle
\lleft( \begin{pmatrix} A \\ B \end{pmatrix}, \begin{pmatrix} X_0 \\ -I \end{pmatrix} \rright) \right
\rrVert ,
\\
\mbox{then {}}\big\|A B^{-1} + X_0\big\| &\le \delta \bigl( \bigl
\llVert \sin \angle \bigl( (\begin{smallmatrix} A_{}
\\
B \end{smallmatrix} \vphantom{smallxxmat} ), (\begin{smallmatrix} X_0
\\
-I \end{smallmatrix} ) \bigr) \bigr\rrVert \bigr) \rlap{,}
\end{align*}
whence the second statement of Theorem~\ref{thm-neqsindist}
follows.

In part 5 of the proof,  condition (\ref{neq-cond-1700}) is used twice.
First, it is one of conditions of the first statement of the theorem:
without it, the matrix $B$ might be singular.
Second, the function $\delta(s)$ is defined on interval $ [0,
\frac{1}{\sqrt{1+\|X_0\|^2}} )$.
\end{proof}

\begin{cora}
Let
$ (\begin{smallmatrix} X_0 \\ -I \end{smallmatrix} )$
be an $(n+d)\times d$ matrix,
and let $ \{ (\begin{smallmatrix}A_m \\ B_m\end{smallmatrix} ),
\allowbreak\hspace{1em}
\vphantom{\textstyle\begin{smallmatrix}A_m \\ B_m\end{smallmatrix}}
m=1,2,\ldots \}$ be a sequence of
$(n+d)\times d$ matrices of rank $d$.  If
$ \llVert \sin\angle (
 (\begin{smallmatrix} A_m \\ B_m \end{smallmatrix} ),\,
 (\begin{smallmatrix} X_0 \\ -I \end{smallmatrix} )
 ) \rrVert  \to 0$ as $m\to\infty$, then:{
\begin{enumerate}
\item[1\char41] the matrix $B_m$ is nonsingular for $m$ large enough,
\item[2\char41] $-A_m^{} B_m^{-1} \to X_0$ as $m{\to}\infty$.
\end{enumerate}}%
\end{cora}

\subsection{Generalized eigenvalue problem for positive semidefinite matrices:
proofs}\label{ssec:GEP:Proofs}
\begin{proof}[Proof of Lemma~\ref{lem:psdgsvd1}.]
For fixed $i$, split the matrix $T$ in two blocks.
Let $T = [T_{i1}, T_{i2}]$, where
$T_{i1}$ is the matrix constructed of the first $i$ columns of $T$, and
$T_{i2}$ is the matrix constructed of the last $n-i+1$ columns of
$T$.
Denote $V_1$ and $V_2$ the column spaces of the matrices $T_{i1}$ and $T_{i2}$,
respectively. Then $\dim V_1 = i$ and $\dim V_2 = n-i+1$.

\paragraph{1}
\textit{The proof of the fact that\querymark{Q10} $\nu_i \in \{ \lambda\ge 0 \mathrel|
\text{``}\exists V, \ \dim V = i :
(A - \lambda B)|_V \le 0 \text{''}\}$ if $\nu_i < \infty$}.
In other words, if $\nu_i < \infty$, then relations
\begin{equation}
\label{eq:1018} \lambda \ge 0, \qquad \dim(V) = i, \qquad (A - \lambda
B)|_V \le 0
\end{equation}
hold true for $\lambda=\nu_i$ and $V = V_1$.

If $v \in V_1$, then $v = T_{i1} x$ for some
$x \in \mathbb{R}^i$.
Hence
\begin{align*}
v^\top (A - \nu_i B) v &= x^\top
T_{i1}^\top (A - \nu_i B) T_{i1}^{}
x
\\
&= x^\top \diag(\lambda_1 {-} \nu_i
\mu_1,\, \ldots,\, \lambda_i {-} \nu_i
\mu_1) x = \sum_{j=1}^i
x_j^2 (\lambda_j - \nu_i
\mu_j).
\end{align*}
The inequality $\lambda_j - \nu_i \mu_j \le 0$
holds true for all $j$ such
that either $\lambda_j = \mu_j = 0$ or $\lambda_j / \mu_j \le
\nu_i$; particularly, it holds true for $j = 1, \ldots, i$. Hence
$v^\top (A - \nu_i B) v \le 0$.

\paragraph{2}
\textit{The proof of the fact that $\nu_i$ is a lower bound of the set
$ \{ \lambda\ge 0 \mathrel|
\text{``}\exists V, \ \dim V = i :
(A - \lambda B)|_V \le 0 \text{''}\}$}.
In other words, if there exists a subspace $V\subset\mathbb{R}^n$
such that the relations \eqref{eq:1018} hold true,
then $\nu_i \le \lambda$.

By contradiction, suppose that $\dim V = i$, $(A - \lambda B)|_V \le 0$, $\nu_i > \lambda \ge 0$.
Then $\nu_i > 0$.

Now prove that $(A - \lambda B)|_{V_2} > 0$.
If  $v \in V_2 \setminus\{0\}$, then $v = T_{i2} x$ for some
$x \in \mathbb{R}^{n-i+1} \setminus\{0\}$.
Then
\[
v^\top (A - \lambda B) v = \sum_{j=i}^n
x_{j+1-i}^2 (\lambda_j - \lambda
\mu_j).
\]
For $j \ge i$, due to the inequality $\nu_j \ge \nu_i > 0$
and the conditions of the lemma, the case $\lambda_j=0$
is impossible; thus $\lambda_j>0$.
Prove the inequality  $\lambda_j - \lambda \mu_j > 0$.
If $\mu_j > 0$, then
$\lambda_j - \lambda \mu_j = (\nu_j - \lambda) \mu_j$.
Since $\nu_j \ge \nu_i > \lambda$, the first factor
$\nu_i - \lambda$ is a positive number. Hence,
$\lambda_j - \lambda \mu_j > 0$.
Otherwise, if $\mu_j = 0$, then
$\lambda_j - \lambda \mu_j = \lambda_j > 0$.
Thus the inequality $\lambda_j - \lambda \mu_j  > 0$
holds true in both cases.
Hence $v^\top (A - \lambda B) v > 0$.
Since this holds for all $v \in V_2 \setminus\{0\}$,
the restriction of the quadratic form $A - \lambda B$
onto the linear subspace $V_2$ is positive definite.

On the one hand, since
$(A - \lambda B)|_V \le 0$ and $(A - \lambda B)|_{V_2} > 0$,
the subspaces $V$ and $V_2$ have a trivial intersection.
On the other hand, since $\dim V + \dim V_2 = n+1 > n$,
the subspaces $V$ and $V_2$ cannot have a trivial intersection.
We got a contradiction.

Hence $\nu_i \le \lambda$, and $\nu_i$ is a lower bound of
$ \{ \lambda\ge 0 \mathrel|
\text{``}\exists V, \ \dim V = i :
(A - \lambda B)|_V \le 0 \text{''}\}$.
That completes the proof of Lemma~\ref{lem:psdgsvd1}.
\end{proof}

Remember that $\pinvb(M)$ is the Moore--Penrose pseudoinverse matrix to $M$;
$\colspana<M>$ is the column span of the matrix $M$.
If matrices $M$ and $N$ are compatible for multiplication,
then $\colspana<M N> \subset \colspana<M>$.
(Furthermore,  $\colspana<M_1> \subset \colspana<M_2>$
if and only if $M_1 = M_2 N$ for some matrix $N$).
Hence, $\colspana<M M^\top> = \colspana<M>$ (to prove it,
we can use the identity $M = M M^\top \pinvp(M^\top)$).
{\sloppy\par}

Since the $n\times n$ covariance matrix $\varSigma$ is positive
semidefinite, for every $k \times n$ matrix $M$ the equality
$\colspana<M \varSigma M^\top> = \colspana<M \varSigma>$ holds true.
This can be proved with use of the matrix square root.

If what follows, for a fixed $(n+d)\times d$ matrix $X$ denote
\[
\Deltapm = C X \bigl(X^\top \varSigma X \bigr)^{\dagger}
X^\top \varSigma ,
\]
where $C$ is an $m \times (n+d)$ matrix,
$\varSigma$ is an $n \times n$ positive semidefinite matrix.

\begin{proof}[Proof of Proposition~\ref{prop-l588}.]
{\it 1, necessity. Relation \eqref{eq:nemset} is a necessary condition for compatibility
of the constraints in \eqref{TLS-fixX568}.}
Let $\Delta \, (I - P_{\varSigma}) = 0$ and $(C-\Delta) X = 0$
for some $m\times(n+d)$ matrix $\Delta$.
Due to $\Delta \, (I - P_{\varSigma}) = 0$,
$\Delta = M \varSigma$ for some matrix $M$.
Then $C X = \Delta X = M \varSigma X$, $X^\top C^\top = X^\top \varSigma M^\top$,
whence $\colspan(X^\top C^\top) \subset \colspan(X^\top \varSigma)$.

\paragraph{1, sufficiency. Relation \eqref{eq:nemset} is a sufficient condition for compatibility
of the constraints in \eqref{TLS-fixX568}}
Let $\colspan(X^\top C^\top) \subset \colspan(X^\top \varSigma)$.
Then $X^\top C^\top = X^\top \varSigma M$
for some matrix $M$.
The constraints $\Delta \, (I - P_{\varSigma}) = 0$, $(C-\Delta) X = 0$ are satisfied
for $\Delta = M^\top \varSigma$, so they are compatible.

\paragraph{2a, eqns.~\eqref{eq:prop-l588-2a1}. If the constraints are compatible, they are satisfied
for {$\Delta = \Deltapm$}} Indeed,
\[
\Deltapm \, (I - P_\varSigma) = C X \bigl(X^\top \varSigma X
\bigr)^{\dagger} X^\top \varSigma\, (I - P_\varSigma) = 0,
\]
since $\varSigma \, (I - P_\varSigma) = 0$.
If the constraints are compatible, then
\[
\colspan \bigl(X^\top \varSigma X \bigr) = \colspan
\bigl(X^\top \varSigma \bigr) \subset \colspan \bigl(X^\top
C^\top \bigr),
\]
whence
\begin{align*}
X^\top \varSigma X \bigl(X^\top \varSigma X
\bigr)^{\dagger} X^\top C^\top &= P_{X^\top \varSigma X}
X^\top C^\top = X^\top C^\top,
\\
\Deltapm X &= C X \bigl(X^\top \varSigma X \bigr)^{\dagger}
X^\top \varSigma X = C X,
\\
(C-\Deltapm) X &= 0.
\end{align*}

\paragraph{2a, eqn.~\eqref{eq:prop-l588-2a2} and 2b.
If the constraints are compatible,
then the constrained least element of {$\Delta \pinvb(\varSigma)
\Delta^\top$} is attained for $\Delta =
\Deltapm$} \textit{The least element is equal to\break $C X
\pinvp(X^\top \varSigma X) X^\top C^\top$.}
Let $\Delta$ satisfy the constraints,
which imply $\Delta P_\varSigma = \Delta$ and $\Delta
X = C X$. Expand the product
\begin{align}
(\Delta - \Deltapm) \pinvb(\varSigma) (\Delta - \Deltapm)^\top
= \Delta \pinvb(\varSigma) \Delta^\top - \Deltapm \pinvb(\varSigma)
\Delta^\top - \Delta \pinvb(\varSigma) \Deltapm^\top +
\Deltapm \pinvb(\varSigma) \Deltapm^\top. \label{eq:momo}
\end{align}
Simplify the expressions for three (of four) summands:
\begin{align*}
\Delta \pinvb(\varSigma) \Deltapm^\top &= \Delta \pinvb(\varSigma)
\varSigma X \bigl(X^\top \varSigma X \bigr)^{\dagger}
X^\top C^\top
\\
&= \Delta P_\varSigma X \bigl(X^\top \varSigma X
\bigr)^{\dagger} X^\top C^\top
\\
&= \Delta X \bigl(X^\top \varSigma X \bigr)^{\dagger}
X^\top C^\top = C X \bigl(X^\top \varSigma X
\bigr)^{\dagger} X^\top C^\top.
\end{align*}
Applying matrix transposition to both sides of the last chain of equalities,
we get
\begin{equation*}
\Deltapm \pinvb(\varSigma) \Delta^\top = C X \bigl(X^\top
\varSigma X \bigr)^{\dagger} X^\top C^\top.
\end{equation*}
For the last summand,
\begin{align*}
\Deltapm \pinvb(\varSigma) \Deltapm^\top &= C X \bigl(X^\top
\varSigma X \bigr)^{\dagger} X^\top \varSigma \pinvb(\varSigma)
\varSigma X \bigl(X^\top \varSigma X \bigr)^{\dagger}
X^\top C^\top
\\
&= C X \bigl(X^\top \varSigma X \bigr)^{\dagger} X^\top
\varSigma X \bigl(X^\top \varSigma X \bigr)^{\dagger}
X^\top C^\top
\\
&= C X \bigl(X^\top \varSigma X \bigr)^{\dagger} X^\top
C^\top.
\end{align*}
Thus,
\eqref{eq:momo}
implies that
\begin{equation}
\Delta \pinvb(\varSigma) \Delta^\top = (\Delta - \Deltapm) \pinvb(
\varSigma) (\Delta - \Deltapm)^\top + C X \bigl(X^\top
\varSigma X \bigr)^{\dagger} X^\top C^\top .
\label{eq:momoshort1}
\end{equation}
Hence
\[
\Delta \pinvb(\varSigma) \Delta^\top \ge C X \bigl(X^\top
\varSigma X \bigr)^{\dagger} X^\top C^\top ,
\]
and statement 2b of the theorem is proved.
For $\Delta = \Deltapm$, equality is attained,
which coincides with \eqref{eq:prop-l588-2a2}.

\paragraph{Remark~\ref{rem:6.2-1}. The least point is attained for a unique $\Delta$}
It is enough to show
that if $\Delta$ satisfies the constraints and
$\Delta \pinvb(\varSigma) \Delta^\top = C X \pinvp(X^\top \varSigma X) X^\top C^\top$,
then $\Delta = \Deltapm$.

Indeed, if $\Delta$ satisfies the constraints
$\Delta \, (I - P_{\varSigma}) = 0$ and $(C-\Delta) X = 0$,
and
$\Delta \pinvb(\varSigma) \Delta^\top = C X \pinvp(X^\top \varSigma X) X^\top C^\top$, then due to \eqref{eq:momoshort1}
\[
(\Delta - \Deltapm) \pinvb(\varSigma) (\Delta - \Deltapm)^\top = 0 .
\]
As $\pinvb(\varSigma)$ is a positive semidefinite matrix,
$(\Delta - \Deltapm) \pinvb(\varSigma) = 0$ and
$(\Delta - \Deltapm) P_\varSigma = (\Delta - \Deltapm) \pinvb(\varSigma) \varSigma = 0$.
Add the equality $\Delta \, (I - P_\varSigma)=0$ (which is one of the constraints)
and subtract the equality $\Deltapm \, (I - P_\varSigma)=0$ (which is one
of equalities \eqref{eq:prop-l588-2a1} and holds true due part 2a
of the theorem).  Obtain
\begin{equation*}
\Delta - \Deltapm = (\Delta - \Deltapm) P_\varSigma + \Delta \, (I -
P_\varSigma) - \Deltapm \, (I - P_\varSigma) = 0 ,
\end{equation*}
whence $\Delta = \Deltapm$.
\end{proof}

\begin{proof}[Proof of Proposition~\ref{prop:simpifdef}.]
{\it 1. Necessity.}
Since the matrices
$C^\top C$ and $\varSigma$ are positive semidefinite,
the matrix pencil
$\langle C^\top C, \varSigma\rangle$
is definite if and only if the matrix
$C^\top C + \varSigma$ is positive semidefinite.
Thus, if
the matrix pencil
$\langle C^\top C, \varSigma\rangle$
is definite, then
the matrix $C^\top C + \varSigma$ is positive definite.
As the columns of the matrix $X$ are linearly independent,
the matrix $X (C^\top C + \varSigma) X^\top =
X^\top C^\top C X  + X^\top \varSigma X$
is positive definite as well,
whence
$\colspan(X^\top C^\top C X  + X^\top \varSigma X) = \mathbb{R}^{n}$.

If the constraints are compatible, then the condition \eqref{eq:nemset} holds
true, whence
\begin{align*}
\mathbb{R}^n &= \colspan \bigl\langle X^\top
C^\top C X + X^\top \varSigma X \bigr\rangle
\\
& \subset \colspan \bigl\langle X^\top C^\top C X \bigr
\rangle + \colspan \bigl\langle X^\top \varSigma X \bigr\rangle
\\
&= \colspan \bigl\langle X^\top C^\top \bigr\rangle +
\colspan \bigl\langle X^\top \varSigma \bigr\rangle
\\
&= \colspan \bigl\langle X^\top \varSigma \bigr\rangle = \colspan
\bigl\langle X^\top \varSigma X \bigr\rangle.
\end{align*}
Since  $\colspan\langle X^\top \varSigma X \rangle = \mathbb{R}^n$,
the matrix $X^\top \varSigma X$ is nonsingular.

\paragraph{2. Sufficiency}
If the matrix $X^\top \varSigma X$ is nonsingular,
then
\[
\colspan \bigl\langle X^\top \varSigma \bigr\rangle = \colspan \bigl
\langle X^\top \varSigma X \bigr\rangle = \mathbb{R}^n
\supset \colspan \bigl\langle X^\top C^\top \bigr\rangle.
\]
Thus the condition \eqref{eq:nemset}, which is the necessary and sufficient
condition for compatibility of the constraints, holds true.
\end{proof}

\begin{proof}[Proof of Proposition~\ref{prop:6.4}.]
Construct simultaneous diagonalization of matrices\break
$X C C^\top X^\top$ and $X \varSigma X^\top$
(according to Theorem~\ref{thm-gedpsp})
that satisfies Remark \ref{rem:remark5.2-1}:
\[
X^\top C^\top C X = \bigl(T^{-1}
\bigr)^\top \varLambda T^{-1}, \qquad X^\top
\varSigma X = \bigl(T^{-1} \bigr)^\top \mathrm{M}
T^{-1}.
\]
Notations $\varLambda$, $\mathrm{M}$,
$T = \begin{bmatrix} T_1 & T_2 \end{bmatrix}$,
$\mu_i$, $\lambda_i$, $\nu_i$
are taken from Theorem~\ref{thm-gedpsp},
Remark~\ref{rem:6.2-1},
and Lemma~\ref{lem:psdgsvd1}.

The subspace
\[
\colspan \bigl\langle X^\top C^\top \bigr\rangle = \colspan
\bigl\langle X^\top C^\top C X \bigr\rangle = \colspan \bigl
\langle \bigl(T^{-1} \bigr)^\top \varLambda T^{-1}
\bigr\rangle = \colspan \bigl\langle \bigl(T^{-1} \bigr)^\top
\varLambda \bigr\rangle
\]
is spanned by columns of the matrix $(T^{-1})^\top$ that correspond to
nonzero $\lambda_i$'s.
Similarly, the subspace
$\colspan\langle X^\top \varSigma\rangle =
\colspan\langle (T^{-1})^\top \textrm{M}\rangle$ is spanned
by columns of the matrix $(T^{-1})^\top$ that correspond to non-zero
$\mu_i$'s.  Note that the columns of the matrix $(T^{-1})^\top$
are linearly independent.
The condition $\colspan\langle
X^\top C^\top\rangle \subset \colspan\langle X^\top \varSigma\rangle$
is satisfied if and only if
$\lambda_i \neq 0 $ for all $i$ such that  $\mu_i \neq 0$ (that is $\nu_i<\infty$,
$i=1,\ldots,d$, where notation $\nu_i = \lambda_i / \nu_i$ comes from Theorem~\ref{thm-gedpsp}).
Thus, due to Proposition~\ref{prop:forremark5.2-1},{\sloppy
\[
\bigl(X^\top \varSigma X \bigr)^{\dagger} = T^{}
\pinvb( \textrm{M}) T^{\top} .
\]}\relax

Construct the chain of equalities:
\begin{align*}
&\min_{\substack{\Delta \, (I - P_\varSigma) = 0\\
(C - \Delta) X = 0}} \lambda_{k+m-d} \bigl(\Delta \pinvb(
\varSigma) \Delta^\top \bigr)
\\
&\quad\stackrel{{\rm(a)}} {=} \lambda_{k+m-d} \bigl(C X \bigl(X^\top
\varSigma X \bigr)^{\dagger} X^\top C^\top \bigr) =
\lambda_{k+m-d} \bigl(C X \, T^{} \pinvb(\mathrm{M})
T^{\top} \, X^\top C^\top \bigr)
\\
&\quad\stackrel{{\rm(b)}} {=} \lambda_{k} \bigl( \pinvb(\textrm{M})
T^{\top} X^\top C^\top C X T^{} \bigr) =
\lambda_{k} \bigl( \pinvb(\textrm{M}) \varLambda \bigr) =
\nu_k
\\
&\quad\stackrel{{\rm(c)}} {=} \min \bigl\{ \lambda\ge 0 : \mbox{``}\exists
V_1{\subset} \mathbb{R}^d,\; \dim V_1{=}k :
\bigl(X^\top C^\top C X- \lambda X^\top \varSigma X
\bigr)|_{V_1} \le 0 \mbox{''} \bigr\}
\\
&\quad\stackrel{{\rm(d)}} {=} \min \bigl\{ \lambda\ge 0 : \mbox{``}\exists V{\subset}
\colspana<X>,\; \dim V{=}k : \bigl(C^\top C - \lambda \varSigma
\bigr)|_V \le 0 \mbox{''} \bigr\}.
\end{align*}
Equality (a) follows from \ref{prop-l588} because the matrix
$C X \pinvp(X^\top \varSigma X) X^\top C^\top$ is the least value
of the expression $\Delta \pinvb(\varSigma) \Delta^\top$ with
constraints $(I - P_\varSigma) \Delta^\top = 0$ and
$(C - \Delta) X = 0$.

Equality (b) follows from the relation between
characteristic polynomials of two products
of two rectangular matrices:
\[
\chi_{C X T \, \pinvb(\mathrm{M}) T^{\top} X^\top C^\top} (\lambda) = (-\lambda)^{m-d} \chi_{\pinvb(\mathrm{M}) T^{\top} X^\top C^\top \, C X T}
(\lambda)
\]
because $C X T$ is an $m \times d$ matrix
and $\pinvb(\mathrm{M}) T^{\top} X^\top
C^\top$ is a $d \times m$ matrix.
Thus, the matrix $C X T \, \pinvb(\mathrm{M}) T^{\top}
X^\top C^\top$ has all the eigenvalues of the matrix
$\pinvb(\mathrm{M}) T^{\top} X^\top C^\top \times C X T =
\pinvb(\mathrm{M}) \varLambda$ and, besides them,
the eigenvalue $0$ of multiplicity $m-d$.
All these eigenvalues are nonnegative.

Equality (c) holds true due to Lemma~\ref{lem:psdgsvd1}.

Since the columns of the matrix $X$ are linearly independent,
there is a one-to-one correspondence between subspaces
of $\colspana<X>$ and of $\mathbb{R}^d$:
if $V$ is a subspace of $\colspana<X>$, then there
exists a unique subspace $V_1 \subset \mathbb{R}^d$,
and for those $V$ and $V_1$,
\begin{itemize}
\item $\dim V = \dim V_1$;
\item
the restriction of the quadratic form $C^\top C - \lambda \varSigma$ to
the subspace $V$ is negative semidefinite if and only if
the restriction of the quadratic form $X^\top C^\top C X - \lambda X^\top \varSigma X$
to the subspace $V_1$ is negative semidefinite.
\end{itemize}
Hence, equality (d) holds true.

Equation \eqref{eq-prop64} is proved. As to
Remark~\ref{remark:forprop6.4}, the minimum in the left-hand side
of \eqref{eq-prop64} is attained for $\Delta = \Deltapm$. The
minimum in the right-hand side of \eqref{eq-prop64} is attained if
the subspace $V$ is a linear span of $k$ columns of the matrix
$X T$ that correspond to the $k$ least $\nu_i$'s.
\end{proof}

\begin{proof}[Proof of Proposition~\ref{prop:6.5}.]
By Lemma~\ref{lem:psdgsvd1} and Proposition~\ref{prop:6.4}, the
inequality \eqref{neq0l709} is equivalent to the obvious
inequality
\begin{align*}
&\min \bigl\{ \lambda\ge 0 : \mbox{``}\exists V{\subset} \colspana<X>,\; \dim
V{=}k : \bigl(C^\top C - \lambda \varSigma \bigr)|_V \le 0
\mbox{''} \bigr\}
\\
&\quad\ge \min \bigl\{ \lambda\ge 0 \mathrel| \text{``}\exists V, \ \dim V = k : (A -
\lambda B)|_V \le 0 \text{''} \bigr\} .
\end{align*}

From the proof it follows that if $\nu_d = \infty$,
then for any $(n+d)\times d$ matrix $X$ of rank $d$
the constraints in \eqref{TLS-fixX568} are not compatible.

Now prove that if $\nu_d < \infty$ and $X = [u_1, u_2, \ldots,
u_d]$, then the inequality in Proposition~\ref{prop:6.5} becomes
an equality.  Indeed, then the constraints in \eqref{TLS-fixX568}
are compatible because they are satisfied for
$\Delta = C T D T^{-1}$, where
\begin{align*}
D &= \diag(d_1, d_2, \ldots, d_{d+n}),
\\
d_k &= \begin{cases}
1 & \mbox{if $\mu_k >0$ and $k\le d$},\\
0 & \mbox{if $\mu_k =0$ or $k > d$}.
\end{cases}
\end{align*}
By Proposition \ref{prop-l588}
\begin{align*}
\min_{\substack{ \Delta (I - P_\varSigma) = 0 \\
(C - \Delta) X = 0}} \lambda_{k+m-d} \bigl(\Delta \pinvb(
\varSigma) \Delta^\top \bigr) &= \lambda_{k+m-d} \bigl(C X
\bigl(X^\top \varSigma X \bigr)^{\dagger} X^\top
C^\top \bigr)
\\
&= \lambda_k \bigl( \bigl(X^\top \varSigma X
\bigr)^{\dagger} X^\top C^\top C X \bigr)
\\
&= \lambda_k \bigl(\pinvbb( \mathrm{M}_d )
\varLambda_d \bigr) = \nu_k,
\end{align*}
where $\mathrm{M}_d = \diag(\mu_1,\ldots,\mu_d)$ and $\varLambda_d =
\diag(\lambda_1,\ldots,\lambda_d)$ are principal submatrices of
the matrices $\mathrm{M}$ and $\varLambda$, respectively.
\end{proof}

\begin{proof}[Proof of Proposition~\ref{prop:gs6.6}]
For every matrix $\Delta$ that satisfies the constraints\break
$(I - P_{\varSigma}) \Delta = 0$ and $\rank(C-\Delta) \le n$,
there exists an $(n+d) \times d$ matrix $X$ of rank $d$
such that $(C-\Delta) X = 0$.
Assuming that such $\Delta$ exists, we get $\nu < +\infty$
because the equalities
$\nu = +\infty$, $(I - P_{\varSigma}) \Delta = 0$,
$\rank X = d$, and $(C-\Delta) X = 0$
cannot hold simultaneously.

We have
\begin{align}
\big\| \Delta \, \bigl(\varSigma^{1/2} \bigr)^{\dagger}
\big\|_F^2 &= \trace \bigl( \Delta \pinvb(\varSigma)
\Delta^\top \bigr) = \sum_{i=1}^{m}
\lambda_{i} \bigl( \Delta \pinvb(\varSigma) \Delta^\top
\bigr)
\nonumber
\\
&= \sum_{i=1}^{m-d} \lambda_i
\bigl( \Delta \pinvb(\varSigma) \Delta^\top \bigr) + \sum
_{k=1}^{d} \lambda_{k+m-d} \bigl( \Delta
\pinvp(\varSigma) \Delta^\top \bigr)
\nonumber
\\
& \ge 0 + \sum_{k=1}^d \nu_k,
\label{eq:L1585}
\end{align}
where the inequalities hold true due to positive semidefiniteness of $\varSigma$
and due to Proposition~\ref{prop:6.5}.

If $\nu_d = \infty$, than the constraints $\Delta \, (I - P_\varSigma) = 0$
and $\rank(C-\Delta) \le n$ are not compatible.
Otherwise,
the equality in \eqref{eq:L1585} is attained for
$\Delta = \Deltaem := C X \pinvp(X^\top \varSigma X)\times X^\top \varSigma$,
where the matrix $X$ consists of first $d$ rows of the matrix $T$,
where $T$ comes from decomposition \eqref{eq:diddecomposCCS}.\vadjust{\goodbreak}

Thus, if the constraints in \eqref{eqTLS118} are compatible,
then the minimum is equal to $ ( \sum_{k=1}^d \nu_k  )^{1/2}$
and is attained at $\Deltaem$.
Otherwise, if the constraints are incompatible, then
by contraposition to the second statement of Proposition~\ref{prop:6.5}
$\nu_d = +\infty$ and
$ ( \sum_{k=1}^d \nu_k  )^{1/2} = +\infty$.

If the minimum in \eqref{eqTLS118} is attained at $\Delta$,
then the inequality \eqref{eq:L1585} becomes an equality,
whence
\begin{align}
\label{eq0:L1605} \lambda_i \bigl(\Delta \pinvb(\varSigma)
\Delta^\top \bigr) &= 0, \quad i=1,\ldots,m-d;
\\
\label{eq0:L1606} \lambda_{k+m-d} \bigl(\Delta \pinvb(\varSigma)
\Delta^\top \bigr)& = \nu_k, \quad k=1,\ldots,d;
\end{align}
in particular,
\[
\lambda_{\max} \bigl(\Delta \pinvb(\varSigma) \Delta^\top
\bigr) = \nu_d.
\]
Remember that $\nu_d$ is the minimum value in \eqref{eqTLS220}.
Thus, the minimum in \eqref{eqTLS220} is attained at $\Delta$,
although it may be also attained elsewhere.
\end{proof}

\begin{proof}[Proof of Proposition~\ref{prop:uniprop}]
\textit{1.}\quad
The monotonicity follows from results of \cite{Mirsky1960}.
The unitarily invariant norm is a symmetric
gauge function of the singular values,
and the symmetric gauge function is
monotonous in non-negative inputs
(see \cite[ineq.~(2.5)]{Mirsky1960}).

\paragraph{2}
Let $\sigma_1(M_1) < \sigma_1(M_2)$ and $\sigma_i(M_1) \le \sigma_i(M_2)$ for all $i=2,\ldots,\min(m,n)$.
Then for all  $k=1,\ldots,\min(m,n)$
\[
\sum_{i=1}^k \sigma_i(M_1)
\le \frac
{\sigma_1(M_1) + \sigma_2(M_1) + \cdots + \sigma_{\min(m,n)}(M_1)}{
\sigma_1(M_2) + \sigma_2(M_1) + \cdots + \sigma_{\min(m,n)}(M_1)} \sum_{i=1}^k
\sigma_i(M_2) .
\]
Due to Ky Fan \cite[Theorem 4]{KyFan1951} or \cite[Theorem 1]{Mirsky1960}, this implies that
\[
\uin{M_1} \le \frac
{\sigma_1(M_1) + \sigma_2(M_1) + \cdots + \sigma_{\min(m,n)}(M_1)}{
\sigma_1(M_2) + \sigma_2(M_1) + \cdots + \sigma_{\min(m,n)}(M_1)} \uin{M_2}.
\]
Since
\[
0 \le \frac
{\sigma_1(M_1) + \sigma_2(M_1) + \cdots + \sigma_{\min(m,n)}(M_1)}{
\sigma_1(M_2) + \sigma_2(M_1) + \cdots + \sigma_{\min(m,n)}(M_1)} < 1 \quad \mbox{and} \quad \uin{M_2}>0,
\]
$\uin{M_1} < \uin{M_2}$.
\end{proof}

\begin{proof}[Proof of Proposition~\ref{prop:gs6.6uin}]
Notice that the optimization problems
\eqref{eqTLS118}, \eqref{eqTLS220}, and
\eqref{eqTLS118uin} have the same constraints.
If the constraints are compatible,
then the minimum in \eqref{eqTLS118}
is attained for
$\Delta = \Deltaem := C X \pinvp(X^\top \varSigma X) X^\top \varSigma$.

\paragraph{\ref{prop:gs6.6uin:part1}}
Let $\DeltaFro$ minimize \eqref{eqTLS118},
and let $\DeltaFeas$ satisfy the constraints.
Then, by Proposition~\ref{prop:6.5} and
eqn.~\eqref{eq0:L1606},
\begin{align*}
\lambda_{k+m-d} \bigl(\DeltaFro \pinvb(\varSigma) \DeltaFro^\top
\bigr) &= \nu_k \le \lambda_{k+m-d} \bigl(\DeltaFeas \pinvb(
\varSigma) \DeltaFeas^\top \bigr), \quad k=1,\ldots,d;
\\
\sigma_{d+1-k} \bigl(\DeltaFro \bigl(\varSigma^{1/2}
\bigr)^{\dagger} \bigr) &\le \sigma_{d+1-k} \bigl(\DeltaFeas \bigl(
\varSigma^{1/2} \bigr)^{\dagger} \bigr), \\
 k&=\max(1,d{+}1{-}m),
\ldots,d;
\\
\sigma_j \bigl(\DeltaFro \bigl(\varSigma^{1/2}
\bigr)^{\dagger} \bigr) &\le \sigma_j \bigl(\DeltaFeas \bigl(
\varSigma^{1/2} \bigr)^{\dagger} \bigr), \quad j=1,\ldots,
\min(d,m);
\end{align*}
by eqn.~\eqref{eq0:L1605}
\begin{align*}
\lambda_i \bigl(\DeltaFro \pinvb(\varSigma) \DeltaFro^\top
\bigr) &= 0, \quad i=1,\ldots,m-d,
\\
\sigma_{m+1-i} \bigl(\DeltaFro \bigl(\varSigma^{1/2}
\bigr)^{\dagger} \bigr) &= 0 \le \sigma_{m+1-i} \bigl(\DeltaFeas
\bigl( \varSigma^{1/2} \bigr)^{\dagger} \bigr) , \quad i \le m-d;
\\
\sigma_j \bigl(\DeltaFro \bigl(\varSigma^{1/2}
\bigr)^{\dagger} \bigr) &= 0 \le \sigma_j \bigl(\DeltaFeas \bigl(
\varSigma^{1/2} \bigr)^{\dagger} \bigr), \quad d+1 \le j \le
\min(m,\:n+d).
\end{align*}
Thus
\begin{equation}
\sigma_j \bigl(\DeltaFro \bigl(\varSigma^{1/2}
\bigr)^{\dagger} \bigr) \le \sigma_j \bigl(\DeltaFeas \bigl(
\varSigma^{1/2} \bigr)^{\dagger} \bigr) \quad \mbox{for all $j \le
\min(m,n+d)$}, \label{neq:sigi}
\end{equation}
whence by Proposition~\ref{prop:uniprop}
$\uin{\DeltaFro \pinvp(\varSigma^{1/2})} \le
\uin{\DeltaFeas \pinvp(\varSigma^{1/2})}$.
Thus $\DeltaFro$ indeed minimizes \eqref{eqTLS118uin}.

\paragraph{\ref{prop:gs6.6uin:part2}}
Let $\DeltaUin$ minimize \eqref{eqTLS118uin},
so the constraints are compatible.
Then $\Deltaem$ minimizes both \eqref{eqTLS118}
and \eqref{eqTLS220}, see Proposition~\ref{prop:gs6.6}.
Thus,
\[
\big\|\DeltaUin \bigl(\varSigma^{1/2} \bigr)^{\dagger}\big\|_{\rm U} \le
\big\|\Deltaem \bigl(\varSigma^{1/2} \bigr)^{\dagger}\big\|_{\rm U},
\]
and by \eqref{neq:sigi}
\[
\sigma_j \bigl(\Deltaem \bigl(\varSigma^{1/2}
\bigr)^{\dagger} \bigr) \le \sigma_j \bigl(\DeltaUin \bigl(
\varSigma^{1/2} \bigr)^{\dagger} \bigr) \quad \mbox{for all $j \le
\min(m,n+d)$}.
\]
Then by Proposition~\ref{prop:uniprop}
(contraposition to part \ref{prop:uniprop:part2})
\begin{align*}
\sigma_1 \bigl(\Deltaem \bigl(\varSigma^{1/2}
\bigr)^{\dagger} \bigr) &= \sigma_1 \bigl(\DeltaUin \bigl(
\varSigma^{1/2} \bigr)^{\dagger} \bigr),
\\
\min_{\substack{\Delta (I-P_\varSigma) = 0\\
\rank(C-\Delta) \le n}} \bigl(\Delta \pinvb(\varSigma)
\Delta^\top \bigr) &= \lambda_{\max} \bigl(\Deltaem \pinvb(
\varSigma)\Deltaem^\top \bigr) = \lambda_{\max} \bigl(\DeltaUin
\pinvb(\varSigma)\DeltaUin^\top \bigr).
\end{align*}
Thus $\DeltaUin$ indeed minimizes \eqref{eqTLS220}.
\end{proof}

\begin{proof}[Proof of Proposition \ref{prop:gep6.7}.]
We can assume that $\mu_i \in \{0, 1\}$ in
\eqref{eq:diddecomposCCS}.

The set of matrices $\Delta$ that satisfy \eqref{eqTLSX124}
depends only on
$\colspan\langle \widehat X_{\rm ext} \rangle$
and does not change after linear transformations of columns of
$\widehat X_{\rm ext}$.

By linear transformations of the columns, the matrix
$T^{-1} \widehat X_{\rm ext}$
can be transformed to the reduced column echelon form.
Thus, there exists such an $(n+d) \times d$
matrix $T_5$ in the column echelon form that
\[
\colspan \langle \widehat X_{\rm ext} \rangle = \colspan \langle T
T_5 \rangle.
\]
Notice that $\rank T_5 = \rank \widehat X_{\rm ext} = d$.

Denote by $d_*$ and $d^*$ the first and the last of the indices $i$
such that $\nu_i = \nu_d$. Then
\begin{align*}
\nu_{d_*-1} &< \mbox{}\nu_{d_*} \quad \mbox{if $d_* \ge 2$;}
\\
\nu_{d_*} &= \cdots = \nu_d = \cdots =
\nu_{d^*};
\\
\nu_{d^*} &< \nu_{d^*+1} \quad \mbox{if $d^* < n+d$.}
\end{align*}

\paragraph{Necessity}
Let $\Delta$ be a point where the constrained minimum in \eqref{eqTLS118} is attained.
Then equalities \eqref{eq0:L1605}--\eqref{eq0:L1606} from the proof of Proposition~\ref{prop:gs6.6}
hold true. Thus, due to Propositions \ref{prop:6.4} and \ref{prop:6.5},
for all  $k=1,\ldots, d$
\[
\min \bigl\{ \lambda\ge 0 : \mbox{``}\exists V{\subset} \colspana<\widehat
X_{\rm ext}>,\; \dim V{=}k : \bigl(C^\top C - \lambda \varSigma
\bigr)|_V \le 0 \mbox{''} \bigr\} =
\nu_k.
\]

According to \ref{remark:forprop6.4}, we can construct a stack of subspaces
\[
V_1 \subset V_2 \subset \cdots \subset V_d
= \colspan\langle \widehat X_{\rm ext} \rangle,
\]
such that $\dim V_k = k$ and
the restriction of the quadratic form $C^\top C - \nu_k \varSigma$ to
the subspace $V_k$ is negative semidefinite, for all $k \le d$.

Now, prove that
\begin{equation}
\label{neqc:L1675} \colspana<{u_i : \nu_i<
\nu_d}> \subset \colspana<\widehat X_{\rm ext}>.
\end{equation}
Suppose the contrary: $\colspana<{u_i : \nu_i<\nu_d}>  \not\subset
\colspana<\widehat X_{\rm ext}>$.
Then there exists $i < d_* $ such that
$u_i \notin \colspana<\widehat X_{\rm ext}>$,
and, as a consequence,
$u_i \notin V_{\max \{j \;:\; \nu_j \le \nu_i\}}$.
Find the least $k$ such that
$u_k \notin V_{\max \{j \;:\; \nu_j \le \nu_k\}}$.
Let $k_*$ and $k^*$ denote the first and the last indices $i$ such that
$\nu_i = \nu_k$.
Then $1 \le k_* \le k \le k^* < d_* \le d \le d^*$
and $u_k \notin V_{k^*}$.

Since $\colspana<u_1, \ldots, u_{k_*-1}> \subset V_{k_*-1}
\subset V_{k^*}$,
\begin{align*}
\dim \bigl(V_{k^*} \cap \colspana<u_{k_*}, \ldots,
u_{n+d}>\bigr) &= \dim \bigl(V_{k^*} / \colspana<u_{1},
\ldots, u_{k_*-1}>\bigr)
\\
&= \dim V_{k^*} - (k_*-1) = k^* - k_* + 1.
\end{align*}
Since $u_k \notin V_{k^*}$,
$u_k \notin V_{k^*} \cap \colspana<u_{k_*}, \ldots, u_{n+d}>$,
\[
\dim\operatorname{span}\big\langle V_{k^*} \cap \colspana<u_{k_*}, \ldots
u_{n+d}>, \: u_k\big\rangle = k^* - k_* + 2.
\]

Now, consider the $(n+d-k_*+1)\times(n+d-k_*+1)$ diagonal matrix
\begin{align*}
D(\lambda) &:= [u_{k_*},\ldots,u_{n+d}]^\top
\bigl(C^\top C - \lambda \varSigma \bigr) [u_{k_*},
\ldots,u_{n+d}]
\\
&= \diag(\lambda_j - \lambda \mu_j, \; j=k_*,
\ldots,n{+}d)
\end{align*}
for various $\lambda$.
For $\lambda = \nu_k = \nu_{k_*}$, the \querymark{Q12}inequality
$\lambda_j - \nu_k \mu_j \ge 0$ holds true for all $j \ge k_*$,
so the matrix $D(\nu_k)$ is positive semidefinite.
For $\lambda = \nu_{k^*+1}$, the inequality $\lambda_j - \nu_{k^*+1} \mu_j \le 0$
holds true for all $k_* \le j \le k^*+1$, so there exists a
$k^* - k_* + 2$-dimensional
subspace of $\mathbb{R}^{n+d-k_*+1}$ where the quadratic form $D(\nu_{k^*+1})$
is negative semidefinite.
For $\lambda < \nu_{k^*+1}$, the inequality $\lambda_j - \lambda \mu_j > 0$
holds true for all $k^*+1 \le j \le n+d$.  Therefore, there exists an  $n+d-k^*$-dimensional
subspace of $\mathbb{R}^{n+d-k_*+1}$ where the quadratic form $D(\lambda)$
is positive definite.
According to the proof of Sylvester's law of inertia, there is no subspace
of dimension $k^*  - k_* + 2 = (n+d-k_*+1) - (n+d-k^*) + 1$
where the quadratic form $D(\lambda)$ is negative semidefinite.
Thus, $\nu_{k^*+1}$ is the least number such that there exists a
$k^*  - k_* + 2$-dimensional subspace where the quadratic form $D(\lambda)$
is negative semidefinite.

Similarly to the chain of equalities in the proof of Proposition~\ref{prop:6.4},
\begin{align}
\label{eq:1707} \nu_{k^*+1} & = \min \bigl\{\lambda \ge 0 : \mbox{``}
\exists V_1, \; \dim V_1 = k^* - k_* + 2\; : \; D(
\lambda) \vert _{V_1} \le 0 \mbox{''} \bigr\}
\nonumber
\\
&= \min \bigl\{\lambda \ge 0 : \mbox{``}\exists V_1, \; \dim
V_1 = k^* - k_* + 2\; :
\nonumber
\\
& \qquad [u_{k_*},\ldots,u_{n+d}]^\top
\bigl(C^\top C - \lambda \varSigma \bigr) [u_{k_*},
\ldots,u_{n+d}] \vert _{V_1} \le 0 \mbox{''}
\bigr\}
\nonumber
\\
&= \min \bigl\{\lambda \ge 0 : \mbox{``}\exists V_1, \; V \subset
\colspana<u_{k_*},\ldots,u_{n+d}>, \; \dim V = k^* - k_* + 2\;
:
\nonumber
\\
& \qquad \bigl(C^\top C - \lambda \varSigma \bigr) \vert
_{V} \le 0 \mbox{''} \bigr\}
\end{align}

The restriction of the quadratic form $C^\top C - \nu_k \varSigma$ to
the subspace  $\colspana<u_{k_*},\allowbreak \ldots,\allowbreak u_{n+d}>$ is positive semidefinite
because $[u_{k_*}, \ldots, u_{n+d}]^\top  (C^\top C - \nu_k \varSigma) \times \allowbreak
[u_{k_*}, \ldots,\allowbreak u_{n+d}] = D(\nu_k)$
is a positive semidefinite diagonal matrix.
Then
\begin{align}
\label{eq:L1722} &\bigl\{v \in \colspana<u_{k_*}, \ldots
u_{n+d}> : v^\top \bigl(C^\top C - \nu_k
\varSigma \bigr) v \le 0 \bigr\}\notag
\\
&\quad= \bigl\{v \in \colspana<u_{k_*}, \ldots, u_{n+d}> :
\bigl(C^\top C - \nu_k \varSigma \bigr) v = 0 \bigr\}
\end{align}
is a linear subspace.
Since this subspace contains the subspace
$V_k \cap \allowbreak \colspana<u_{k_*},\allowbreak \ldots, \allowbreak u_{n+d}>$
(as the quadratic form $C^\top C - \nu_k \varSigma$
is negative semidefinite on $V_k$)
and the vector $u_k$
(as $u_k \in \colspana<u_{k_*},\allowbreak
\ldots,\allowbreak u_{n+d}>$
and $u_k^\top (C^\top C - \nu_k \varSigma) u_k^{}
= \lambda_k - \nu_k \mu_k = 0$),
it contains $\colspan\langle V_{k^*} \cap\allowbreak  \colspana<u_{k_*},\allowbreak
\ldots,\allowbreak u_{n+d}>, \: u_k\rangle$. But, as $\nu_k < \nu_{k^*+1}$,
this contradicts \eqref{eq:1707}.

Now, prove that
\begin{equation}
\label{neqc:L1740} \colspana<\widehat X_{\rm ext}> \subset
\colspana<{u_i : \nu_i \le \nu_d}>.
\end{equation}

Due to \eqref{neqc:L1675},
\[
\colspana<\widehat X_{\rm ext}> = \operatorname{span}\big\langle \colspana<\widehat
X_{\rm ext}> \cap \colspana<u_{d_*},\ldots,u_{n+d}>, \:
u_1, \ldots, u_{d_*-1} \big\rangle.
\]

Hence, to prove \eqref{neqc:L1740}, it is enough to
show that
\begin{equation}
\label{neqc:L1763} \colspana<\widehat X_{\rm ext}> \cap \colspana<u_{d_*}
, \ldots, \nu_{n+d}> \subset \colspana<u_{d_*} , \ldots,
\nu_{d^*}>.
\end{equation}

The restriction of the quadratic form $C^\top C - \nu_d \varSigma$
to the subspace $\colspana<u_{d_*},\allowbreak\ldots,u_{n+d}>$ is positive
semidefinite. Hence
\begin{align}
\label{eq:subspace3554} &\bigl\{v \in \colspana<u_{d_*}, \ldots
u_{n+d}> : v^\top \bigl(C^\top C - \nu_d
\varSigma \bigr) v \le 0 \bigr\}\notag
\\
&\quad= \bigl\{v \in \colspana<u_{d_*}, \ldots u_{n+d}> :
v^\top \bigl(C^\top C - \nu_d \varSigma \bigr) v
= 0 \bigr\}
\end{align}
is a linear subspace (see
equation \eqref{eq:L1722}). This subspace contains
the subspaces
$\colspana<\widehat X_{\rm ext}> \cap \colspana<u_{d_*} , \ldots,
\nu_{n+d}>$ and $\colspana<u_{d_*} , \ldots,  \nu_{d^*}>$.
Denote the dimension of the subspace \eqref{eq:subspace3554}:
\sloppy{
\[
d_2 = \dim \bigl\{v \in \colspana<u_{d_*}, \ldots
u_{n+d}> : v^\top \bigl(C^\top C - \nu_d
\varSigma \bigr) v = 0 \bigr\}.
\]}\relax
If \eqref{neqc:L1763} does not hold,
then $d_2 > d^* - d_* + 1$;
$d_2 \ge d^* - d_* + 2$.
Then
\[
\exists V \subset\colspana<u_{d_*}, \ldots u_{n+d}> ,\; \dim
V = d_2 \; : \; \bigl(C^\top C - \nu_d
\varSigma \bigr)|_V \le 0
\]
(as an instance of such a subspace $V$, we can take the one defined in \eqref{eq:subspace3554}).
Then, taking a $d^* - d_* + 2$-dimensional subspace of $V$,
we get
\[
\exists V \subset\colspana<u_{d_*}, \ldots u_{n+d}> ,\; \dim
V = d^* - d_* + 2 \; : \; \bigl(C^\top C - \nu_d \varSigma
\bigr)|_V \le 0.
\]
Due to \eqref{eq:1707} (for $k=d$),
$\nu_{d^*+1} \le \nu_d$,
which does not hold true.

Assuming the contrary to \eqref{neqc:L1763},
we got a contradiction.
Hence, \eqref{neqc:L1763} and \eqref{neqc:L1740} hold true.

\paragraph{Sufficiency}
Remember that $T = [u_1,\ldots,u_{n+d}]$ is an
$(n+d)\times(n+d)$ matrix of generalized
eigenvectors of the matrix pencil
$\langle C^\top C,\, \varSigma \rangle$,
and respective generalized eigenvalues are arranged in ascending order.
By means of linear operations of the columns,
the matrix $T^{-1} \hxx$ can be transformed into
the reduced column echelon form.
In other words, there exists such an $n\times n$ nonsingular matrix $T_8$,
that the  $(n+d)\times n$ matrix
\begin{equation}
T_5 = T^{-1} \hxx T_8 \label{eq:T5}
\end{equation}
is in the reduced column echelon form.
The equality \eqref{eq:T5} implies that
\begin{equation}
\label{eq:hxxT1T5} \colspana<\hxx> = \colspana<T T_5>.
\end{equation}
If condition \eqref{neq0l709} holds, then in
representation~\eqref{eq:hxxT1T5}
the matrix $T_5$ has the following block structure
\[\begin{htpicture}
T_5 = \begin{array}{|c|c|}
\hline
I_{d^* - 1}                  & 0_{(d^*-1) \times (d - d^* + 1)} \\
\hline
0_{(d^*-d_*+1)\times(d^*-1)} & T_{61}          \\
\hline
\multicolumn{2}{|c|}{0_{(n - d^*) \times d}}
\\ \hline
\end{array}\,,\end{htpicture}
\]
where $T_{61}$ is a $(d^*-d_*+1) \times (d-d_*+1)$ reduced column echelon matrix.
(Any of the blocks except $T_{61}$ may be an ``empty matrix''.)

Since the columns of $T_5$ are linearly independent, the columns of
$T_{61}$ are linearly independent as well.
Hence the matrix $T_{61}$ may
be appended with columns such that the resulting matrix $T_6 =
[T_{61}, T_{62}]$ is nonsingular. Perform the Gram--Schmidt
orthogonalization of columns of the matrix $T_6$ by constructing such
an upper-triangular matrix
\[\begin{htpicture}
T_7 = \begin{pmatrix} T_{71} & T_{72} \\ 0 & T_{74}
\end{pmatrix} = \begin{array}{|c|c|}\hline T_{71} & T_{72} \\ \hline
 0_{(d^*-d)\times(d-d_*+1)} & T_{74}
\\ \hline
\end{array}\end{htpicture}
\]
that $T_7^\top T_6^\top T_6^{} T_7^{} = I_{d^* - d_* + 1}$.

Change the basis in the simultaneous diagonalization
of the matrices $C^\top C$ and $\varSigma$.
Denote
\[
T_{\rm new} = \bigl[u_1, \ldots u_{d_*-1},
[u_{d_*}, \ldots u_{d^*}] T_6 T_7,
u_{d^*+1}, \ldots u_{n+d} \bigr].
\]
If $\nu_d > 0$, the equation \eqref{eq:diddecomposCCS}
with $T_{\rm new}$ substituted for $T$ holds true, since
\[
T_{\rm new}^\top C^\top C T_{\rm new}^{}
= \varLambda,\qquad T_{\rm new}^\top \varSigma
T_{\rm new}^{} = \mathrm{M}.
\]
(Here we use that
$\lambda_{d_*} = \cdots = \lambda_{d^*}$,
$\mu_{d_*} = \cdots = \mu_{d^*}$.
If $\nu_d = 0$, then the latter equation may or may not hold true.)
The subspace
\begin{align*}
\colspana<\widehat X_{\rm ext}> = \colspana<T T_5> &= \colspan
\bigl\langle u_1, \ldots u_{d_*-1}, [u_{d_*},
\ldots u_{d^*}] T_{61} \bigr\rangle
\\
&= \colspan \bigl\langle u_1, \ldots u_{d_*-1},
[u_{d_*}, \ldots u_{d^*}] T_{61} T_{71}
\bigr\rangle
\end{align*}
is spanned by the first $d$ columns of the matrix $T_{\rm new}$.

It can be easily verified that
$\colspana<\hxx^\top C^\top> = \colspana<T_8^\top T_5^\top \varLambda>$
and
$\colspana<\hxx^\top \varSigma> = \colspana<T_8^\top T_5^\top \mathrm{M}>$.
The condition
$\colspana<\hxx^\top C^\top> \subset \colspana<\hxx^\top \varSigma>$
holds true if (and only if) $\nu_d < \infty$.
Thus, due to Proposition~\ref{prop-l588},
if the condition $\nu_d < \infty$ holds true,
then the constraints
  $\Delta \, (I - P_\varSigma) = 0$ and
  $(C - \Delta) \widehat{X}_{\rm ext}=0$
are compatible.{\sloppy\par}

Let $\Deltapm$ be a common point of minimum in
\begin{equation*}
\lambda_{k+m-d} \bigl(\Deltapm \pinvb(\varSigma) \Deltapm^\top
\bigr) = \min_{\substack{\Delta (I - P_\varSigma) = 0\\
(C - \Delta) \hxx = 0}} \lambda_{k+m-d} \bigl(\Delta
\pinvb(\varSigma) \Delta^\top \bigr)
\end{equation*}
for all $k=1,\ldots,d$, such that
$
\Deltapm \, (I - P_\varSigma) = 0 $ and $
(C - \Deltapm) \hxx = 0$;
such $\Deltapm$ exists due to Remark~\ref{remark:forprop6.4}.
By Proposition~\ref{prop:6.5},
\[
\lambda_{k+m-d} \bigl(\Deltapm \pinvb(\varSigma) \Deltapm^\top
\bigr) = \nu_k, \quad k=1,\ldots,d,
\]
and, from the proof of Preposition~\ref{prop:gs6.6},
\[
\lambda_i \bigl(\Deltapm \pinvb(\varSigma) \Deltapm^\top
\bigr) = 0, \quad i=1,\ldots,m-d.
\]
The minimum in \eqref{eqTLS118} is attained at $\Delta=\Deltapm$.

The case $\nu_d = 0$ is trivial: then \eqref{neq0l709}
imply that $C \widehat{X}_{\rm ext} = 0$.
Then $\Delta = 0$ satisfies the constraints $\Delta \, (I - P_\varSigma) = 0$ and
$(C - \Delta) \widehat{X}_{\rm ext}=0$
and minimizes the criterion function in \eqref{eqTLS118}.
\end{proof}

\begin{proof}[Proof of Proposition \ref{prop-6eximl}.]
Remember that if $\nu_d < \infty$, then the constraints
in \eqref{eqTLS220} are compatible, and the minimum
is attained and is equal to $\nu_d$;
see Proposition~\ref{prop:6.5}.
Otherwise, if $\nu_d = \infty$, then the constraints
in \eqref{eqTLS220} are incompatible.

Transform the expression for the functional \eqref{eqf-l816}:
\begin{align}
Q_1(X) &:= \lambda_{\max} \bigl( \bigl(X^\top
\varSigma X \bigr)^{-1} X^\top C^\top C X \bigr)
\nonumber
\\
&= \lambda_{\max} \bigl(C X \bigl(X^\top \varSigma X
\bigr)^{-1} X^\top C^\top \bigr)
\nonumber
\\
&= \min_{\Delta_1 \in \mathbb{R}^{m \times (n+d)}\, :\,
\Delta_1 (I - P_\varSigma) = 0, \,
(C - \Delta_1) X = 0} \lambda_{\max} \bigl(
\Delta_1^{} \pinvb(\varSigma) \Delta_1^\top
\bigr). \label{eq:Q1_l3786}
\end{align}
Here we used the rule how eigenvalues of the matrix product change
when the matrices are swapped, and we also used
Propositions \ref{prop-l588} and \ref{prop:simpifdef}.
By Proposition~\ref{prop:6.5}, $Q_1(X) \ge \nu_d$.

If the minimum in \eqref{eqTLS220} \& \eqref{eqTLSX124} is attained
(say at some point $(\Delta, \widehat X_{\rm ext})$),
then the constraints in the right-hand side of \eqref{eq:Q1_l3786}
are compatible for $X = \widehat X_{\rm ext}$
(particularly,
$\Delta$ is a matrix that satisfies the constraints).
Then by Proposition~\ref{prop:simpifdef}
the matrix $\widehat{X}_{\rm ext}^\top \varSigma \widehat{X}_{\rm ext}^{}$
is nonsingular.
Thus, for $X=\widehat{X}_{\rm ext}$,
minimum in the right-hand of \eqref{eq:Q1_l3786} is attained at $\Delta_1 = \Delta$
(because $\Delta$ satisfies stronger constraints of \eqref{eq:Q1_l3786}
and brings a minimum to the same functional
with weaker constraints of \eqref{eqTLS220}).

Hence,
\begin{align*}
Q_1(\widehat{X}_{\rm ext}) &:= \min_{\Delta_1 \in \mathbb{R}^{m \times (n+d)}\, :\,
\Delta_1 (I - P_\varSigma) = 0 , \,
(C - \Delta_1) \widehat{X}_{\rm ext} = 0}
\lambda_{\max} \bigl(\Delta_1^{} \pinvb(
\varSigma) \Delta_1^\top \bigr)
\\
&= \bigl(\Delta \pinvb(\varSigma) \Delta^\top \bigr) =
\nu_d,
\end{align*}
which is the minimum value of $Q_1$.

Transform the expression for the functional \eqref{eqf-l816mm}:
\begin{align*}
&\lambda_{\max} \bigl( \bigl(X^\top \varSigma X
\bigr)^{-1} X^\top \bigl(C^\top C - m \varSigma
\bigr) X \bigr)
\\
&\quad= \lambda_{\max} \bigl( \bigl(X^\top \varSigma X
\bigr)^{-1} X^\top \bigl(C^\top C \bigr) X - m
I_{n+d} \bigr) = Q_1(X) - m .
\end{align*}
Hence, the functionals \eqref{eqf-l816} and \eqref{eqf-l816mm} attain their minimal values
at the same points.
\end{proof}

\section{Conclusion}
The linear errors-in-variables model is considered.
The errors are assumed to have the same covariance matrix for each observation
and to be independent between different observations,
however some variables may be observed without errors.
Detailed proofs of the consistency theorems for the
TLS estimator, which were first stated in \cite{Shklyar2011}, are presented.

It is proved that that
the final estimator $\widehat X$ for explicit-notation
regression coefficients
(i.e., for $X_0$ in \eqref{eq:linregww} or \eqref{eq:A0X0B0},
and not the estimator $\hxx$ for $\xxtrue$ in
equation~\eqref{eq:C0X0ext0},
which sets the relationship between the regressors and response variables
\textit{implicitly\/})
is unique,
either with high probability or eventually.
This means that
in the classification used in \cite{Hnetynkova2011},
the TLS problem is of 1st class set $\mathcal{F}_1$
(the solution is unique and ``generic''),
with high probability or eventually.

As by-product, we get that if in the definition of the
estimator the Frobenius norm is replaced by the spectral norm,
then the consistency theorems still hold true.
The disadvantage of using spectral norm is that
the estimator $\widehat X$ is not unique then.
(The set of solutions to the minimal spectral norm problem
contains the set of solutions to the TLS problem.
On the other hand, it is possible that the
minimal spectral norm problem has solutions,
but the TLS problem has not -- this is
the TLS problem of 1st class set $\mathcal{F}_3$;
the probability of this random event tends to 0.)

Results can be generalized to any unitary invariant matrix norm. I do not know whether they hold true
for non-invariant norms such as the maximum absolute entry,
which is studied in \cite{Hladic2017}.

\end{document}